\documentclass{article}
\usepackage{amsmath, amssymb, amsthm}
\usepackage[all]{xy}
\usepackage[margin=0.9in, dvips]{geometry}
\usepackage{graphicx}
\usepackage{mathrsfs}
\usepackage{enumerate}
\usepackage{hyperref}
\usepackage{color}
\usepackage{bm}
\usepackage{tikz}
\usepackage{pgfplots}
\usepackage{enumitem}
\usepackage{xcolor}
\usetikzlibrary{shapes.misc}
\tikzset{
    cross/.pic = {
    \draw[rotate = 45] (-#1,0) -- (#1,0);
    \draw[rotate = 45] (0,-#1) -- (0, #1);
    }
}

\pgfplotsset{my style/.append style={axis x line=middle, axis y line=
middle, xlabel={$x$}, ylabel={$y$}, axis equal }}

\theoremstyle{definition}

\newtheorem{theorem}{Theorem}[section]
\newtheorem{lemma}[theorem]{Lemma}
\newtheorem{corollary}[theorem]{Corollary}
\newtheorem{proposition}[theorem]{Proposition}

\newtheorem*{theoremA}{Theorem A}
\newtheorem*{theoremB}{Theorem B}
\newtheorem*{theoremC}{Theorem C}
\newtheorem*{theoremD}{Theorem D}

\newtheorem*{definition}{Definition}
\newtheorem*{remark}{Remark}

\newcommand{\bb}{\mathbb}
\newcommand{\mk}{\mathfrak}
\newcommand{\ml}{\mathcal}

\newcommand{\la}{\langle}
\newcommand{\ra}{\rangle}

\newcommand{\ld}{\lambda}
\newcommand{\pl}{\partial}
\newcommand{\Hom}{\text{Hom}}

\newcommand{\ds}{\displaystyle}

\newcommand{\ep}{\varepsilon}

\newcommand{\Z}{\mathbb{Z}}
\newcommand{\T}{\mathbb{T}}
\newcommand{\R}{\mathbb{R}}
\newcommand{\C}{\mathbb{C}}

\newcommand{\N}{\mathbb{N}}

\renewcommand{\P}{\mathbb{P}}

\renewcommand{\Im}{\mathrm{Im}}
\newcommand{\Ld}{\Lambda}

\newcommand{\tr}{\mathrm{tr}}

\newcommand{\g}{\gamma}
\newcommand{\G}{\Gamma}

\newcommand{\og}{\omega}

\renewcommand{\O}{\Omega}

\renewcommand{\mod}{\;\mathrm{mod}\;}
\newcommand{\quaddd}{\quad \quad \quad \quad}
\newcommand{\quadd}{\quad \quad}
\newcommand{\vol}{\mathrm{vol}}
\newcommand{\Diff}{\mathrm{Diff}}
\newcommand{\Aut}{\mathrm{Aut}}
\newcommand{\Out}{\mathrm{Out}}
\newcommand{\nt}{\noindent}

\newcommand{\onu}{\overline{\nu}}

\newcommand{\uog}{\underline{\omega}}
\newcommand{\loc}{\mathrm{loc}}

\title{Stationary measures and orbit closures of uniformly expanding random dynamical systems on surfaces}
\author{Ping Ngai (Brian) Chung}
\date{}

\begin{document}
\maketitle
\setcounter{section}{0}

\begin{abstract}
We study the problem of classifying stationary measures and orbit closures for non-abelian action on a surface with a given smooth invariant measure. Using a result of Brown and Rodriguez Hertz, we show that under a certain finite verifiable average growth condition, the only nonatomic stationary measure is the given smooth invariant measure, and every orbit closure is either finite or dense. Moreover, every point with infinite orbit equidistributes on the surface with respect to the smooth invariant measure. This is analogous to the results of Benoist-Quint and Eskin-Lindenstrauss in the homogeneous setting, and the result of Eskin-Mirzakhani in the setting of moduli spaces of translation surfaces. We then apply this result to two concrete settings, namely discrete perturbation of the standard map and Out($F_2$)-action on a certain character variety. We verify the growth condition analytically in the former setting, and verify numerically in the latter setting.
\end{abstract}

\section{Introduction}

Given a  Riemannian manifold $M$ and an acting semigroup $\G$, the closure of the $\G$-orbit of some points of $M$ may exhibit fractal-like structure. For instance in the case when $M$ is a compact manifold and $\G$ is generated by a single Anosov diffeomorphism, there are orbit closures with fractional Hausdorff dimension. A one-dimensional example is the action of $\N$ on the circle $S^1 = \R / \Z$ generated by 
$$ x \mapsto 3x \mod 1. $$
By the Birkhoff ergodic theorem, we know that (Lebesgue-)almost every point on the circle has dense orbit. Nonetheless the orbit of every rational number is clearly finite, and one can get orbit closures that are neither finite nor the whole circle, for instance the standard Cantor middle third set.

It turns out that if one consider instead the action of a larger group, the situation becomes more rigid. Furstenberg \cite{F} showed that the orbits of the action of $\N^2$ generated by 
$$ x \mapsto 2x \mod 1 \quaddd \text{ and } \quaddd x \mapsto 3x \mod 1, $$
are either finite or dense. Moreover, he famously asked whether all the ergodic invariant Borel probability measures are either finitely supported or the Lebesgue measure on $S^1$. Major progress on this conjecture was made by Rudolph \cite{R}, who showed that the ergodic invariant measures either have zero-entropy for the action of every one-parameter subgroup or is the Lebesgue measure on $S^1$.

In two or higher dimensions, similar phenomena have been observed. For example, the action of $\Z$ on the $2$-torus $\bb{T}^2 = \R^2 / \Z^2$ generated by the matrix
$$ \begin{pmatrix} 2 & 1 \\ 1 & 1 \end{pmatrix} $$
has orbits that are neither finite nor dense. In fact using the theory of Markov partitions \cite{Bo}, one can conjugate this system to a subshift of finite type to obtain orbit closures of any Hausdorff dimension between $0$ and $2$. If one consider instead the nonabelian action on $\bb{T}^2$ generated by, say,
$$ \begin{pmatrix} 2 & 1 \\ 1 & 1 \end{pmatrix} \quad \text{ and } \quad \begin{pmatrix} 1 & 1 \\ 1 & 2 \end{pmatrix}, $$
then it follows from a result of Bourgain, Furman, Lindenstrauss and Mozes \cite{BFLM} that the orbits are either finite or dense. In fact Benoist and Quint has proved in a series of papers \cite{BQ1, BQ2, BQ3} a number of such orbit closure classifications and the corresponding measure rigidity results. A special case of their result is the following: let $\mu$ be a finitely supported measure on $\mathrm{SL}(n, \Z)$ and let $\G_\mu \subset \mathrm{SL}(n, \Z)$ be the closed subgroup generated by the support of $\mu$. If $\G_\mu$ is ``large enough'', in this case this means that every finite-index subgroup of $\G_\mu$ acts irreducibly on $\R^n$, then every ergodic $\mu$-stationary probability measure on $\bb{T}^n$ is either finitely supported or the Haar measure on $\bb{T}^n$. In particular every $\mu$-stationary probability measure is $\mathrm{SL}(n, \Z)$-invariant. They used this measure rigidity result to show that every orbit closure is either finite or dense, by first showing a stronger equidistribution result. 

The results of Benoist and Quint are in the setting of homogeneous dynamics, where one considers the natural action of a Lie group $G$ acting on a homogeneous space $G/\Lambda$. In \cite{BQ1}, it was proved that if $\mu$ is a compactly supported measure on a simple real Lie group $G$, and the subgroup $\G \subset G$ generated by the support of $\mu$ is Zariski dense in $G$, then every $\G$-orbit is either finite or dense. Moreover, the corresponding $\mu$-stationary probability measures are either finitely supported or the Haar measure on $G/\Lambda$, hence in particular are $\G$-invariant. The result was extended to a general real Lie group $G$ in \cite{BQ2}, where they showed that assuming the Zariski closure of $\G$ is semisimple, Zariski connected with no compact factor, any $\mu$-stationary measure is homogeneous. This result was further generalized by Eskin-Lindenstrauss \cite{EL} where they relaxed the assumption on $\G$ to the ``uniform expansion'' assumption to include many cases where the Zariski closure of $\G$ is not semisimple. In contrast with the case of abelian actions (for instance Rudolph's theorem mentioned above), the measure classification has no entropy assumption, and the orbit closure classification follows as a corollary of the measure rigidity theorem. 

In this paper, we study the question of measure rigidity and orbit closure classification in the setting of smooth dynamics, i.e. the action of a subgroup of diffeomorphisms on a manifold $M$. In particular, we shall prove positivity of Lyapunov exponent, measure rigidity and orbit closure classification theorems in the following two settings.
\begin{itemize}
\item Discrete random perturbation of the standard map.
\item Outer automorphism group action on the character variety $\Hom(F_2, \mathrm{SU(2)}) // \mathrm{SU(2)}$. 
\end{itemize}
The first setting was studied by Blumenthal, Xue and Young \cite{BXY}, where they considered a ``continuous'' random perturbation of the standard map and obtained positivity of Lyapunov exponent, even though positivity of exponent for the standard map is notoriously hard. Their method, however, does not apply to discrete perturbations that we consider in this paper, as it is no longer clear that any stationary measure is absolutely continuous with respect to Lebesgue. This will be explained in Section \ref{smap}. 

The second setting was studied by Goldman \cite{G}, which is based on his earlier work \cite{G2}. In \cite{G}, the ergodic decomposition of the $\Out(F_2)$-action on the character variety $\Hom(F_2, \mathrm{SU(2)})/\mathrm{SU(2)}$ is given. The topological dynamics was studied by Previte and Xia \cite{PX}, who proved that on each ergodic component, every $\Out(F_2)$-orbit is either finite or dense. Their method uses crucially the fact that $\Out(F_2)$ is generated by Dehn twists. In this paper, we shall prove that for some finite set of generators $\ml{S}$ of $\G := \Out(F_2)$ that does not contain any nontrivial powers of Dehn twist, every $\G$-orbit on each ergodic component is either finite or dense. This will be explained in Section \ref{cvariety}.

Both results are part of a more general theorem concerning the volume-preserving action of a group $\G \subset \Diff^2(M)$ on a closed surface $M$. The measure rigidity problem in this setting was studied by Brown and Rodriguez-Hertz \cite{BR}. Based on the ``exponential drift'' technique first introduced in \cite{BQ1} and some ideas in \cite{EM}, they proved that in this setting, if ``the stable distribution is not nonrandom'' (see Section \ref{mrigidity} for the precise definition), then the stationary measures are either finitely supported, or the restriction of the volume on a positive volume subset. In this paper, we will build on the work of \cite{BR} to give a more verifiable (but stronger) criterion on the acting group $\G$ so that the stationary measures and orbit closures can be classified. Such a criterion should, on one hand, be strict enough to rule out the case of a one-parameter acting group (in which case we can see from above that there can be measures of arbitrary Hausdorff dimension in general), and, on the other hand, be flexible enough to include many larger group $\G$. We will then verify this criterion in both of the aforementioned settings. 

Our measure rigidity result relies heavily on the result of Brown and Rodriguez-Hertz \cite{BR}, hence only works in the two-dimensional case. The assumption we introduce will be stronger than that of \cite{BR}, in order to give us the proof of the orbit closure classification. Nonetheless, such an assumption is a finite criterion and hence can be checked, at least in principle, in concrete settings. 

\subsection{Main results}

In this paper, we shall prove positivity of Lyapunov exponent, measure rigidity and orbit closure classification in the following two settings. 

\begin{enumerate}
\item Discrete random perturbation of the standard map
\begin{theoremA}
Let $\bb{T}^2 := \R^2 / (2\pi \Z)^2$ be the $2$-torus. For $L > 0$, $\ep > 0$ and positive integer $r$, let
\begin{itemize}
\item $F_L: \bb{T}^2 \to \bb{T}^2$ be the standard map $F_L(x, y) = (L \sin x + 2x - y, x)$,
\item $F_{L, \og}: \bb{T}^2 \to \bb{T}^2$ be the perturbation $F_{L, \og}(x, y) := F_L(x+\og, y)$ by $\og \in \O := \{k\ep: k = 0, \pm 1, \pm2, \ldots, \pm r\}$,
\end{itemize}
Let $\delta \in (0, 1)$. There exists an integer $r_0 = r_0(\delta) > 0$ such that if $r \geq r_0$ and $\ep \in [L^{-1+\delta}, 1/(2r+1))$, then for all large enough $L$,
\begin{enumerate}
\item the random dynamical system defined by $F_{L, \O} := \{F_{L, \og}: \og \in \O\} \subset \Diff^2(\bb{T}^2)$ has positive Lyapunov exponent with respect to the Lebesgue measure on $\bb{T}^2$,
\item every orbit of the system defind by $F_{L, \O}$ is either finite or dense.
\end{enumerate}
\end{theoremA}
\item Outer automorphism group action on character variety
\begin{theoremB}
Let $\mk{X}_s := \Hom_s (F_2, \mathrm{SU(2)})//\mathrm{SU(2)}$ be the relative character variety corresponding to the boundary conjugacy class $s \in [-2, 2]$. Each $\mk{X}_s$ has a natural finite measure $\ld_s$ inherited from the natural measure on $\Hom(F_2, \mathrm{SU(2)})$ that is invariant under the natural action of $\Out(F_2)$ (see Section \ref{cvariety} for the precise definitions and motivations).

There exists a finite set $\ml{S} \subset \Out(F_2)$ without any nontrivial powers of Dehn twists such that for the semigroup $\G$ generated by $\ml{S}$, and for $s = 1.99$, 
\begin{enumerate}[label=(\alph*)]
\item the only $\G$-invariant measure $\nu$ on $\mk{X}_s$ that is not finitely supported is the natural finite measure $\ld_s$.
\item Every orbit of $\G$ on $\mk{X}_s$ is either finite or dense,
\item Each dense $\G$-orbit equidistributes (with respect to $\ml{S}$) on $\mk{X}_s$ (in the precise sense defined in Proposition \ref{equidistribution}). 
\end{enumerate}
\end{theoremB}
\end{enumerate}

In \cite{BXY}, Theorem A(a) was proved when $\O = [-\ep, \ep]$, and $\ep > e^{-L^{2-\delta}}$. However, in this paper, we shall prove a stronger condition (called \emph{uniform expansion}), and we are only able to show this for $\ep > L^{-1+\delta}$. In fact, in a subsequence paper \cite{BXY2}, the same authors essentially showed uniform expansion in the case when $\O = [-\ep, \ep]$ and $\ep > L^{-1+\delta}$ \cite[Prop. 9]{BXY2}. Their method, however, does not apply in this discrete setting, since their approach relies heavily on the fact that any stationary measure is absolutely continuous with respect to Lebesgue measure (see \cite[Lem. 5]{BXY} and \cite[Lem. 8]{BXY2}), which is not necessarily true in the discrete setting.  

In \cite{PX}, the orbit closure classification in Theorem B was proved for $\G = \Out(F_2)$ without going through a measure rigidity result. Instead, the topological dynamics was analyzed directly using critically the fact that $\Out(F_2)$ is generated by two Dehn twists. These Dehn twists take a particularly simple form on the space, which allow an explicit analysis of the orbits generated by them. In this paper, we shall prove uniform expansion for generators $\ml{S}$ of $\Out(F_2)$ that does not have any nontrivial powers of Dehn twists, hence does not admit such explicit analysis. The difference between these two results is analogous to the classical setting of the action on the $2$-torus $\bb{T}^2$ generated by 
$$ \begin{pmatrix} 1 & 1 \\ 0 & 1 \end{pmatrix}, \quaddd \begin{pmatrix} 1 & 0 \\ 1 & 1 \end{pmatrix}, $$
where the action by each individual generator is rotation on a circle, versus the action generated by hyperbolic elements in $\mathrm{SL}(2, \Z)$ that generate a subgroup Zariski dense in $\mathrm{SL}(2, \R)$, where the generic orbit (though certainly not all orbit) of each individual generator is dense in $\bb{T}^2$. 

Our method in the proof of Theorem B goes through a numerical verification using a computer program. We demonstrate such verification on one particular shell $s = 1.99$ and for one particular set of generators $\ml{S}$, though just by some derivative bounds (to be made explicit in Section \ref{cvariety}) the same result can be extended to nearby shells. Such verification is faster for $s$ close to $2$, though there is no theoretical obstruction in applying the same verification to any shells $\mk{X}_s$ with $s \in (-2, 2)$ (just the computation time grows as $s \to -2$). There is also no theoretical obstruction in applying it to other finite subsets $\ml{S}$ that generate a non-elementary subgroup $\G \subset \Out(F_2)$. 

\subsection{Uniform expansion}
As mentioned in the introduction, both theorems are special cases of a more general result. In this section, we shall introduce a general criterion called \emph{uniform expansion}, and state that this criterion implies positivity of Lyapunov exponents, measure rigidity and orbit closure classification.

Given a Riemannian manifold $M$, let $\Diff^k(M)$ be the group of $C^k$ diffeomorphisms on $M$. Given a measure $m$ on $M$, let $\Diff^k_m(M)$ be the group of $C^k$ diffeomorphisms on $M$ that preserve $m$, i.e.
$$ \Diff^k_m(M) := \{f \in \Diff^k(M): f_* m = m \}. $$
Throughout this paper, any measure is assumed to be a Borel probability measure on the corresponding topological space. 

\begin{definition}
A probability measure $\nu$ on $M$ is called \emph{$\mu$-stationary} if 
$$ \mu * \nu = \nu, \quaddd \text{ where } \quad \mu * \nu = \int_{\Diff^2(M)} f_* \nu \; d\mu(f). $$
\end{definition}

\begin{definition}
Let $M$ be a Riemannian manifold, $\mu$ be a measure on $\Diff^2(M)$. We say that $\mu$ is \emph{uniformly expanding} if there exists $C > 0$ and $N \in \N$ such that for all $x \in M$ and $v \in T_x M$, 
$$ \int_{\Diff^2(M)} \log \frac{\| D_x f(v) \|}{\| v\|} d\mu^{(N)}(f) > C. $$
Here $\mu^{(N)} := \mu * \mu * \cdots * \mu$ is the $N$-th convolution power of $\mu$. We remark that if $M$ is compact, this is equivalent to the weaker formulation where we allow $C$ and $N$ to depend on $x$ and $v$ (see e.g. \cite[Lem. 4.3.1]{LX}, where such weaker criterion is called ``weakly expanding''). 
\end{definition}

Sometimes we say that a finite subset $\ml{S} \subset \Diff^2(M)$ is \emph{uniformly expanding} if the uniform measure supported on $\ml{S}$ is uniformly expanding in the above sense. Note that in this case the integral in the uniform expansion condition reduces to a finite sum. 

The goal of the first half of the paper is to classify $\mu$-stationary measures on a closed surface $M$ and the corresponding orbit closures if $\mu$ is uniformly expanding and supported on $\Diff^2_m M$ for some \emph{smooth} measure $m$ on $M$, i.e. a Borel probability measure $m$ equivalent to the Riemannian volume on $M$.

\begin{theoremC} \label{mainmeasurerigidity}
Let $M$ be a closed surface (compact connected two-dimensional Riemannian manifold) and $m$ be a smooth measure on $M$. Let $\mu$ be a uniformly expanding probability measure on $\Diff^2_m(M)$ with 
\begin{equation}
\tag{*}
 \int_{\Diff^2_m(M)} \log^+(|f|_{C^2}) + \log^+(|f^{-1}|_{C^2}) \; d\mu(f) < \infty. 
 \end{equation}
Let $\nu$ be an ergodic, $\mu$-stationary Borel probability measure on $M$. Then
\begin{enumerate}[label=(\alph*)]
\item \label{mainpositiveexponent} $\nu$ has positive Lyapunov exponent;
\item \label{mainmeasurerigidity} either $\nu$ is finitely supported, or $\nu = m$. 
\end{enumerate}
\end{theoremC}
\noindent This result was proved in \cite[Thm. 4.1.4]{LX}, where they used this statement to prove a large deviation result. We shall recall the proof in Section \ref{pexponent} and \ref{mrigidity} for completeness.

Here we are more concerned with the following orbit closure classification which follows from Theorem C, and its applications in concrete settings. 


\begin{theoremD} \label{mainorbitclosure}
Let $M$ be a closed surface, $m$ be a smooth measure on $M$, and $\ml{S} \subset \Diff^2_m(M)$ be a finite subset of diffeomorphisms that preserve $m$. Let $\G \subset \Diff^2_m(M)$ be the subsemigroup generated by $\ml{S}$. If $\ml{S}$ is uniformly expanding, then 
\begin{enumerate}[label=(\alph*)]
\item every orbit of $\G$ is either finite or dense, 
\item every dense $\G$-orbit equidistributes on $M$ (in the precise sense defined in Proposition \ref{equidistribution}).
\end{enumerate}
\end{theoremD}

Note that we could have replaced the word ``subsemigroup'' with ``subgroup'' to get a weaker statement. Also if $\ml{S}$ is uniformly expanding, then $\G$ cannot be cyclic (see Lemma \ref{nonrandom} below). An analogous statement has been proved in greater generality in the homogeneous setting by Eskin and Lindenstrauss \cite{EL}. 

In the setting of homogeneous dynamics, uniform expansion has been verified in some cases. For instance, let $G$ be a real semisimple Lie group with no compact factors and $\Lambda$ be a discrete subgroup of $G$. Let $\mu$ be a countably supported probability measure on $G$ whose support generates a Zariski dense subgroup of $G$. Then $\mu$ is uniformly expanding, see e.g. \cite[Lem. 4.1]{EMar}, the idea of which goes back to Furstenberg \cite{F2}. As a second example, one may consider the case of the $\mathrm{SL}(n, \Z)$-action on the $n$-torus $\T^n := \R^n / \Z^n$. Let $\mu$ be a finitely supported probability measure on $\mathrm{SL}(n, \Z)$ such that the support of $\mu$ generates a Zariski dense subgroup of $\mathrm{SL}(n, \R)$. Using the classical theory of product of random matrices (for instance in Goldsheid and Margulis \cite{GM}), one can show that $\mu$ is uniformly expanding (see e.g. the proof of Theorem 4.1.3 in \cite{LX} for the precise argument). Clearly uniform expansion is a $C^1$-open property, therefore any small $C^1$-perturbations of these examples also support uniformly expanding measures. 

\subsection{Verification of Uniform Expansion}

Theorem A and B are both proved by verifying uniform expansion and then applying Theorem C and D. Theorem A will be proved in Section \ref{smap} by verifying uniform expansion analytically. Theorem B will be proved in Section \ref{cvariety} by verifying uniform expansion numerically, using an algorithm described in Section \ref{computer}. The context and motivation will be provided in the respective sections.

Other than the fact that these examples are interesting in their own right, they are also chosen to illustrate how to overcome two difficulties in the verification of uniform expansion. 

First of all, as we saw in Theorem C, uniform expansion is a stronger criterion than positivity of Lyapunov exponent, and the latter is notoriously difficult to verify for one-parameter group actions without some sort of uniform hyperbolicity. The reason is that even strong expansion in the early stages of the dynamics can be cancelled out by strong contraction in the future, for instance when the dynamics hit a region where it behaves like a rotation, such ``backtracking'' phenomenon may occur. In our examples, there are small rotation regions for each individual map. Nonetheless we show that as long as the random dynamics enter these rotation regions with small enough probability, the overall dynamics is expanding on average. 

Secondly, it is clear that if the dynamics is generated by a single volume-preserving hyperbolic diffeomorphism, then uniform expansion never holds, since the stable direction is contracted by the dynamics. For higher rank actions, it is still possible that the contracting directions of the maps may overlap for some subset of points but not all. Note that this does not happen in the homogeneous setting, in the sense that if the contracting directions are separated at one point, then by homogeneity, they are separated at every point of the space. In our examples, the contracting directions may overlap in a codimension one subset, and again we show that uniform expansion occurs as long as the random dynamics enter a neighborhood of such subset with small enough probability. Proposition \ref{generalcriterion} illustrates that rotation regions and overlapping contracting directions are essentially the only two obstructions to uniform expansion. 

For Theorem A, we are able to verify uniform expansion directly since at each point, with high probability, the map has strong expansion in the same (horizontal) direction. Moreover, one can compute with high accuracy the separation of the contracting directions of the maps. These allow us to understand exactly where the rotation regions and overlapping contracting directions occur. In particular, for each point and each direction, we can obtain an upper bound on the probability that the map contracts in that direction after $n$ steps. Depending on how small the separation of the contracting directions is, one can then choose a suitable $N$ so that uniform expansion occurs. 

For Theorem B, however, the contracting directions of each map vary for different points on the space. In particular, we can no longer prove explicitly that backtracking occur with low probability (though we expect so). Therefore we can only check unifom expansion at each point on a fine enough grid, and then show that such expansion still occur at neighboring points using a $C^2$-bound. 

The paper is structured as follows: 
\begin{itemize}
\item In Section \ref{pexponent}, positivity of Lyapunov exponents for uniformly expanding systems (Theorem C (a)) is proved (Proposition \ref{positiveexponent}). 
\item In Section \ref{mrigidity}, classification of stationary measures of uniformly expanding systems (Theorem C (b)) is proved using a result of Brown and Rodriguez-Hertz \cite{BR} (Proposition \ref{measurerigidity}). 
\item In Section \ref{oclosures}, using the measure rigidity result in Section \ref{mrigidity}, an equidistribution result (Proposition \ref{equidistribution}) will be proved. The orbit closure classification (Theorem D) is then obtained as a corollary (Proposition \ref{orbitclosure}).
\item In Section \ref{gcriterion}, we introduce a geometric way to view uniform expansion and prove a general criterion for uniform expansion (Proposition \ref{generalcriterion}). 
\item In Section \ref{smap}, the setting of perturbation of the standard map is introduced, and uniform expansion is verified analytically in this setting (Proposition \ref{uniexpand}). This proves Theorem A. 
\item In Section \ref{computer}, an algorithm to check uniform expansion is presented. 
\item In Section \ref{cvariety}, the setting of the $\Out(F_2)$ action on character variety is introduced, and uniform expansion is verified using the algorithm introduced in Section \ref{computer}. This proves Theorem B.
\end{itemize}

\subsection*{Acknowledgements}
The author is grateful to his advisor Alex Eskin for introducing him to this circle of problems and for his helpful discussions, encouragement and patience. His advices and insights are invaluable to this work. He would also like to express his gratitudes to Aaron Brown for numerous insightful discussions. It is a pleasure to thank Kiho Park and Disheng Xu for reading over an earlier draft and suggesting helpful improvements to the paper. He would also like to thank Amie Wilkinson for helpful discussions about the proof of Proposition \ref{allofM}.

\section{Positive exponent} \label{pexponent}

We first recall the celebrated Oseledets theorem in the setting of random dynamical systems. Here we adopt the notation in \cite{BR} and define $f_{\og}^n := \og_{n-1} \circ \og_{n-2} \circ \cdots \circ \og_1 \circ \og_0$ for $\og = (\og_0, \og_1, \og_2, \ldots) \in \Diff^2(M)^\N$. Let $\sigma: \Diff^2(M)^\N \to \Diff^2(M)^\N$ be the left shift map given by $(\og_0, \og_1, \og_2, \ldots) \mapsto (\og_1, \og_2, \og_3, \ldots)$. 

\begin{proposition}[Random Oseledets multiplicative ergodic theorem]
Let $M$ be a closed smooth Riemannian manifold, $\mu$ be a measure on $\Diff^2(M)$ satisfying the moment condition (*). Let $\nu$ be an ergodic, $\mu$-stationary Borel probability measure. Then there are numbers $\ld_1(\nu) > \ld_2(\nu) > \cdots > \ld_\ell(\nu)$ such that for $\mu^\N$-almost every sequence $\og \in \Diff^2(M)^\N$ and $\nu$-almost every $x \in M$, there is a filtration
$$ T_x M = V_\og^1(x) \supsetneqq V_\og^2(x) \supsetneqq \cdots \supsetneqq V_\og^\ell(x) \supsetneqq V_\og^{\ell+1} = 0 $$
such that for $v \in V_\og^{k}(x) \setminus V_\og^{k+1}(x)$, 
$$ \lim_{n \to \infty} \frac{1}{n} \log \frac{\| D_x f_\og^n(v) \|}{\| v \|} = \ld_k(\nu). $$
The subspaces $V_\og^i(x)$ are invariant in the sense that
$$ D_x f_\og V_\og^k(x) = V_{\sigma(\og)}^k(f_\og(x)). $$
\end{proposition}
\noindent For a proof of the theorem, see e.g. \cite[Prop. I.3.1]{LQ}. 

\begin{proposition} [Uniform positive exponent] \label{positiveexponent}
Let $M$ be a closed surface, $\mu$ be a uniformly expanding probability measure on $\Diff^2(M)$ satisfying (*). Then there exists a uniform constant $\ld_\mu > 0$, depending only on $\mu$, such that for \emph{all} $x \in M$, and $\mu^\N$-almost every $\og \in \Diff^2(M)^\N$, there exists $\ld(\og, x) \geq \ld_\mu$ such that
\begin{align*}
\liminf_{n \to \infty} \frac{1}{n} \log \| D_x f_\og^n \| = \ld(\og, x).
\end{align*}
In particular for all ergodic, $\mu$-stationary probability measure $\nu$, for $\nu$-almost every $x \in M$ and $\mu^\N$-almost every $\og$, the top Lyapunov exponent $\ld_1(\nu) = \ld(\og, x) \geq \ld_\mu > 0$. 
\end{proposition}

\begin{proof}[Sketch of Proof]
The point of this proposition is that assuming uniform expansion, Oseledets theorem holds for \emph{every} point $x \in M$ and almost every sequence $\og \in \Diff^2(M)^\N$, and the top exponent is positive. See Lemma 4.3.5 of \cite{LX} for a more refined version of this proposition, where it is shown that there is an Oseledets splitting for every point. Here we only need positivity of exponent. We include a sketch of the proof here for completeness. 

Let $T^1 M$ be the unit tangent bundle of $M$. By definition of uniform expansion, there exists $C > 0$ and $N \in \N$ such that for all $(x, v) \in T^1 M$, 
$$ \int \log \| D_x f(v) \| d\mu^{(N)}(f) > C. $$

Let $(x, v_0) \in T^1 M$. For each $\og \in \Diff^2(M)^\N$ and $n \in \N$, let
$$ (x_n, v_n) = (x_n(\og), v_n(\og)) := \left( f_\og^n(x), 
\frac{D_x f_\og^n (v_0)}{\| D_x f_\og^n (v_0) \|} \right) $$
be the image of $(x, v_0)$ in $T^1 M$ after $n$ steps of the random dynamics following the sequence $\og$. 
For $k \geq 1$, consider the event 
$$ X_k(\og) := \log \| D_{x_{(k-1)N}} f_{\sigma^{(k-1)N}\og}^{N} (v_{(k-1)N}) \| - \int \log \| D_{x_{(k-1)N}} f(v_{(k-1)N}) \| d\mu^{(N)}(f). $$
Notice that
$$ X_k(\og) = \log \frac{\| D_x f_{\og}^{kN} (v_0) \|}{\| D_x f_{\og}^{(k-1)N} (v_0) \|}  - \int \log \| D_{x_{(k-1)N}} f(v_{(k-1)N}) \| d\mu^{(N)}(f). $$
Let $S_j = \sum_{k=1}^n X_k$. Then
$$ S_j(\og) = \log \| D_x f_\og^{jN} (v_0) \| - \sum_{k=1}^j \int \log \| D_{x_{(k-1)N}} f(v_{(k-1)N}) \| d\mu^{(N)}(f). $$
Thus by uniform expansion,
$$ \log \| D_x f_\og^{jN} (v_0) \| = S_j(\og) + \sum_{k=1}^j \int \log \| D_{x_{(k-1)N}} f(v_{(k-1)N}) \| d\mu^{(N)}(f) \geq S_j(\og) + jC. $$
The main observation is that the family $\{S_n\}_{n \in \N}$ form a square integrable martingale. Then by the strong law of large numbers for square integrable martingales, for $\mu^\N$-almost every $\og \in \Diff^2(M)^\N$, we have the limit  
$$ \lim_{n \to \infty} \frac{S_n}{n} = 0. $$
Thus if we write $j = \lfloor n/N \rfloor$, then $\ds\lim_{n \to \infty} j/n = 1/N$, and we have for almost every $\og$,
\begin{align*}
\liminf_{n \to \infty} \frac{1}{n} \log \| D_x f_\og^n(v_0) \| &= \liminf_{n \to \infty} \frac{1}{n} \log \| D_{x_{jN}} f_{\sigma^{jN}\og}^{n-jN} (v_{jN}) \| + \liminf_{n \to \infty} \frac{1}{n} \log \|  D_x f_\og^{jN} (v_0) \| \\
&\geq \liminf_{n \to \infty} \left( 0+\frac{S_j(\og)}{n} + \frac{jC}{n} \right) \geq \frac{C}{N} > 0. 
\end{align*}
Hence we can take $\ld_\mu := C/N$.
\end{proof}

\section{Measure rigidity} \label{mrigidity}

We prove the measure rigidity result in this section. The precise statement was already proved in \cite[Thm. 4.1.4]{LX}. We include the proof here for completeness. 

The main input of the proof is a result of Brown and Rodriguez-Hertz \cite[Thm. 3.4]{BR}. This result provides a trichotomy for the ergodic $\mu$-stationary Borel probability measures $\nu$: either the stable distribution is non-random, $\nu$ is finitely supported or $\nu$ is an ergodic component of the volume on $M$. The uniform expansion condition eliminates the possibility that the stable distribution is non-random. The same condition also implies that the volume is $\G$-ergodic using a refined version of the classical Hopf argument inspired by \cite[Sect. 10]{DK}, as detailed in \cite[Prop. 4.4.1]{LX}. 

\subsection{Main statement}

\begin{proposition}[Measure Rigidity] \label{measurerigidity}
Let $M$ be a closed surface, $\G \subset \Diff^2(M)$ be a subgroup that preserve a smooth measure $m$ on $M$. Let $\mu$ be a uniformly expanding probability measure on $\Diff^2(M)$ with $\mu(\G) = 1$ satisfying (*). Let $\nu$ be an ergodic, $\mu$-stationary Borel probability measure on $M$. Then either $\nu$ is finitely supported or $\nu = m$. 
\end{proposition}

\noindent Following \cite{BR}, we write
$$ E_\og^s(x) := \bigcup_{\ld_j < 0} V_\og^j(x) = \left\{ v \in T_x M : \limsup_{n \to \infty} \frac{1}{n} \log \frac{\| D_x f_\og^n(v) \|}{\| v \|} < 0 \right\}. $$
for the \emph{stable Lyapunov subspace} for the word $\og$ at the point $x \in M$. We say that the stable distribution is \emph{non-random} if there exists $\mu$-almost surely invariant $\nu$-measurable subbundle $\hat{V} \subset TM$ such that $\hat{V}(x) = E_\og^s(x)$ for $(\mu^\N \times \nu)$-almost every $(\og, x)$, i.e. $Df(E_{\og}^s(x)) = E_\og^s(f(x))$ for $\mu$-a.e. $f \in \Diff^2(M)$ and $\nu$-a.e. $x \in M$. 

Given a smooth probability measure $m$ on $M$, let $\Diff^2_m(M) := \{f \in \Diff^2(M) \mid f_* m = m\}$. We recall the theorem of Brown and Rodriguez Hertz.

\begin{theorem} \cite[Thm. 3.4]{BR}
Let $M$ be a closed surface, $\G \subset \Diff^2(M)$ be a subgroup that preserve a smooth measure $m$ on $M$. Let $\mu$ be a uniformly expanding probability measure on $\Diff^2_m(M)$ with $\mu(\G) = 1$ satisfying (*). Let $\nu$ be an ergodic, hyperbolic $\mu$-stationary Borel probability measure on $M$. Then either
\begin{enumerate}[label=(\arabic*)]
\item $\nu$ has finite support,
\item the stable distribution $E_\og^s(x)$ is non-random, or
\item $\nu$ is - up to normalization - the restriction of $m$ to a positive volume subset. 
\end{enumerate}
\end{theorem}

It remains to refine the conclusion of this theorem using the condition of uniform expansion. We will eliminate the second possibility in the next lemma, and refine the third possibility in the next subsection.

\begin{lemma} \label{nonrandom}
If $\mu$ is uniformly expanding, then the stable distribution is not non-random. 
\end{lemma}

\begin{proof}
Assume that the stable distribution $E_\og^s(x)$ is non-random, i.e. there is a $\mu$-almost surely invariant subbundle $\hat{V} \subset TM$ with $\hat{V}(x) = E_\og^s(x)$ for $(\mu^{\N} \times \nu)$-a.e. $(\og, x)$. By definition of the stable distribution, for $\nu$-almost every $x \in M$, for all large enough $n$, we have $\log (\| D_x f_w^n (v) \| / \| v \|) < 0$ for all nonzero $v \in E_\og^s(x)$. Hence by taking average, for $\nu$-almost all $x \in M$, and for all nonzero $v \in \hat{V}(x)$, we have 
$$ \int_{\Diff^2(M)} \log \frac{\| D_x f(v) \|}{\| v \|} d\mu^{(n)}(f) < 0 $$
for all large enough $n$. However, this contradicts the uniform expansion property of $\mu$, as it is striaghtforward from definition that there exists $C > 0$ and $N \in \N$ such that for all $x \in M$, nonzero $v \in T_xM$ and $k \in \N$, 
$$ \int_{\Diff^2(M)} \log \frac{\| D_x f(v) \|}{\| v\|} d\mu^{(kN)}(f) > kC. $$
\end{proof}

\subsection{Ergodicity}

The main theorem of \cite[Thm. 3.1]{BR} did not assume the existence of a smooth invariant measure, in which case the third possibility is that the stationary measure is SRB (see \cite[Def. 6.8]{BR} for a precise definition). The existence of a smooth invariant measure $m$ allows the authors to refine the third possibility to being a restriction of $m$ to a positive volume subset using a local ergodicity argument (see \cite[Ch. 13]{BR}), as stated above. 

In this section, using uniform expansion, we further refine the third possibility to show that the stationary measure has to be the smooth invariant measure $m$.

\begin{proposition} \label{allofM}
Let $M$ be a closed (connected) surface, $\mu$ be a Borel probability measure on $\Diff^2_m(M)$. Suppose there exists a positive volume subset $A \subset M$ such that $\nu := \frac{1}{m(A)}m|_A$ is an ergodic $\mu$-stationary Borel probability measure. If $\mu$ is uniformly expanding, then in fact $\nu = m$. 
\end{proposition}

This is proved in \cite[Prop. 4.4.1]{LX} based on ideas from \cite[Sect. 10]{DK}. For completeness we give a detailed outline of the proof. 

The main idea of the proof is to perform a version of the classical Hopf argument. Rather than transversing along the stable and unstable leaves as in the setting of Anosov systems, the argument goes by transversing along the stable leaves $W_\og^s(x)$ and $W_{\og'}^s(x)$ of two distinct words $\og, \og'$ with suitable geometric and dynamical properties.

\subsubsection{Classical facts about the stable manifolds of a random system}

We first collect some standard facts about stable manifolds of a random dynamical system. 

Given $x \in M$ and $\og \in \Diff^2(M)^\N$, let
$$ W_\og^s(x) := \left\{y \in M \mid \limsup_{n \to \infty} \frac{1}{n} \log d(f_\og^n(x), f_\og^n(y)) < 0 \right\}. $$
There exists a $(\mu^\N \times \vol)$-co-null set $\Ld \subset \Diff^2(M)^\N \times M$ such that $W_{\og}^s(x)$ is a $C^2$-embedded curve in $M$ for all $(\og, x) \in \Ld$. We call $W_\og^s(x)$ the \emph{global stable manifold} at $x$ for $\og$. 

We define local stable manifolds using the classical stable manifold theorem (we only list properties needed for our purpose). 
\begin{theorem}[Local stable manifold theorem] \label{localstable}
Let $\ld_\mu > 0$ be the constant from Proposition \ref{positiveexponent}. For every $0 < \ep < \ld_\mu / 200$, for $\mu^\N$-almost every word $\og \in \Diff^2(M)^\N$, there exists a full volume set $\Ld_\og \subset M$ and a measurable family of \emph{local stable manifolds} $\{W_{\og, \loc}^s(x)\}_{x \in \Ld_\og}$ with the following properties:
\begin{enumerate}[label=(\alph*)]
\item $W_{\og, \loc}^s(x)$ is a $C^2$ embedded curve, i.e. the image of a $C^2$ embedding $\psi: (-1, 1) \to M$.
\item $T_x W_{\og, \loc}^s(x) = E_\og^s(x)$. 
\item for $n \geq 0$, $f_\og^n(W_{\og, \loc}^s(x)) \subset W_{\og, \loc}^s(f_\og^n(x))$.
\item for $y, z \in W_{\og, \loc}^s(x)$ and $n \geq 0$, 
$$ d(f_\og^n(y), f_\og^n(z)) \leq L(\og, x) e^{(-\ld_\mu + \ep)n} d(y, z), $$
where $L: \Diff^2(M)^\N \times M \to [1, \infty)$ is a Borel measurable function such that for all $x \in \Ld_\og$ and $n \geq 0$,
$$ L(\sigma^n(\og), f_\og^n(x)) \leq e^{n \ep} L(\og, x). $$
Here $\sigma: \Diff^2(M)^\N \to \Diff^2(M)^\N$ is the left shift given by $\sigma(\og)_n := \og_{n+1}$.
\item $W_\og^s(x) = \ds\bigcup_{n \geq 0} (f_\og^n)^{-1}(W^s_{\sigma^n(\og), \loc}(f_\og^n(x)))$.
\end{enumerate}
\end{theorem}
\noindent We refer to \cite[Ch. 7]{BP} for a treatment in the deterministic setting, and \cite[Ch. III.3]{LQ} in the random setting.

\begin{definition}[Measures on stable leaves] We recall the following notions related to the induced volume measure on the local stable manifolds.
\begin{enumerate}
\item Given $r > 0$ and $(\og, x) \in \Ld$, let $W_{\og, r}^s(x) := \{y \in W_\og^s(x) \mid d_{W^s}(x, y) < r\}$, where $d_{W^s}$ is the Riemannian distance along the $C^2$-curve $W_\og^s(x)$.
\item Given a $C^1$-curve $\g$ on $M$, there is a natural measure on $\g$ induced by the restriction of the Riemannian metric on $M$ to $\g$. We call this measure the \emph{leaf-volume} of $\g$, denoted $\vol_\g$. 
\item Given a measurable subset $T \subset W_\og^s(x)$ for some word $\og$ and point $x \in M$, we write 
$$ \vol_{W^s}(T) := \vol_{W_\og^s(x)}(T), $$
as the dependence on $\og$ and $x$ is clear from the definition of $T$. 
\item Unless otherwise specified, ``almost every'' point on $\g$ means almost every point with respect to the leaf-volume. 
\end{enumerate}
\end{definition}

We will also need the standard fact that for $(\mu^\N \times \vol)$-almost every $(\og, x)$, the stable manifold $W_\og^s(x)$ satisfies two versions of absolute continuity that we will describe in the next lemma.

By Lusin theorem and Theorem \ref{localstable} (a), for all $\delta > 0$, there exists a measurable subset $Q \subset M$ with $\vol(Q) > 1 - \delta$ such that $W_{\og, \loc}^s(y)$ varies continuously in $y \in Q$ in the $C^2$ topology.

\begin{lemma}[Absolute Continuity]\label{uniformabsolutecontinuity} 
For $(\mu^\N \times \vol)$-almost every $(\og, x) \in \Diff^2(M)^\N \times M$, for sufficiently small $R > 0$, the family of local stable manifolds $\ml{F} := \{W_{\og, \loc}^s(y)\}_{y \in Q \cap B(x, R)}$ satisfies the following properties:
\begin{enumerate}
\item For all $y \in Q \cap B(x, R)$, $W_{\og, \loc}^s(y)$ intersects $\pl B(x, R)$ at two points.
\item For $y, y' \in Q \cap B(x, R)$, if $y' \in W_{\og, \loc}^s(y)$, then $W_{\og, \loc}^s(y) \cap B(x, R) = W_{\og, \loc}^s(y') \cap B(x, R)$.
\end{enumerate}
Then the following two versions of absolute continuity hold (we write $\ml{F}(y)$ for the element in $\ml{F}$ containing the point $y$). 
\begin{enumerate}[label = (AC\arabic*)]
\item Let $\g_1$ and $\g_2$ be two $C^1$-curves in $B(x, R)$ everywhere uniformly transverse to $\ml{F}$.
Let 
$$ T_1 := \g_1 \cap \bigcup_{y \in \g_2} \ml{F}(y), \quaddd \text{ and } \quaddd T_2 := \g_2 \cap \bigcup_{y \in T_1} \ml{F}(y). $$
Define the holonomy map $h_\ml{F}: T_1 \to T_2$ given by ``sliding'' along the leaves in $\ml{F}$, i.e. $h_\ml{F}(y)$ is the only point in $\g_2 \cap \ml{F}(y)$ for all $y \in T_1$.

Then on $T_2$, we have
$$ \vol_{\g_2} \ll (h_\ml{F})_* \vol_{\g_1}. $$
\item For any Borel subset $A \subset M$, we have
$$ \vol(A) = 0 \quaddd \Leftrightarrow \quaddd \vol_{W_\og^s(y)}(A \cap W_\og^s(y)) = 0 \quad \text{ for } \quad \vol\text{-a.e.} \;\; y \in M. $$
\end{enumerate}
\end{lemma}
See \cite[Ch. 8]{BP} for a statement in the case of deterministic systems, \cite[Sect. 4.2]{LY} or \cite[Sect. III.5]{LQ} for a statement in the case of random systems. 

\subsubsection{Implications of uniform expansion}

One consequence of uniform expansion is uniform control on the angles between stable directions of different words.

\begin{lemma}[Uniform avoidance of the stable direction] \cite[Prop. 4.4.4]{LX} \cite[Prop. 3]{Z} \label{uniformangle}
If $\mu$ is uniformly expanding, then there exists $\alpha > 0$ with the following property: \\
for any $(x, v) \in T^1 M$, there exists a subset $\G_{x, v} \subset \Diff^2(M)^\N$ with $\mu^\N(\G_{x, v}) > 0.99$ such that, for any $\og \in \G_{x, v}$, 
$$ \measuredangle (E_\og^s(x), v) > \alpha. $$
\end{lemma}

Another property of uniformly expanding systems is that for every point on the surface, the dynamics exhibit uniform hyperbolicity for a large proportion of words. This implies uniform control on the lengths and curvatures of the local stable manifolds. 

\begin{lemma}[Uniform control of the local stable manifolds] \cite[Prop. 4.4.9]{LX} \cite[Prop. 3]{Z} \label{uniformsize}
If $\mu$ is uniformly expanding, then there exist a constant $\ell = \ell(\mu) > 0$ with the following properties: \\
for any $x \in M$, there exists a subset $\Ld_x \subset \Diff^2(M)^\N$ with $\mu^\N(\Ld_x) > 0.99$ such that for all $\og \in \Ld_x$,
\begin{enumerate}[label=(\roman*)]
\item $W_{\og, \ell}^s(x) \subset \subset W_{\og, \loc}^s(x)$,
\item the \emph{angle change} of the curve $\exp_x^{-1}(W_{\og, \ell}^s(x))$ is less than $\alpha/100$.
\end{enumerate}
Here $\alpha$ is as in Lemma \ref{uniformangle}, and for a $C^1$-curve $\g: [a, b] \to \R^2$, the \emph{angle change} of $\g$ is 
$$ \max_{t, s \in [a, b]} \measuredangle(\g'(t), \g'(s)). $$
The notation $A \subset \subset B$ in (i) means $A$ is compactly contained in $B$, i.e. the closure of $A$ is compact and is contained in $B$. Note that (i) implies that the leaf-volume of $W_{\og, \ell}^s(x)$ is at least $2\ell$ since the condition implies, in particular, that $W_{\og, \ell}^s(x) \subsetneq W_\og^s(x)$. 
\end{lemma}
We say that $W_{\og, \loc}^s(x)$ is a \emph{nice curve} if $\og \in \Ld_x$. 

We shall use these constants $\alpha$ and $\ell$, which depend only on $\mu$, later in the proof. The set $\Ld_x$ of words in Lemma \ref{uniformsize} will also appear a few times in the proof.

\subsubsection{Basin of $\nu$}

We will consider the classical notion of a basin of $\nu$ in this random setting, and remark that to show that $\nu = m$, it suffices to show that the basin $B(\nu)$ has full volume. This will be used in {\bf Step 1} below.

\begin{definition}
Given $x \in M$, $\og \in \Diff^2(M)^\N$ and a continuous function $\varphi: M \to \R$, define the $\og$-\emph{Birkhoff average} of $\varphi$ at $x$ as
$$ S_\og(\varphi)(x) := \lim_{n \to \infty} \frac{1}{n} \sum_{j = 0}^{n-1} \varphi(f_{\og}^j(x)) $$
if the limit on the right exists.
\end{definition}

\begin{definition}
Given an ergodic $\mu$-stationary measure $\nu$ on $M$, define the \emph{basin} of $\nu$, denoted $B(\nu) \subset M$, as the set of points $x \in M$ such that for any continuous function $\varphi: M \to \R$ and $\mu^\N$-almost every $\og \in \Diff^2(M)^\N$, 
$$ S_\og(\varphi)(x) = \int_M \varphi \; d\nu. $$
\end{definition}

\begin{lemma} \label{ergodicity}
If $\vol(B(\nu)) = 1$, then $\nu = m$. 
\end{lemma}
 
\begin{proof}
Assume that $\mathrm{vol}(B(\nu)) = 1$. Then $m(B(\nu)) = 1$. Let $\varphi \in C^0(M)$. By the pointwise ergodic theorem (and the argument in the proof of Lemma \ref{wregular}), there exists a function $\overline{\varphi}(x)$ such that for $(\mu^\N \times m)$-a.e. $(\og, x)$, 
$$ S_\og(\varphi)(x) = \overline{\varphi}(x) \quaddd \text{ and } \quaddd \int \overline{\varphi}(x) \; dm(x) = \int \varphi(x) \; dm(x). $$
On the other hand, by definition of the basin $B(\nu)$, for all $x \in B(\nu)$, we have
$$ S_\og(\varphi)(x) = \int \varphi \;d\nu. $$
Therefore $\overline{\varphi}(x) = \int \varphi \; d\nu$ for all $x \in B(\nu)$. But since $m(B(\nu)) = \mathrm{vol}(B(\nu)) = 1$, we have
$$ \int \varphi(x) \; dm(x) = \int \overline{\varphi}(x) \; dm(x) = \int_{B(\nu)} \overline{\varphi}(x) \; dm(x) = \int \varphi(x) \; d\nu(x). $$
Since $\varphi$ is arbitrary, we have $\nu = m$.
\end{proof}

\subsubsection{Reduction to a local argument via Lebesgue density theorem}

By Lemma \ref{ergodicity}, it suffices to argue that the basin has full volume. In this section, we argue that it suffices to show that in every (uniformly) small enough ball, the basin either has zero density or has density bounded from below by a positive uniform constant. This allows us to reduce the problem to a local argument in a small ball. This will be used in {\bf Step 2} below.

\begin{definition}[Density]
Given a Borel measurable subset $U \subset M$, a point $x \in M$ and $r > 0$, define the \emph{density} of $U$ in the ball $B(x, r)$ as
$$ \vol(U : B(x, r)) := \frac{\vol(U \cap B(x, r))}{\vol(B(x, r))}. $$
\end{definition}

\begin{lemma} \label{goodlemma}
Assume that a measurable subset $U \subset M$ satisfies the following: there exist $c > 0$ and $R_0 > 0$ such that for all $x \in M$ and positive $r < R_0$, either
$$ \vol(U: B(x, r)) = 0 \quaddd \text{ or } \quaddd \vol(U: B(x, r)) > c. $$
Then $\mathrm{vol}(U) = 0$ or $1$. 
\end{lemma}

\begin{proof}
Assume the contrary that $\mathrm{vol}(U) \in (0, 1)$. Clearly the assumption continues to hold if we decrease $c$. Thus without loss of generality assume that $0 < c < 1/2$. 

Since $\vol(U)$ and $\vol(U^c)$ are both positive by assumption, by Lebesgue density theorem, there exist $y \in U$, $z \in U^c$ and $r \in (0, R_0)$ such that 
$$ \vol(U: B(y, r)) > 1 - c \quaddd \text{ and } \quaddd \vol(U: B(z, r)) < c/4. $$

Now observe that the function $x \mapsto \vol(U: B(x, r))$ is continuous in $x \in M$ for fixed $U \subset M$ and $r > 0$. Since $M$ is connected, there exists $x \in M$ such that $\vol(U: B(x, r)) = c/2$. This yields a contradiction.
\end{proof}

In the rest of this section, we shall find uniform constants $c > 0$ and $R_0 > 0$ so that the assumptions of Lemma \ref{goodlemma} hold for the basin $U = B(\nu)$. 

\subsubsection{Regular points}

Similar to the proof of ergodicity in \cite[Sect. 10]{DK}, we define a notion of regular points, and show that almost every point on $M$ is regular. This will be used in {\bf Step 3} of the main argument.

Informally, the notions of regular points can be summarized as follows: for $x \in M$ and $\og \in \Diff^2(M)^\N$, 
\begin{enumerate}
\item $x$ is $\og$-\emph{regular} if the $\og$-Birkhoff averages at $x$ agree with the $\og'$-Birkhoff averages at $x$ for $\mu^\N$-a.e. $\og'$.
\item $x$ is \emph{regular} if for $\mu^\N$-a.e. $\og$, $x$ is $\og$-regular and almost every $y \in W_{\og}^s(x)$ is $\og$-regular.
\end{enumerate}

\begin{definition}
For $\og \in \Diff^2(M)^\N$, a point $x \in M$ is called $\og$-\emph{regular} if for $\mu^\N$-almost every $\og' \in \Diff^2(M)^\N$, for any continuous function $\varphi: M \to \R$, we have
$$ S_\og(\varphi)(x) = S_{\og'}(\varphi)(x) $$
(in particular the Birkhoff averages exist).
\end{definition}

\begin{remark} \label{manyregular}
Note that if $x$ is $\og$-regular, then for $\mu^\N$-almost every $\og' \in \Diff^2(M)^\N$, $x$ is $\og'$-regular.
\end{remark}

\begin{lemma} \cite[Cor. I.2.2, Page 24]{Kif} \label{wregular}
For $\mu^\N \times \mathrm{vol}$-almost every $(\og, x) \in \Diff^2(M)^\N \times M$, $x$ is $\og$-regular.
\end{lemma}

\begin{definition}
A point $x \in M$ is called \emph{regular} if for $\mu^\N$-almost every word $\og \in \Diff^2(M)^\N$, $x$ is $\og$-regular and almost every point $y \in W_\og^s(x)$ is $\og$-regular. 
\end{definition}

It can be shown using Lemma \ref{wregular} and absolute continuity of the stable manifolds that almost every point on $M$ is regular.

\begin{lemma} \cite[Lem. 4.4.18]{LX} \label{almostregular}
$\mathrm{vol}$-almost every point $x \in M$ is regular. 
\end{lemma}

\begin{proof}
We need to show that the set
$$ B_1 = \{(\og, x) \in \Diff^2(M)^\N \times M \mid \vol_{W^s}(\{y \in W_\og^s(x) \mid y \text{ is not } \og\text{-regular}\}) > 0\} $$
has $\mu^\N \times \mathrm{vol}$-measure zero.
By Lemma \ref{wregular}, for $\mu^\N$-almost every word $\og$, we have
$$ \vol(\{y \in M \mid y \text{ is not } \og\text{-regular}\}) = 0. $$
By absolute continuity of the foliation $W_\og^s$ (Lemma \ref{uniformabsolutecontinuity} (AC2), ignore a null set of words $\og$ if necessary), for $\vol$-almost every point $x \in M$, we have
$$ \vol_{W^s}(\{y \in W_\og^s(x) \mid y \text{ is not } \og\text{-regular}\}) = 0. $$
This is enough to show that $B_1$ has measure zero.
\end{proof}

%
%
%

The following lemma is a direct consequence of the definitions, and will be used repeatedly in {\bf Step 6}.

\begin{lemma}\cite[Lem. 4.4.19]{LX} \label{basin}
Given an ergodic $\mu$-stationary measure $\nu$ on $M$ and $\og \in \Diff^2_m(M)^\N$, if $x, y \in M$ are both $\og$-regular and $y \in W_\og^s(x)$, then $x \in B(\nu)$ if and only if $y \in B(\nu)$. 
\end{lemma}

\begin{proof}
For any continuous function $\varphi: M \to \R$, and for $\mu^\N$-almost every $\og' \in \Diff^2(M)^\N$, we have
$$ S_{\og'}(\varphi)(x) = S_{\og}(\varphi)(x) = S_{\og}(\varphi)(y) = S_{\og'}(\varphi)(y), $$
where the second equality uses the fact that $y \in W_\og^s(x)$, and $M$ is compact so that $\varphi$ is uniformly continuous. The first and third equalities use the fact that $x$ and $y$ are $\og$-regular. Therefore the leftmost term equals $\int \varphi\;d\nu$ if and only if the rightmost term equals $\int \varphi\;d\nu$.
\end{proof}

\subsubsection{Basic setup of the Hopf argument}

Using Lemma \ref{uniformangle} and \ref{uniformsize}, we can set up the Hopf argument in a small local ball $B(x_0, r)$ containing a regular point $x$ by finding two words $\og, \og' \in \Diff^2(M)^\N$ whose local stable manifolds through $x$ have nice geometric and dynamical properties. Throughout this subsection we shall fix $x_0 \in M$ and $r > 0$. 
We first give an outline of the main argument (see Figure \ref{unifacpic} for an illustration). 
\begin{enumerate}[label = {\bf Step \arabic*:}]
\item By Lemma \ref{ergodicity}, to show that $\nu = m$, it suffices to show that $\vol(B(\nu)) = 1$.
\item By Lemma \ref{goodlemma}, to show that $\vol(B(\nu)) = 1$, it suffices to show that for some uniform constants $R_0 > 0$ and $c > 0$, for all $x_0 \in M$ and $r < R_0$, either $\vol(B(\nu): B(x_0, r)) = 0$ or $\vol(B(\nu): B(x_0, r)) > c$. We will choose $R_0$ in subsection \ref{choiceR}. We fix $x_0 \in M$ and $r < R_0$ in the rest of the outline. 
\item Assume that $\vol(B(\nu): B(x_0, r)) > 0$. Choose a regular point $x$ in $B(\nu) \cap B(x_0, r)$.
\item Choose words $\og, \og'$ and a subset $T \subset W_{\og, \loc}^s(x) \cap B(x_0, r)$ with positive leaf-volume such that for all $y \in T$, 
\begin{enumerate}[label=(\roman*)]
\item $W_{\og, \loc}^s(x)$ and $W_{\og', \loc}^s(y)$ are nice curves (in the sense of Lemma \ref{uniformsize}) and uniformly transverse;
\item $x$ and $y$ are $\og$-regular; 
\item $y$ and almost every $z \in W_{\og', \loc}^s(y)$ are $\og'$-regular.
\end{enumerate}
We will choose $\og, \og'$ and $T$ in subsection \ref{choicew}.
\item Construct a good set $U' \subset B(x_0, r)$ with (uniformly) positive density in $B(x_0, r)$, a word $\og''$, and a subset $T' \subset T$ with positive leaf-volume such that for all $p \in U'$, 
\begin{enumerate}[label=(\roman*)]
\item $W_{\og'', \loc}^s(p)$ is a nice curve, and is uniformly transverse to the family $\{W_{\og', \loc}^s(y)\}_{y \in T'}$. 
\item $p$ is $\og''$-regular, 
\item the set of intersection points between $W_{\og'', \loc}^s(p)$ and $\{W_{\og', \loc}^s(y)\}_{y \in T'}$ that are both $\og'$-regular and $\og''$-regular has positive leaf-volume in $W_{\og'', \loc}^s(p)$.
\end{enumerate}
We will choose $T'$ in subsection \ref{choiceT}. We will choose the word $\og''$, the set $U'$ and the uniform positive lower bound $c_3$ on the density of $U'$ in subsection \ref{choiceU}.
\item Apply Lemma \ref{basin} to show that $U' \subset B(x_0, r)$ is contained in the basin $B(\nu)$. In fact, for $p \in U'$,
\begin{enumerate}[label=(\roman*)]
\item $x \in B(\nu)$ by {\bf Step 3}.
\item Let $y \in T' \subset T \subset W_{\og, \loc}^s(x)$. Both $x$ and $y$ are $\og$-regular, so by (i) and Lemma \ref{basin}, $y \in B(\nu)$.
\item Let $z \in W_{\og', \loc}^s(y)$ for some $y \in T'$. Suppose that $z$ is both $\og'$-regular and $\og''$-regular. By (ii), $y \in B(\nu)$. Since $y$ is $\og'$-regular, by Lemma \ref{basin}, $z \in B(\nu)$. 
\item By {\bf Step 5}, a positive leaf-volume set of points $z$ in $W_{\og'', \loc}^s(p)$ are in $W_{\og', \loc}^s(y)$ for some $y \in T'$, and are $\og'$-regular and $\og''$-regular. By (iii), $z \in B(\nu)$. Since $p$ is $\og''$-regular, by Lemma \ref{basin}, $p \in B(\nu)$.
\end{enumerate}
This concludes the argument, since $U' \subset B(x_0, r) \cap B(\nu)$ and has (uniformly) positive density in $B(x_0, r)$.
\end{enumerate}
In the rest of this section, we shall make {\bf Step 4-6} precise by choosing the appropriate parameters.

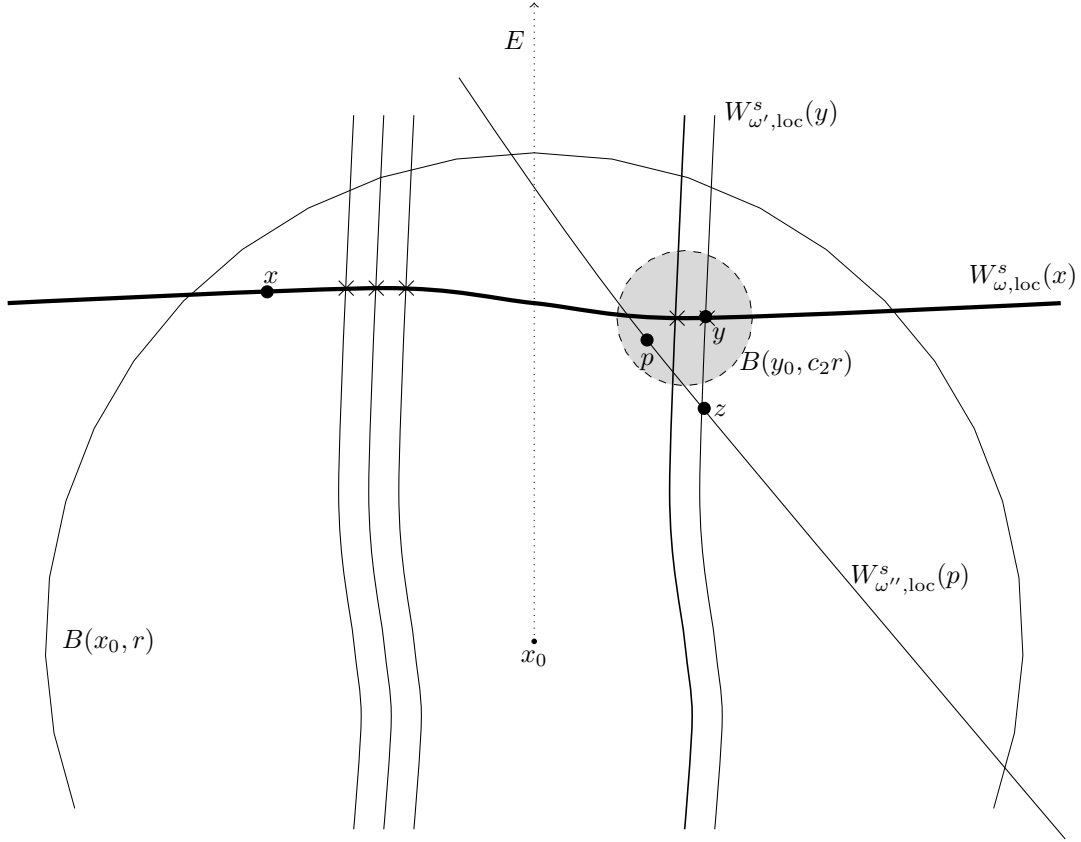
\begin{figure}[ht]
\centering
\begin{tikzpicture}
\draw [ultra thin] [domain=-20:200] plot ({6.5*cos(\x)}, {6.5*sin(\x)});
\node [right] at (-6.4, 0) {$B(x_0, r)$};
\fill[fill = gray!30] (2, 4.3) circle (0.9);
\draw [dashed, ultra thin] [domain=-180:180] plot ({2+0.9*cos(\x)}, {4.3+0.9*sin(\x)});
\node [right] at (2.6, 3.7) {$B(y_0, c_2 r)$};
\draw [ultra thick] plot [smooth] coordinates {(-7, 4.5) (-2, 4.7) (0, 4.5) (2, 4.3) (7, 4.5)}; 
\draw [thin] plot [smooth] coordinates {(-1, 7.5) (1.5, 4.01) (7.06, -2.63)}; 
\draw [thin] plot [smooth] coordinates {(-2.4, -2.5) (-2.3, -1) (-2.4, 0) (-2.6, 2) (-2.4, 7)}; 
\draw [thin] plot [smooth] coordinates {(-2, -2.5) (-1.9, -1) (-2, 0) (-2.2, 2) (-2, 7)}; 
\draw [thin] plot [smooth] coordinates {(-1.6, -2.5) (-1.5, -1) (-1.6, 0) (-1.8, 2) (-1.6, 7)}; 
\draw [semithick] plot [smooth] coordinates {(2, -2.5) (2.1, -1) (2, 0) (1.8, 2) (2, 7)}; 
\draw [thin] plot [smooth] coordinates {(2.4, -2.5) (2.5, -1) (2.4, 0) (2.2, 2) (2.4, 7)}; 
\draw [dotted, ->] (0, 0) -- (0, 8.5);
\node [left] at (0, 8) {$E$};
\node [below] at (0, 0) {$x_0$};
\node [above] at (-3.5, 4.65) {$x$};
\node [below] at (2.46, 4.3) {$y$};
\node [right] at (2.26, 3.08) {$z$};
\node [below] at (1.5, 3.95) {$p$};
\node [above] at (6.5, 4.5) {$W_{\og, \loc}^s(x)$};
\node [above] at (5, 0.5) {$W_{\og'', \loc}^s(p)$};
\node [right] at (2.4, 7) {$W_{\og', \loc}^s(y)$};
\draw [fill] (0, 0) circle [radius=0.03]; 
\draw [fill] (-3.55, 4.65) circle [radius=0.08]; 
\draw [fill] (2.28, 4.32) circle [radius=0.08]; 
\draw [fill] (2.26, 3.10) circle [radius=0.08]; 
\draw [fill] (1.5, 4.01) circle [radius=0.08]; 
\path (-2.5, 4.7) pic[black] {cross=4pt};
\path (-2.1, 4.7) pic[black] {cross=4pt};
\path (-1.7, 4.7) pic[black] {cross=4pt};
\path (2.3, 4.3) pic[black] {cross=4pt};
\path (1.9, 4.3) pic[black] {cross=4pt};
\end{tikzpicture}
\caption{Illustration of the main argument ($U'$ is a subset of the shaded region $B(y_0, c_2 r)$ with positive density; $ \times \in T \subset W_{\og, \loc}^s(x)$)}
\label{unifacpic}
\end{figure}

\subsubsection{Choice of the radius $R_0$} \label{choiceR}

We choose $R_0 = R_0(\alpha, \ell) > 0$ with the following properties: for positive $r < R_0$ and $y \in B(x_0, r)$,
\begin{enumerate}
\item (Angle between off center tangent vectors) 
\begin{itemize}
\item for $v, w \in T_y M$, if $\measuredangle(v, w) > \alpha$, then $\measuredangle(D_y \exp_{x_0}^{-1} v, D_y \exp_{x_0}^{-1} w) > \alpha / 2. $
\item for $v, w \in T_y M$, if $\measuredangle(v, w) > \alpha / 4$, then $\measuredangle(D_y \exp_{x_0}^{-1} v, D_y \exp_{x_0}^{-1} w) > \alpha / 8. $
\end{itemize}
\item (Angle change of curves) given a $C^2$-curve $\g_1 \subset B(x_0, r)$ through $y$, if $\exp_y^{-1} \g_1$ has angle change less than $\alpha/100$, then $\exp_{x_0}^{-1} \g_1$ has angle change less than $\alpha/99$.
\item Also choose $R_0 < \ell / 10$, where $\ell = \ell(\mu)$ is the constant from Lemma \ref{uniformsize}. 
\end{enumerate}
Such conditions hold for small enough $r$ such that for $y \in B(x_0, r)$, the map $D_y \exp_{x_0}^{-1} : T_y M \to TT_{x_0}M$ is close enough to the identity (using the $C^2$ assumption and compactness of the manifold, the appropriate constants depend only on $\alpha$, $\ell$ and the geometry of the smooth Riemannian manifold $M$, in particular $R_0$ can be taken independent of $x_0$). 

\subsubsection{Choice of the words $\og, \og'$, the set $T \subset W_{\og, \loc}^s(x)$ and the constant $c_1$ (for Step 4)}  \label{choicew}
\begin{lemma} \cite[Lem. 4.4.20]{LX} \label{niceset}
For any $x_0 \in M$ and positive $r < R_0$ (from subsection \ref{choiceR}), let $x \in B(x_0, r) \cap B(\nu) \setminus \{x_0\}$ be a regular point. Then there exist words $\og, \og' \in \Diff^2(M)^\N$, a subset $T \subset W_{\og, \loc}^s(x) \cap B(x_0, r)$ and a constant $0 < c_1 = c_1(\alpha) < 1$ with the following properties.
\begin{enumerate}
\item $x$ is $\og$-regular,
\item $W_{\og, \loc}^s(x)$ is a nice curve, i.e. $\og \in \Ld_x$ as in Lemma \ref{uniformsize},
\item the set of $\og'$-regular points has full volume in $M$,
\item the leaf-volume of $T \subset W_\og^s(x)$ is at least $c_1 r$,
\item for any $y \in T$, 
\begin{enumerate}
\item $y$ is $\og$-regular and $\og'$-regular,
\item $W_{\og', \loc}^s(y)$ is a nice curve.
\item $d(y, \pl B(x_0, r)) > c_1 r$,
\item $\measuredangle(E_{\og}^s(y), E_{\og'}^s(y)) > \alpha$,
\end{enumerate}
\end{enumerate}
\end{lemma}

\begin{proof}
We have the following properties of $x$:
\begin{enumerate}[label=(\roman*)]
\item for $\mu^\N$-a.e. $\og$, $x$ is $\og$-regular and almost every $y \in W_\og^s(x)$ is $\og$-regular since $x$ is regular.
\item for at least $99\%$ of the words $\og$ (with respect to $\mu^\N$), $W_{\og, \loc}^s(x)$ is a nice curve by Lemma \ref{uniformsize}.
\end{enumerate}
Note that $x \neq x_0$. Let $v$ be the initial vector in $T_x M$ of the geodesic from $x$ to $x_0$, and $v^\perp \in \P(T_x M)$ be the orthogonal complement of $v$ in $T_x M$.
\begin{enumerate}[label=(\roman*), resume]
\item for at least $99\%$ of the words $\og$ (with respect to $\mu^\N$), $\measuredangle(E_\og^s(x), v^\perp) > \alpha$ by Lemma \ref{uniformangle}. 
\end{enumerate}
{\bf Choice of $\og$:} Let $\og$ be one of the $99\%$ words that satisfy (i), (ii) and (iii). Since $W_{\og, \loc}^s(x)$ is a nice curve, it contains an $\ell$-neighborhood of $x$ with $\ell > 10r$, and we have a uniform bound on angle change of $\exp_x^{-1}(W_{\og, \ell}^s(x))$. \\

\noindent{\bf Choice of $c_1$:} Using Euclidean geometry, (iii) implies that there exist $c_1 = c_1(\alpha) > 0$ and a $C^2$-segment $\g \subset W_{\og, \loc}^s(x)$ such that 
\begin{enumerate}[label=(\roman*), resume]
\item $\vol_{W^s}(\g) > 2c_1 r$, 
\item for all $y \in \g$, $d(y, \pl B(x_0, r)) > c_1 r$. 
\end{enumerate}

Now for the almost every $y \in \g$ that is $\og$-regular, we have the following properties of $y$:
\begin{enumerate}[label=(\roman*$_y$), resume]
\item for $\mu^\N$-a.e. $\og'$, $y$ is $\og'$-regular by Remark \ref{manyregular}.
\item for at least $99\%$ of the words $\og'$ (with respect to $\mu^\N$), $W_{\og', \loc}^s(y)$ is a nice curve by Lemma \ref{uniformsize}.
\item for at least $99\%$ of the words $\og'$ (with respect to $\mu^\N$), $\measuredangle(E_\og^s(y), E_{\og'}^s(y)) > \alpha$ by Lemma \ref{uniformangle}. 
\end{enumerate}
Now consider the set 
$$ G := \{(\og', y) \in \Diff^2(M)^\N \times \g \mid y \text{ is $\og$-regular, and } \og' \text{ satisfies (vi$_y$), (vii$_y$), (viii$_y$)}\}. $$
Almost every $y \in \g$ is $\og$-regular by (i). For each $y \in \g$, at least $98\%$ of the words $\og'$ satisfy (vi$_y$), (vii$_y$), (viii$_y$). Thus by Fubini's theorem, 
$$ \mu^\N \times \vol_{W^s}(G) \geq 0.98 \;\vol_{W^s}(\g). $$
By Fubini's theorem again,
\begin{enumerate}[label=(\roman*), resume]
\item for at least $96 \%$ of the words $\og'$, 
$$ \vol_{W^s}(\{y \in \g \mid y \text{ is $\og$-regular, and } \og' \text{ satisfies (vi$_y$), (vii$_y$), (viii$_y$)}\}) > 0.5 \;\vol_{W^s}(\g) > c_1 r. $$
\end{enumerate}
Here in the last inequality, we have used (iv). Recall that 
\begin{enumerate}[label=(\roman*), resume]
\item for $\mu^\N$-almost every word $\og'$, $ \vol(\{z \in M \mid z \text{ is }\og'\text{-regular}\}) = 1$ by Lemma \ref{wregular}.
\end{enumerate}
{\bf Choice of $\og'$ and $T$:} Let $\og'$ be one of at least $96\%$ words that satisfy (ix) and (x). Let
$$ T := \{y \in \g \mid y \text{ is $\og$-regular, and } \og' \text{ satisfies (vi$_y$), (vii$_y$), (viii$_y$)}\}. $$
We can verify each property:
\begin{enumerate}
\item This follows from (i).
\item This follows from (ii).
\item This follows from (x).
\item This follows from (ix). 
\item for $y \in T$, 
\begin{enumerate}
\item This follows from the definition of $T$ and (vi$_y$).
\item This follows from (vii$_y$).
\item This follows from (v).
\item This follows from (viii$_y$).
\end{enumerate}
\end{enumerate}
\end{proof}

\subsubsection{Choice of the direction $E \in \Bbb{P}(T_{x_0}M)$, the ball $B(y_0, c_2 r)$, and the set $T' \subset T \subset W_{\og, \loc}^s(x)$} \label{choiceT}
Let $r < R_0$. Now lift the ball $B(x_0, r)$ to the tangent space at $x_0$ via the inverse of the exponential map $\exp_{x_0}^{-1}$. Let $x \in B(x_0, r)$. Recall that
\begin{enumerate}
\item since $W_{\og, \loc}^s(x)$ is a nice curve by Lemma \ref{niceset} (2), the angle change of $\exp_x^{-1}W_{\og, \ell}^s(x)$ is less than $\alpha/100$.
\item Also by Lemma \ref{niceset} (5d), for any $y \in T$, $\measuredangle(E_\og^s(y), E_{\og'}^s(y)) > \alpha$.
\end{enumerate}
By the choice of $R_0$, we have
\begin{enumerate}
\item the angle change of $\exp_{x_0}^{-1}W_{\og, \ell}^s(x)$ is less than $\alpha / 99$,
\item for all $y \in T$, $\measuredangle(\exp_{x_0}^{-1} W_{\og}^s(y), \exp_{x_0}^{-1} W_{\og'}^s(y)) > \alpha / 2$. 
\end{enumerate}

\noindent{\bf Choice of $E$:} By compactness of $\Bbb{P}(T_{x_0}M)$ and $\vol_{W^s}(T) > 0$, there exists a direction $E \in \Bbb{P}(T_{x_0}(M))$ such that 
\begin{enumerate}
\item $\vol_{W^s}(\{y \in T \mid \measuredangle(E, \exp_{x_0}^{-1} W_{\og'}^s(y)) < \alpha/100\}) > 0$, and
\item for each $E' \in \Bbb{P}(T_{x_0}(M))$ with $\measuredangle(E, E') < \alpha/100$, and each tangent vector $v$ to the curve $\exp_{x_0}^{-1} W_{\og, \loc}^s(x)$ on $T_{x_0} M$, we have $\measuredangle(E', v) > \alpha / 4$.
\end{enumerate}

\noindent{\bf Choice of $c_2$:} Now take a constant $c_2 = c_2(\alpha, c_1) > 0$ small enough so that $c_2 < c_1/2$ and the following property holds: for any $y_0 \in M$ and $z_1, z_2 \in B(y_0, c_2 r)$, if two $C^1$-curves $\g_1$ and $\g_2$ on $M$ satisfy the following properties:
\begin{enumerate}
\item $z_1 \in \g_1$ and $z_2 \in \g_2$,
\item $\g_i$ contains an $(c_1 r/2)$-neighborhood (within the curve) of $z_i$ for $i = 1, 2$,
\item the angle changes of $\exp_{y_0}^{-1}\g_1$ and $\exp_{y_0}^{-1}\g_2$ are less than $\alpha/99$,
\item $\measuredangle(\exp_{y_0}^{-1}\g_1, \exp_{y_0}^{-1}\g_2) > \alpha / 8$,
\end{enumerate}
then $\g_1$ and $\g_2$ intersect at least once. \\

\noindent{\bf Choice of $y_0$:} Take $y_0 \in T$ such that 
$$ \vol_{W^s}(\{y \in T \cap B(y_0, c_2 r) \mid \measuredangle(E, \exp_{x_0}^{-1} W_{\og'}^s(y)) < \alpha/100\}) > 0. $$

\noindent{\bf Choice of $T'$:} Let $T' := \{y \in T \cap B(y_0, c_2 r) \mid \measuredangle(E, \exp_{x_0}^{-1} W_{\og'}^s(y)) < \alpha/100\}$. Then $\vol_{W^s}(T') > 0$. 

\subsubsection{Choice of the good set $U' \subset B(x_0, r)$, the word $\og''$ and the constant $c_3$ (for Step 5)} \label{choiceU}

\begin{lemma} \label{betterset}
Define $R_0$ as in subsection \ref{choiceR}, the words $\og, \og'$ as in subsection \ref{choicew}, and $E, y_0, c_2, T'$ as in subsection \ref{choiceT}. Let $r < R_0$. Then there exists a uniform constant $c_3 = c_3(c_2) > 0$, a measurable set $U' \subset B(x_0, r)$ with $\vol(U') > c_3 \vol(B(x_0, r))$ and a word $\og''$ such that for all $p \in U'$, 
\begin{enumerate}[label=(\alph*)]
\item $W_{\og'', \loc}^s(p)$ is a nice curve.
\item almost every $z \in W_{\og''}^s(p)$ is $\og'$-regular, where $\og'$ is the chosen word in subsection \ref{choicew}.
\item $p$ is $\og''$-regular and almost every point in $W^s_{\og''}(p)$ is $\og''$-regular.
\item The angle
$$ \measuredangle(D_p \exp_{x_0}^{-1} E_{\og''}^s(p), E) > \alpha / 2, $$
where $E \in \Bbb{P}(T_{x_0}M)$ is the direction chosen in subsection \ref{choiceT}.
\end{enumerate}

\end{lemma}

\begin{proof}
We first collect a few facts that hold for $\vol$-almost every points $p \in M$ and a large set of words $\og''$.
\begin{enumerate}[label=(\alph*)]
\item By Lemma \ref{uniformsize}, for any $p \in M$, for at least $99\%$ of the words $\og''$, $W_{\og'', \loc}^s(p)$ is a nice curve.
\item By (AC2) and Lemma \ref{niceset}(3), for $\vol$-almost every $p \in M$ and $\mu^\N$-a.e. $\og''$, almost every $z \in W_{\og''}^s(p)$ is $\og'$-regular.
\item By Lemma \ref{almostregular}, $\vol$-almost every point $p \in M$ is regular, i.e. for $\mu^\N$-a.e. $\og''$, $p$ is $\og''$-regular and almost every point in $W^s_{\og''}(p)$ is $\og''$-regular.
\item By Lemma \ref{uniformangle} and the choice of $R_0$, for any $p \in B(x_0, r)$, for at least $99\%$ of the words $\og''$,
$$ \measuredangle(D_p \exp_{x_0}^{-1} E_{\og''}^s(p), E) > \alpha / 2, $$
where $E \in \Bbb{P}(T_{x_0}M)$ is the direction in subsection \ref{choiceT}.
\end{enumerate}
Hence for $\vol$-a.e. $p \in B(x_0, r)$, there are at least $98\%$ of the words $\og''$ such that (a)-(d) hold. Now consider the small ball $B(y_0, c_2 r)$ chosen in subsection \ref{choiceT}. Since $c_2 < c_1$ and $d(y_0, \pl B(x_0, r)) > c_1 r$ (since $y_0 \in T$), $B(y_0, c_2 r) \subset B(x_0, r)$. \\

\noindent{\bf Choice of $U'$ and $\og''$:} By Fubini's theorem, there exists a word $\og''$ such that the subset 
$$ U' := \{p \in B(y_0, c_2 r) \mid \text{ (a)-(d) hold for } p \text{ with respect to } \og''\} \subset B(x_0, r) $$
has volume $\vol(U') > 0.5 \;\vol(B(y_0, c_2 r))$, where $B(y_0, c_2 r)$ is the ball from subsection \ref{choiceT}. \\

\noindent{\bf Choice of $c_3$:} Now we can take a uniform constant $c_3 = c_3(c_2) > 0$ such that $\vol(U') > c_3 \vol(B(x_0, r))$. 
\end{proof}

The set $U'$ and the word $\og''$ are related to the $\og'$-local stable curves through $T'$ in the following manner.

\begin{lemma} \label{intersection}
Define $T' \subset W_{\og, \loc}^s(x)$ as in subsection \ref{choiceT}. Let $U := \bigcup_{y \in T'} W_{\og', \loc}^s(y)$. Then for all $p \in U'$, 
$$ \vol_{W^s}(\{z \in W_{\og'', \loc}^s(p) \cap U \mid z \text{ is } \og'\text{-regular and } \og''\text{-regular}\}) > 0. $$
\end{lemma}

\begin{proof}
Let $p \in U'$ and $y \in T'$. Note that $p, y \in B(y_0, c_2 r)$. Let $z_1 = p$ and $z_2 = y$. We verify properties 1-4 in the choice of $c_2$ in subsection \ref{choiceT} for the local stable curves 
$$ \g_1 := W_{\og'', c_1 r / 2}^s(p) \subset W_{\og'', \loc}^s(p) \quaddd \text{ and } \quaddd \g_2 := W_{\og', c_1 r / 2}^s(y) \subset W_{\og', \loc}^s(y). $$
Note that since $c_2 < c_1 / 2$,  $d(y_0, \pl B(x_0, r)) > c_1 r$ (since $y_0 \in T$) and $p, y \in B(y_0, c_2 r)$, we have $\g_1, \g_2 \subset B(x_0, r)$.
\begin{enumerate}
\item Clearly $p \in \g_1$ and $y \in \g_2$.
\item By definition of $\g_1$ and $\g_2$, $\g_i$ is the $(c_1 r / 2)$-neighborhood of $z_i$ in the local stable curve.
\item Note that $W_{\og'', \loc}^s(p)$ and $W_{\og', \loc}^s(y)$ are nice curves by Lemma \ref{betterset}(a) and Lemma \ref{niceset} (5b), so $\g_1$ and $\g_2$ have bounded angle change in their respective tangent spaces. Now using the choice of $R_0$ applied to the tangent space at $y_0$, we conclude the bound on angle changes in $T_{y_0} M$. 
\item By the choice of $R_0$, $E$ and $T'$, one can readily verify that $\measuredangle(\exp_{x_0}^{-1} \g_1, \exp_{x_0}^{-1}\g_2) > \alpha / 4$. Apply the choice of $R_0$ again, $\measuredangle(\exp_{y_0}^{-1} \g_1, \exp_{y_0}^{-1}\g_2) > \alpha / 8$. 
\end{enumerate}
Therefore properties 1-4 in the choice of $c_2$ are satisfied, thus $W_{\og'', \loc}^s(p)$ intersects $W_{\og', \loc}^s(y)$ for all $y \in T'$, with angle at least $\alpha / 4$ on $T_{x_0} M$. 

Now $W_{\og, \loc}^s(x)$ is uniformly transverse to $W_{\og', \loc}^s(y)$ for $y \in T'$ by Lemma \ref{niceset} (5d). Apply (AC1) to the holonomy $h_{W^s}$ between the transversals $W_{\og'', \loc}^s(p)$ and $W_{\og, \loc}^s(x)$ along the family of local stable curves $\{W_{\og', \loc}^s(y)\}_{y \in T'}$. By the previous paragraph, $h_{W^s}$ is a bijection from $W_{\og'', \loc}^s(p) \cap U$ to $T' \subset W_{\og, \loc}^s(x)$. Since $T'$ has positive leaf-volume in $W_{\og, \loc}^s(x)$, by (AC1), $W_{\og'', \loc}^s(p) \cap U$ has positive leaf-volume. 

\noindent Now the conclusion holds since almost every point in $W_{\og'', \loc}^s(p)$ is $\og'$-regular and $\og''$-regular by Lemma \ref{betterset} (b, c).
\end{proof}

\subsubsection{Conclude the proof of Proposition \ref{allofM} and Proposition \ref{measurerigidity}}

\begin{proof}[Proof of Proposition \ref{allofM}]
The proof goes by performing the Hopf argument in a local ball $B(x_0, r)$ with $r < R_0$, combining the pieces built in previous sections.
\begin{enumerate}[label = {\bf Step \arabic*:}]
\item {\bf It suffices to show that the basin $B(\nu)$ has full volume.} \\
By Lemma \ref{ergodicity}, to show that $\nu = m$, it suffices to show that $\vol(B(\nu)) = 1$. 
\item {\bf It suffices to show that the basin has nontrivial density in each small local ball $\mk{B} = B(x_0, r)$.} \\
Note that $\vol(B(\nu)) \geq \vol(A) > 0$ since a full volume subset of $A$ is in the basin $B(\nu)$ by the pointwise ergodic theorem and that $\nu = \frac{1}{m(A)}m|_A$. By Lemma \ref{goodlemma}, to show that $B(\nu)$ has full volume, it suffices to show that there exist $c > 0$ and $R_0 > 0$ such that for all $x_0 \in M$ and positive $r < R_0$ that satisfy $\mathrm{vol}(B(x_0, r) \cap B(\nu)) > 0$, we have
$$ \mathrm{vol}(B(x_0, r) \cap B(\nu)) > c \;\mathrm{vol}(B(x_0, r)). $$
We choose $R_0$ as in subsection \ref{choiceR}, and will choose $c = c_3$ from subsection \ref{choiceU} in {\bf Step 6}. In particular $R_0 < \ell /10$.

In the rest of the proof we fix $x_0 \in M$ and $r \in (0, R_0)$. Let $\mk{B} := B(x_0, r)$. 
\item {\bf Choose a regular point $x$ in the local ball $\mk{B}$.} \\
By Lemma \ref{almostregular}, the set of regular points in $M$ has full volume. Thus for fixed $x_0 \in M$ and $r < R_0$ with $\mathrm{vol}(B(x_0, r) \cap B(\nu)) > 0$, one can choose a regular point $x \in B(x_0, r) \cap B(\nu) \setminus \{x_0\}$.
\item {\bf Choose two words $\og, \og'$ with transverse local stable manifolds in $\mk{B}$.} \\
Choose words $\og, \og' \in \Diff^2(M)^\N$ as in subsection \ref{choicew} and a subset $T' \subset W_{\og, \loc}^s(x)$ as in subsection \ref{choiceT}. \\
Let
$$ U := \bigcup_{y \in T'} W_{\og', \loc}^s(y). $$
\item {\bf Choose a good set $U'$ with positive density in $\mk{B}$, a word $\og''$ and a subset $T' \subset T$ with positive leaf-volume}. \\
We choose the good set $U' \subset \mk{B}$, the word $\og''$ and the subset $T' \subset T$ as in subsection \ref{choiceU}.
\item {\bf The good set $U'$ is contained in the basin $B(\nu)$}. \\
Let $p \in U'$. Now we claim that $p \in B(\nu)$. In fact
\begin{enumerate}[label=(\roman*)]
\item $x \in B(\nu)$ by the choice in {\bf Step 3}.
\item For all $y \in T'$, note that $T' \subset W_\og^s(x)$ and $x, y$ are $\og$-regular by Lemma \ref{niceset} (1, 5a). Therefore by Lemma \ref{basin}, $y \in B(\nu)$. 
\item Suppose $z \in W_{\og'', \loc}^s(p) \cap U$ is $\og'$-regular. By the definition of $U$, there exists $y \in T'$ such that $z \in W_{\og', \loc}^s(y)$. Recall that $y \in T'$ is $\og'$-regular from Lemma \ref{niceset} (5a). Therefore by Lemma \ref{basin}, $z \in B(\nu)$. 
\item By Lemma \ref{intersection}, the set of points in $W_{\og'', \loc}^s(p) \cap U$ that are $\og'$-regular and $\og''$-regular has positive leaf-volume. Let $z$ be one such point. Note that $p \in W_{\og''}^s(z)$, and $p$ is $\og''$-regular by Lemma \ref{betterset}(c). Therefore by Lemma \ref{basin}, $p \in B(\nu)$. 
\end{enumerate}
Therefore $U' \subset \mk{B} \cap B(\nu)$, hence
$$ \vol(\mk{B} \cap B(\nu)) \geq \vol(U') > c_3 \;\vol(\mk{B}) $$
by Lemma \ref{betterset}, as desired.
\end{enumerate}
\end{proof}

\begin{proof}[Proof of Proposition \ref{measurerigidity}] 
Since $\mu$ is uniformly expanding, by Proposition \ref{positiveexponent}, any ergodic $\mu$-stationary measure $\nu$ has positive Lyapunov exponent. Hence in the case of volume-preserving diffeomorphisms on surfaces, it is hyperbolic. Now by \cite[Thm. 3.4]{BR}, either $\nu$ is finitely supported, the stable distribution is non-random, or $\nu$ is the restriction of $m$ to a positive volume subset. By Lemma \ref{nonrandom}, the second possibility is eliminated. In the third possibility, by Proposition \ref{allofM}, we have $\nu = m$. The result follows. 
\end{proof}


\subsection{Comparison with Brown-Rodriguez Hertz}

The following proposition may be viewed as a motivation for the assumption of uniform expansion, in view of the theorem \cite[Thm. 3.4]{BR}.

\begin{proposition} \label{notue}
Let $M$ be a closed surface, $\mu$ be a Borel probability measure on $\Diff^2(M)$. If $\mu$ is not uniformly expanding, then there exists an ergodic $\mu$-stationary measure $\nu$ on $M$ and a $\mu$-almost surely invariant $\nu$-measurable subbundle $\hat{V} \subset TM$ in which the top Lyapunov exponent is nonpositive. 

In particular, if $\mu$ is supported on $\Diff^2_m(M)$ for some smooth measure $m$ on $M$, then $\mu$ is uniformly expanding if and only if every ergodic $\mu$-stationary measure $\nu$ on $M$ has a positive Lyapunov exponent and the stable distribution is not non-random with respect to $\nu$. 
\end{proposition}


To prove this proposition, we first note that each map $f \in \Diff^2(M)$ induces the projective action on the unit tangent bundle $T^1 M$ by
$$ f \cdot (x, v) = \left( f(x), \frac{D_x f(v)}{\| D_x f(v) \|} \right). $$
From now on we shall abuse the notation and write $f(x, v) := f \cdot (x, v)$. 

In the case that $\mu$ is uniformly expanding, we first construct an ergodic stationary measure on $T^1 M$ which does not exhibit exponential growth on average.
\begin{lemma} \label{badmeasure}
If $\mu$ is not uniformly expanding, then there exists an ergodic $\mu$-stationary measure $\onu'$ on $T^1 M$ such that
$$ \iint \log \| D_x f(v) \| d\mu(f) d\onu'(x, v) \leq 0. $$
\end{lemma}
\begin{proof}
Fix $\ep > 0$. Since $\mu$ is not uniformly expanding, for all positive integer $N$, there exists $(x_N, v_N) \in T^1 M$ such that
\begin{equation} \label{contra}
\int \log \| D_{x_N} f(v_N) \| d\mu^{(N)}(f) < \ep. 
\end{equation}
Let 
$$ \nu_N := \frac{1}{N} \sum_{n = 0}^{N-1} \int \delta_{f(x_N, v_N)} d \mu^{(n)}(f), $$
and let $\onu$ be any weak-* limit point of $\{\nu_N\}$. Note that $\onu$ is a $\mu$-stationary measure on $T^1 M$ since
$$ \mu * \nu_N =  \frac{1}{N} \sum_{n = 0}^{N-1} \int \mu * \delta_{f(x_N, v_N)} d \mu^{(n)}(f) = \frac{1}{N} \sum_{n = 0}^{N-1} \int \delta_{f(x_N, v_N)} d \mu^{(n+1)}(f) $$
and hence as $N \to \infty$, 
$$ \mu * \nu_N - \nu_N = \frac{1}{N} \left( \int \delta_{f(x_N, v_N)} d \mu^{(N)}(f) - \delta_{(x_N, v_N)} \right) \to 0. $$

For $f \in \Diff^2(M)$ and $(x, v) \in T^1 M$, let
$$ \Phi(f, (x, v)) := \log \| D_x f(v) \|. $$
Note that for each $N \in \N$ and $\og =(\og_0, \og_1, \og_2, \ldots) \in \Diff^2(M)^\N$, 
\begin{equation} \label{chainrule}
\log \| D_x f_\og^N (v) \| = \sum_{n = 0}^{N-1} \Phi(\og_n, f_\og^n(x, v)). 
\end{equation}
Since the first argument of $\Phi(\og_n, f_\og^n(x, v))$ depends only on the $(n+1)$-th coordinate of $\og$, and the second argument depends only on the first $n$ coordinates of $\og$, we have
$$ \int \log \| D_{x} f_\og^N (v) \| d\mu^{\N}(\og) = \sum_{n = 0}^{N-1} \int \Phi(\og_n, f_\og^n(x, v)) d\mu^\N(\og) = \sum_{n = 0}^{N-1} \int \Phi(g, f(x, v))d\mu(g) d\mu^{(n)}(f). $$
On the other hand, the left hand side is $\ds\int \log \| D_x f(v) \| d\mu^{(N)}(f)$. Therefore if we set $(x, v) = (x_N, v_N)$, by the definition of $\nu_N$ and (\ref{contra}), for all $N \in \N$, 
$$ \int \int \Phi(g, (x, v)) \; d\mu(g) \; d\nu_N(x, v) < \frac{\ep}{N}. $$
By continuity of $\Phi$ and weak-* convergence, we have upon taking limit
$$ \int \int \Phi \; d\mu \;d \onu \leq 0. $$
Let $\onu'$ be an ergodic component of $\onu$ such that 
$$ \int \int \Phi \; d\mu \; d\onu' \leq 0, $$
which exists since $\onu$ is a convex combination of its ergodic components. This measure $\onu'$ satisfies the desired properties.
\end{proof}

\begin{proof}[Proof of Proposition \ref{notue}]

Assume that $\mu$ is not uniformly expanding. Consider the measure $\onu'$ given by Lemma \ref{badmeasure}. Let $\nu := \pi_* \onu'$, where $\pi: T^1 M \to M$ is the natural projection. Then note that $\nu$ is an ergodic $\mu$-stationary measure on $M$ since $\pi$ is equivariant with respect to the action by $\Diff^2(M)$. Let $\{\onu'_x\}$ be a family of conditional measures of $\onu'$ along the partition of $T^1 M$ into fibers over $M$. 

Let $F$ be the skew product map on $\Diff^2(M)^\N \times T^1 M$ defined by $F(\og, x) = (\sigma(\og), \og_0(x))$. Recall that $\onu'$ is an ergodic $\mu$-stationary measure on $T^1 M$ if and only if $\mu^\N \times \onu'$ is an ergodic $F$-invariant measure on $\Diff^2(M)^\N \times T^1 M$ (\cite[Lem. I.2.3, Thm. I.2.1]{Kif}). Consider the following map on $\Diff^2(M)^\N \times T^1 M$,
$$ \Psi(\og, (x, v)) := \log \| D_x \og_0 (v) \|. $$
By the pointwise ergodic theorem, for $\nu$-a.e. $x \in M$ and $\onu'_x$-a.e. $v \in T^1_x M$, 
\begin{equation} \label{birkhoff}
\lim_{N \to \infty} \frac{1}{N} \sum_{n=0}^{N-1} \Psi(\sigma^n(\og), f_\og^n (x, v)) = \int \int \Psi \; d\mu^\N \;d\onu'  \quaddd \text{ for } \mu^\N \text{-a.e. }\og.
\end{equation}
Note that since $\Psi$ depends only on the first coordinate of $\og$, by Lemma \ref{badmeasure},
\begin{equation} \label{nonpositive}
 \int \int \Psi \; d\mu^\N \;d\onu' = \int \int \log \| D_x f(v) \| d\mu(f) \;d\onu'(x, v) \leq 0. 
\end{equation}

Now the support of $\onu'$ spans a $\mu$-a.s. invariant $\nu$-measurable subbundle $\hat{V} \subset TM$ (not necessarily proper). Apply (\ref{birkhoff}) again, we have that the top Lyapunov exponent in $\hat{V}$ is nonpositive. 

Finally, to show the second assertion, assume that $\mu$ is supported on $\Diff^2_m(M)$ for some smooth measure $m$ on $M$ and $\mu$ is not uniformly expanding. 

In the volume preserving case, for each ergodic $\mu$-stationary measure $\nu$, either all exponents are zero for $\nu$-a.e. $x$, or there is a positive and a negative exponent for $\nu$-a.e. $x$. If all the Lyapunov exponents of $\nu$ are zero, we are done. Hence we may assume that $\nu$ has a positive exponent. By Oseledets theorem, for $\mu^\N \times \nu$-a.e. $(\og, x) \in \Diff^2(M)^\N \times M$, the tangent vectors in $T_x M$ outside of $E_\og^s(x)$ have exponential growth. Since vectors in $\hat{V}(x)$ have nonpositive top exponent, $\hat{V}(x) \subset E_\og^s(x)$ for $\nu$-a.e. $x \in M$. Since $E_\og^s(x)$ is one-dimensional, we have $\hat{V}(x) = E_\og^s(x)$. Since $\hat{V}$ is $\mu$-a.s. invariant, we have that the stable distribution $E_\og^s(x)$ is non-random. This shows the ``if'' direction. The ``only if'' direction follows from Proposition \ref{positiveexponent} and Lemma \ref{nonrandom}.

\end{proof}

%

\section{Equidistribution and Orbit closures} \label{oclosures}

We now prove an equidistribution statement from the measure rigidity result using the existence of a Margulis function, which follows from uniform expansion. We follow the strategy in \cite{EMM}, the idea of which goes back to \cite{EMar} and \cite{EMaMo}. The orbit closure classification then follows. The assumptions we make in this section are slightly weaker than Theorem D, though Theorem D suffices for the applications in the subsequent sections.

\begin{proposition}[Equidistribution] \label{equidistribution}
Let $M$ be a closed surface, $\G \subset \Diff^2(M)$ be a subsemigroup that preserves a smooth measure $m$ on $M$. Let $\mu$ be a uniformly expanding probability measure on $\Diff^2_m(M)$ with $\mu(\G) = 1$ satisfying
\begin{equation} 
\tag{**}
 \int_{\Diff^2(M)} |f|_{C^2}^{\delta} + |f^{-1}|_{C^2}^{\delta}\; d\mu(f) < \infty \quaddd \text{ for all sufficiently small } \delta > 0. 
\end{equation}
Suppose $x \in M$ has infinite $\G$-orbit. Then for any continuous function $\varphi \in C(M)$, 
$$ \lim_{n \to \infty} \frac{1}{n} \sum_{k=1}^n \int_{\Diff^2(M)} \varphi(f(x)) \; d \mu^{(k)}(f) = \int_M \varphi \; dm. $$
\end{proposition}

\noindent Clearly if $\mu$ is finitely supported, then (**) is satisfied. Also assumption (**) is stronger than (*). 

\begin{proposition}[Orbit Closures] \label{orbitclosure}
Let $M$ be a closed surface, $\G \subset \Diff^2(M)$ be a subsemigroup that preserves a smooth measure $m$ on $M$. Let $\mu$ be a uniformly expanding probability measure on $\Diff^2_m(M)$ with $\mu(\G) = 1$ satisfying (**). Then every orbit of $\G$ is either finite or dense. 
\end{proposition}

\nt The following lemma shows that if $\mu$ is uniformly expanding, then there exists a so-called Margulis function. 
\begin{lemma} \label{margulis}
Suppose $\mu$ is a uniformly expanding measure. Then there exists a proper continuous function $u: M \times M \setminus \Delta \to \R_+$, $c < 1$, $b > 0$ and a positive integer $n_0$ such that for all $(x, y) \in M \times M \setminus \Delta$,
$$ \int u(f(x), f(y)) d\mu^{(n_0)}(f) \leq c u(x, y) + b. $$
\end{lemma}

\begin{proof}
The proof is similar to Lemma 10.8 of \cite{V}. We can take 
$$ u(x, y) := d(x, y)^{-\delta}, $$
where $\delta \in (0, 1)$ is a small number to be determined. Fix $x \in M$ and $v \in T_x M$. Consider the function
$$ \phi_n(\delta) := \int_{\Diff^2(M)} \left(\frac{\| D_x f(v) \|}{\| v\|} \right)^{-\delta} d\mu^{(n)}(f). $$
This is a differentiable function in $\delta$, with
$$ \phi_n'(\delta) = -\int_{\Diff^2(M)} \left(\frac{\| D_x f(v) \|}{\| v\|} \right)^{-\delta} \log \left(\frac{\| D_x f(v) \|}{\| v\|} \right) d\mu^{(n)}(f). $$
By uniform expansion, there exists $C > 0$ and $N \in \N$ (independent of $x$ and $v$) such that 
$$ \phi_{N}'(0) = -\int_{\Diff^2(M)} \log \left(\frac{\| D_x f(v) \|}{\| v\|} \right) d\mu^{(N)}(f) < -C. $$
Since $\phi_N(0) = 1$, for small enough $\delta > 0$ (can be chosen independent of $x$ and $v$ using the compactness of $M$ and $T^1 M$), we have 
$$ \phi_{N}(\delta) = \int_{\Diff^2(M)} \left(\frac{\| D_x f(v) \|}{\| v\|} \right)^{-\delta} d\mu^{(N)}(f) < 1 - \frac{C\delta}{2}. $$ 
Take such a $\delta$ in the definition of $u$, and let $n_0 = N$. Then we have
$$ \int_{\Diff^2(M)} \left(\frac{\| D_x f(v) \|}{\| v\|} \right)^{-\delta} d\mu^{(n_0)}(f) < 1 - \frac{C \delta}{2}. $$
Let $c = 1 - C\delta/4$. Take $\ep > 0$ small enough such that for all $x, y \in M \times M \setminus \Delta$ with $d(x, y) < \ep$, 
$$ \int \frac{d(f(x), f(y))^{-\delta}}{d(x, y)^{-\delta}} d \mu^{(n_0)}(f) < 1 - \frac{C\delta}{4} = c. $$
For $0 < d(x, y) < \ep$, 
$$ \int u(f(x), f(y)) d\mu^{(n_0)}(f) = \int d(f(x), f(y))^{-\delta} d\mu^{(n_0)}(f) < c d(x, y)^{-\delta} = c u(x, y). $$
Now using the moment condition (**) (take a smaller $\delta > 0$ if necessary), we can take some $b > 0$ so that for all $x, y \in M$ with $d(x, y) \geq \ep$, 
$$ \int d(f(x), f(y))^{-\delta} d \mu^{(n_0)}(f) \leq b. $$
Hence for all $(x, y) \in M \times M \setminus \Delta$, 
$$ \int u(f(x), f(y)) d\mu^{(n_0)}(f) \leq c u(x, y) + b. $$
\end{proof}


\begin{corollary} \label{orbitmargulis}
Suppose $\mu$ is a uniformly expanding measure and $\ml{N} \subset M$ is a finite $\G$-orbit. Then there exists a proper continuous function $f_\ml{N}: M \setminus \ml{N} \to \R_+$, $c < 1$, $b > 0$ and a positive integer $n_0$ such that for all $x \in M \setminus \ml{N}$,
$$ \int f_\ml{N}(f(x)) d\mu^{(n_0)}(f) \leq c f_\ml{N}(x) + b. $$
Here $c$ and $b$ depend only on the size of $\ml{N}$. Moreover, for each $x \in M \setminus \ml{N}$, there exists a positive integer $n(x)$ such that for all $n > n(x)$, 
$$ (\mu^{(n)} * \delta_x)(f_\ml{N}) = \int f_\ml{N}(f(x)) d\mu^{(n)}(f) \leq b_1, $$
where $b_1 = b_1(b, c)$. For each compact subset $F \subset M \setminus \ml{N}$, we can take $n(x)$ such that $\sup_{x \in F} n(x) < \infty$. 
\end{corollary}

\begin{proof}
Let $u: M \times M \setminus \Delta \to \R_+$ be the function as in Lemma \ref{margulis} with the corresponding $c < 1$ and $b > 0$, and define the function $f_\ml{N}: M \setminus \ml{N} \to \R$ by 
$$ f_\ml{N}(x) := \frac{1}{|\ml{N}|} \sum_{y \in \ml{N}} u(x, y). $$
\noindent Take the positive integer $n_0$ as in Lemma \ref{margulis}. Then for all $x \in M \setminus \ml{N}$, 
\begin{equation*} 
 \int f_\ml{N}(f(x)) d \mu^{(n_0)}(f) = \frac{1}{|\ml{N}|} \int \sum_{y \in \ml{N}} u(f(x), y) d \mu^{(n_0)}(f) = \frac{1}{|\ml{N}|} \int \sum_{y \in \ml{N}} u(f(x), f(y)) d \mu^{(n_0)}(f) \leq c f_\ml{N}(x) + b.
\end{equation*}
Here we used that $\ml{N}$ is $\G$-invariant in the second equality. This gives the first assertion. 

For the second assertion, from the above, for all positive integer $k$ and $x \in M \setminus \ml{N}$,
\begin{equation*} 
 (\mu^{(kn_0)} * \delta_x) (f_\ml{N}) = \int f_\ml{N}(f(x)) d \mu^{(kn_0)}(f) \leq c^k f_\ml{N}(x) + \frac{b}{1 - c}. 
\end{equation*}
Therefore for all $n \geq 0$, 
\begin{equation*} 
 (\mu^{(n)} * \delta_x) (f_\ml{N}) = \int f_\ml{N}(f(x)) d \mu^{(n)}(f) \leq c^{\lfloor n/n_0 \rfloor} \mu^{(i)} * \delta_x (f_\ml{N}) + \frac{b}{1 - c},
\end{equation*}
where $i := n - n_0 \lfloor n/n_0 \rfloor < n_0$.
Now for any compact $F \subset M \setminus \ml{N}$, there exists some positive integer $m_F$ such that for all $n > m_F$,
$$ c^{\lfloor n/n_0 \rfloor} \mu^{(i)} * \delta_x (f_\ml{N}) < \frac{b}{1 - c} \quaddd \text{ for all } \quad 0 \leq i \leq n_0, \quad x \in F. $$
Then for any $n > m_F$ and $x \in F$,
$$ (\mu^{(n)} * \delta_x)(f_\ml{N}) \leq \frac{2b}{1 - c} =: b_1. $$
\end{proof}

\begin{corollary} \label{upperbound}
Suppose $\mu$ is a uniformly expanding measure and $\ml{N} \subset M$ is a finite $\G$-orbit. Take $f_\ml{N}, c, b$ as in Corollary \ref{orbitmargulis}. Suppose $\nu$ is an ergodic $\mu$-stationary measure on $M$ with $\nu(\{f_{\ml{N}} < \infty\}) > 0$. Then 
$$ \int f_\ml{N}(x) d \nu(x) \leq B, $$
where $B$ depends only on $b, c$. 
\end{corollary}

\begin{proof}
For each positive integer $n$, let $f_{\ml{N}, n} := \min\{f_\ml{N}, n\}$. By the Birkhoff ergodic theorem, for $\mu^\N \times \nu$-a.e. $(\og, x) \in \G^\N \times M$, 
$$ \lim_{m \to \infty} \frac{1}{m} \sum_{k=1}^m f_{\ml{N}, n}(f_\og^k(x)) = \int f_{\ml{N}, n}(x) d\nu(x), $$
where for $\og = (\og_0, \og_1, \ldots) \in \G^\N$, $f_\og^k := \og_{k-1} \circ \og_{k-2} \circ \cdots \circ \og_0$. Pick a point $x_0 \in M \setminus \ml{N}$ such that the convergence holds for $\mu^{\N}$-a.e. $\og \in \G^\N$ (note that we can pick such $x_0 \notin \ml{N}$ since $\nu(\{f_{\ml{N}} < \infty\}) = \nu(M \setminus \ml{N}) > 0$). By Egorov's theorem, we can take a subset $\G' \subset \G^\N$ with $\mu^{\N}(\G') \geq 1/2$ such that at $x = x_0$, the convergence is uniform on $\og \in \G'$. Then there exists a positive integer $m_n$ such that for all $m > m_n$ and $\og \in \G'$, 
$$ \frac{1}{m} \sum_{k=1}^m f_{\ml{N}, n}(f_\og^k(x_0)) \geq \frac{1}{2} \int f_{\ml{N}, n}(x) d\nu(x). $$
Integrating over $\og \in \G^\N$, we have for all $m > m_n$, 
$$ \frac{1}{m} \sum_{k=1}^m \int f_{\ml{N}, n}(f(x_0)) d\mu^{(k)}(f) \geq \frac{1}{4} \int f_{\ml{N}, n}(x) d\nu(x). $$
By Corollary \ref{orbitmargulis}, for large enough $m$, the left hand side is at most some constant $B' = B'(b, c)$. Therefore for all $n$, 
$$  \int f_{\ml{N}, n}(x) d\nu(x) \leq 4B'. $$
Taking the limit $n \to \infty$, we have the assertion. 
\end{proof}

\begin{proposition} \label{countable}
The number of points with finite $\G$-orbit is countable. 
\end{proposition}

\begin{proof}
It suffices to show that for each positive integer $n$, there are finitely many $\G$-orbits of size $n$. Suppose the contrary that there are infinitely many $\G$-orbits of size $n$. Then by compactness of $M$, they have an accumulation point $x \in M$, hence there exists a sequence of points $x_i \in M$ with finite $\G$-orbit of size $n$ such that $d(x_i, x_{i+1}) \to 0$ as $i \to \infty$. Fix an $\ep = \ep(B, n, \delta)> 0$ (to be determined later), and a large enough $j$ such that $d(x_j, x_{j+1}) < \ep$. By deleting finitely many points from the sequence if necessary, we may assume $x_j$ and $x_{j+1}$ are in different $\G$-orbits. For each $i \in \N$, let $\nu_i$ be the ergodic $\G$-invariant (hence $\mu$-stationary) measure on $M$ supported on the $\G$-orbit $\ml{N}_i$ of $x_i$ with uniform distribution, i.e. $\nu_i(x) = 1/n$ for each $x \in \ml{N}_i$, and let $f_i := f_{\ml{N}_i}$ be the function defined in Corollary \ref{orbitmargulis} with the corresponding upper bound $B = B(b, c)$ as in Corollary \ref{upperbound}. As $x_{j+1} \notin \ml{N}_j$, $f_j(x_{j+1}) < \infty$. Hence $\nu_{j+1}(f_j < \infty) \geq 1/n > 0$. Therefore by Corollary \ref{upperbound}, 
\begin{equation}
\tag{***}
 \int f_j(x) d\nu_{j+1}(x) \leq B. 
 \end{equation}
On the other hand, recall from definition that $f_j(x) = \frac{1}{|\ml{N}_j|}\sum_{y \in \ml{N}_j} u(x, y)$ where $u(x, y) = d(x, y)^{-\delta}$ for some $\delta > 0$ chosen in the proof of Lemma \ref{margulis}. Thus
$$ \int f_j(x) d\nu_{j+1}(x) = \frac{1}{n^2} \sum_{x \in \ml{N}_{j+1}} \sum_{y \in \ml{N}_j} u(x, y) \geq \frac{1}{n^2} u(x_{j+1}, x_j) > \frac{1}{n^2} \ep^{-\delta}. $$
Taking $\ep$ small enough such that $\ep^{-\delta} \geq 2Bn^2$, this leads to a contradiction to (***). 
\end{proof}

Define
$$ \overline{\mu}^{(n)} := \frac{1}{n} \sum_{k=1}^{n} \mu^{(k)} . $$

\begin{lemma} \label{small}
Let $\ml{N}$ be a finite $\G$-orbit in $M$. The for any $\ep > 0$, there exists an open set $\O_{\ml{N}, \ep}$ containing $\ml{N}$ with $(\O_{\ml{N}, \ep})^c$ compact such that for any compact $F \subset M \setminus \ml{N}$ there exists a positive integer $n_F$, such that for all $x \in F$ and $n > n_F$, we have
$$ (\overline{\mu}^{(n)} * \delta_x)(\O_{\ml{N}, \ep}) < \ep. $$
\end{lemma}

\begin{proof}
The proof follows that of Proposition 3.3 in \cite{EMM}. Take the function $f_\ml{N}: M \setminus \ml{N} \to \R_+$ as in Corollary \ref{orbitmargulis} with the corresponding $c < 1$, $b > 0$ and positive integer $n_0$. Let 
$$ \O_{\ml{N}, \ep} := \left\{x \in M: f_\ml{N}(x) > \frac{1}{\ep} \left(\frac{2b}{1 - c}+1 \right) \right\}. $$

By Corollary \ref{orbitmargulis}, for each compact subset $F \subset M \setminus \ml{N}$, there exists $b_1 = 2b/(1-c) > 0$ and positive integer $m_F$ such that for all $n > m_F$ and $x \in F$, 
$$ (\mu^{(n)} * \delta_x)(f_\ml{N}) \leq b_1. $$
Therefore there exists a positive integer $n_F \geq m_F$ such that for all $n > n_F$ and $x \in F$, 
$$ (\overline{\mu}^{(n)} * \delta_x)(f_\ml{N}) \leq b_1 + 1. $$
Thus for all $n > n_F$, $x \in F$ and $L > 0$, we have
$$ (\overline{\mu}^{(n)} * \delta_x)(\{p \in M: f_\ml{N}(p) > L\}) < \frac{b_1+1}{L}. $$
Therefore by the choice of $\O_{\ml{N}, \ep}$, we know that $(\overline\mu^{(n)}* \delta_x)(\O_{\ml{N}, \ep}) < \ep$. Moreover, it is clear from the definition of $f_\ml{N}$ and the choice of $u$ in Lemma \ref{margulis} that 
$$ (\O_{\ml{N}, \ep})^c = \left\{x \in M: f_\ml{N}(x) \leq \frac{1}{\ep} \left(\frac{2b}{1 - c}+1 \right) \right\} $$
is compact. 
\end{proof}

\begin{proof}[Proof of Proposition \ref{equidistribution}]
Assume that the conclusion of the assertion does not hold. Then there exists $\varphi \in C(M)$, $\ep > 0$, $x \in M$ with infinite $\G$-orbit and a subsequence $n_k \to \infty$ such that
$$ |(\overline\mu^{(n_k)} * \delta_x)(\varphi) - m(\varphi)| \geq \ep. $$
By compactness of the space of probability measures on $M$ with the weak-* topology, we may assume that $\overline\mu^{(n_k)} * \delta_x \to \nu$ for some probability measure $\nu$. 

\noindent First note that $\nu$ is a $\mu$-stationary measure. By Proposition \ref{countable}, there are at most countably many finite $\G$-orbits. Therefore by Proposition \ref{measurerigidity}, we have the ergodic decomposition of $\nu$: 
$$ \nu = \sum_{\ml{N} \subset M} a_\ml{N} \nu_{\ml{N}} + am, $$
where the sum is over all finite $\G$-orbit $\ml{N}$. Here $a, a_\ml{N} \in [0, 1]$, and $\nu_{\ml{N}}$ is the probability measure supported on the finite $\G$-orbit $\ml{N}$ with uniform distribution. It remains to show that $a_\ml{N} = 0$ for all finite $\G$-orbit $\ml{N}$. 

For each finite $\G$-orbit $\ml{N}$, as $x \notin \ml{N}$ by assumption, we may apply Lemma \ref{small} with $\ml{N}$ and compact $F = \{x\}$. Then for any $\ep > 0$, there exists a positive integer $n_x$ such that for all $n > n_x$, $(\overline\mu^{(n)} * \delta_x)((\O_{\ml{N}, \ep})^c) \geq 1 - \ep$. Passing to the limit along the subsequence $n_k \to \infty$, we have
$$ \nu((\O_{\ml{N}, \ep})^c) \geq 1 - \ep. $$
As $\ep > 0$ is arbitrary, we have $\nu(\ml{N}) = 0$. Hence $a_\ml{N} \leq \nu(\ml{N}) = 0$. 
\end{proof}

\begin{proof}[Proof of Proposition \ref{orbitclosure}] 
This is an immediate consequence of Proposition \ref{equidistribution}, as every nonempty open subset of $M$ has positive volume. 
\end{proof}

\section{Geometric interpretation of uniform expansion} \label{gcriterion}

In the rest of the paper, we study how to verify uniform expansion in concrete settings. In this section, we give a geometric perspective of uniform expansion by visualizing it on the hyperbolic disk. 

\subsection{Cartan decomposition and hyperbolic geometry}

Let $F \in SL_2(\R)$. Throughout we identify the real projective line $\bb{P}^1 = \bb{P}^1(\R)$ with $\R / \pi \Z$ as metric spaces, i.e. we identify each line in $\R^2$ through the origin with the angle it makes with the positive horizontal axis. Recall that the Cartan decomposition of $F$ is given by
$$ F = r_{-\varphi} a_\ld r_\theta, \quaddd \text{ where } \; r_\theta = \begin{pmatrix} \cos \theta & \sin \theta \\ -\sin \theta & \cos \theta \end{pmatrix}\; \text{ and } \; a_\ld = \begin{pmatrix} \ld & 0 \\ 0 & \ld^{-1} \end{pmatrix}, $$
for some $\ld \geq 1$ and $\varphi, \theta \in S^1 = \R / 2\pi \Z$. Moreover, 
$$ \ld = \| F \| := \sup_{v \in \R^2 \setminus \{0\}} \frac{\| Fv \|}{\| v \|} $$
is the (operator) \emph{norm} of the matrix $F$. We remark that if $\ld = \| F \|> 1$, then $\varphi$ and $\theta$ are uniquely defined modulo $\pi$, i.e. correspond to a unique element in $\P^1$. We call $\theta \in \P^1$ \emph{the expanding direction} of $F$ since 
$$ \| F(\theta) \| = \sup_{\theta' \in \P^1} \|F(\theta')\| = \ld, $$
where $F(\theta)$ is the vector $F \begin{pmatrix} \cos \theta \\ \sin \theta \end{pmatrix}$. It is easy to see that if we let $\theta_F := \theta + \pi/2 \in \P^1$, then 
$$ \| F(\theta_F) \| = \inf_{\theta' \in \P^1} \| F(\theta') \| = \ld^{-1}. $$
Hence for $\| F \| > 1$, we call $\theta_F = \theta + \pi/2 \in \P^1$ the \emph{contracting direction} of $F$. Notice also that $\varphi \in \P^1$ and $\varphi+\pi/2 \in \P^1$ are the contracting and expanding directions of $F^{-1}$. 

In certain computation we find it helpful to have an explicit formula to compute the contraction direction and the norm given the matrix $F \in SL_2(\R)$. This is given by the following simple lemma. 

\begin{lemma} \label{simple}
Let $F = \begin{pmatrix} a & b \\ c & d \end{pmatrix} \in SL_2(\R)$ with $\| F \| > 1$. Then 
\begin{enumerate}[label=(\alph*)]
\item the contracting direction $\theta_F \in \P^1$ satisfies
$$ \tan 2\theta_F = \frac{2(ab+cd)}{a^2 + c^2 - b^2 - d^2}, $$
here we follow the convention that $1/0 = \infty$ and that $\tan \varphi = \infty$ implies $\varphi = \pi/2 \in \P^1$. \\
\item The norm $\ld := \| F \|$ satisfies
$$ \ld^2 + \ld^{-2} = a^2 + b^2 + c^2 + d^2. $$
In particular, if $a^2 + b^2 + c^2 + d^2 \gg 1$, then 
$$ \ld \sim \sqrt{a^2 + b^2 + c^2 + d^2}. $$
\end{enumerate}
\end{lemma}

\begin{proof}
Part (a) is a straightforward computation by considering the function
$$ f(\theta) := \| F(\theta) \|^2  = \left\| \begin{pmatrix} a & b \\ c & d \end{pmatrix} \begin{pmatrix} \cos \theta \\ \sin \theta \end{pmatrix} \right\|^2. $$
Notice that for $\| F \| > 1$, $f$ is not a constant function, and the expanding and contracting directions are precisely the critical points of $f$, i.e. when $f'(\theta) = 0$. 

For part (b), we observe that 
$$ \tr (F^T F) = a^2 + b^2 + c^2 + d^2. $$
On the other hand, if we write $F = r_{-\varphi} a_\ld r_\theta$, then 
$$ F^T F = (r_{-\theta} a_\ld r_{\varphi})(r_{-\varphi} a_\ld r_\theta) = r_{-\theta} a_\ld^2 r_\theta. $$
Hence its trace equals $\ld^2 + \ld^{-2}$. 
\end{proof}

We also find it helpful to think of each $F \in SL_2(\R)$ as a point of the unit tangent bundle of the hyperbolic plane in the disk model $T^1 \bb{D}$, using the identification $T^1 \bb{D} \leftrightarrow PSL_2(\R) := SL_2(\R) / \{\pm I\}$ (Figure \ref{hyperbolic}). Recall that the group $PSL_2(\R)$ is the group of orientation-preserving isometries of the hyperbolic plane $\bb{H}^2 := \{z \in \C: \Im(z) > 0\}$, which can be identified isometrically with the hyperbolic disk $\bb{D} := \{w \in \C: |w| < 1\}$ via the map $z \mapsto (z - i) / (z+i)$. $PSL_2(\R)$ acts simply transitively on the unit tangent bundle $T^1 \bb{D}$, hence one can identify $PSL_2(\R)$ with $T^1\bb{D}$ so that the identity element $e$ corresponds to the unit vector based at the origin pointing rightward. Moreover the identification is such that the isometry $g$ on $T^1 \bb{D}$ corresponds to the \emph{right} multiplication by the inverse $g^{-1}$ on $PSL_2(\R)$. We visualize the base point on the disk model $\bb{D} \leftrightarrow SO(2) \backslash SL_2(\R)$. For instance, the matrix $F = r_{-\varphi} a_\ld r_\theta \in SL_2(\R)$ corresponds to the point $P_F \in \bb{D}$ with polar coordinates $\left( 2 \log \ld, 2\theta \right)$ (the first coordinate measured in hyperbolic distance) and the unit tangent vector with angle $2(\theta - \varphi)$ from the positive real axis.


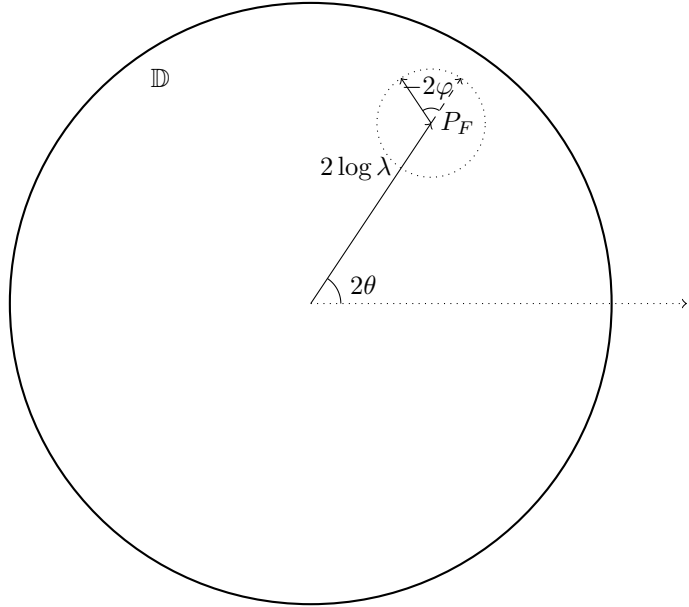
\begin{figure}[h]
\centering
\begin{tikzpicture}
\draw [thick] (0,0) circle [radius=4];
\draw [dotted, ->] (0, 0) -- (5, 0);
\draw [->] (0, 0) -- (1.6, 2.4);
\draw [->] (1.6, 2.4) -- (1.2, 3);
\draw [dashed, ->] (1.6, 2.4) -- (2, 3);
\draw [domain=0:56.31] plot ({0.4*cos(\x)}, {0.4*sin(\x)});
\draw [domain=56.31:123.69] plot ({1.6+0.2*cos(\x)}, {2.4+0.2*sin(\x)});
\draw [dotted] (1.6, 2.4) circle [radius=0.72];
\node [above] at (0.7, 0) {$2 \theta$};
\node [above] at (1.55, 2.6) {$-2 \varphi$};
\node [right] at (1.6, 2.4) {$P_F$};
\node [left] at (1.2, 1.8) {$2 \log \ld$};
\node at (-2, 3) {$\bb{D}$};
\end{tikzpicture}
\caption{The matrix $F = r_{-\varphi} a_\ld r_\theta \in SL_2(\R)$ in the hyperbolic disk}
\label{hyperbolic}
\end{figure}

Hence one can read off the norm of $F$ from the distance between $P_F$ and the origin, and read off the contracting direction from the angle from the positive axis. 

Now we relate this picture with uniform expansion. From now on, we assume that $\mu$ is finitely supported, so that the uniform expansion condition reduces to a finite sum. For simplicity, for the moment we also assume that the maps in the support of $\mu$ have the same mass. Let $\O := \{f_1, f_2, \ldots, f_d\} \subset \Diff^2(M)$ be the support of $\mu$. Then $\mu$ is uniformly expanding if there exists $C > 0$ and $N \in \N$ such that for all $x \in M$ and $v \in T_x M$, 
$$ \sum_{\og \in \O^N} \log \frac{\| D_x f_\og^N(v)\|}{\| v \|} > C. $$
Here we recall that for $\og = (\og_1, \og_2, \ldots, \og_N) \in \O^N$ and $1 \leq i \leq N$, $f_\og^i := \og_i \circ \og_{i-1} \circ \cdots \circ \og_1$. Note that by picking a measurably varying basis for the tangent bundle $TM$, we can identify $D_x f_\og^N$ as an element in $SL_2(\R)$. Note that if $\theta \in \P^1$ is the contracting direction of $D_x f_\og^N$, then $\log \| D_x f_\og^N(\theta) \| < 0$. In particular if for some $x \in M$, $\theta \in \P^1$ is close to the contracting direction of $D_x f_\og^N$ for many words $\og \in \O^N$, then uniform expansion cannot hold. Hence verifying uniform expansion amounts to checking that the contracting directions of $D_x f_\og^N$ are ``spread out'' enough. On the hyperbolic disk, for each $x \in M$, we can draw the matrices $D_x f_\og^N$ as endpoints of a tree from the origin, where each node with graph distance $i$ from the origin corresponds to a matrix $D_x f_\og^i$ (Figure \ref{basictree}, the dashed lines indicate the contracting directions of $D_x f_\og^N$ for $N = 3$). Hence verifying uniform expansion reduces to studying the geometry of the contracting directions. 

\begin{figure}[ht]
\centering
\begin{tikzpicture}
\draw [thick] (0,0) circle [radius=4];
\draw [->] (0, 0) -- (1, 1);
\draw [->] (0, 0) -- (-1.5, -1);
\draw [->] (1, 1) -- (2, 0.8);
\draw [->] (1, 1) -- (0.9, 3);
\draw [->] (-1.5, -1) -- (-2.5, -0.6);
\draw [->] (-1.5, -1) -- (-0.5, -3);
\draw [->] (2, 0.8) -- (2.6, 0.6);
\draw [->] (2, 0.8) -- (2.6, 1.5);
\draw [->] (0.9, 3) -- (1.3, 2.8);
\draw [->] (0.9, 3) -- (0.5, 3.7);
\draw [->] (-2.5, -0.6) -- (-3, 1);
\draw [->] (-2.5, -0.6) -- (-3, -1.8);
\draw [->] (-0.5, -3) -- (1, -2.5);
\draw [->] (-0.5, -3) -- (-1.7, -2.4);
\draw [line width = 0.05mm, dashed, ->] (0, 0) -- (4.385, 1.01);
\draw [line width = 0.05mm, dashed, ->] (0, 0) -- (3.898, 2.25);
\draw [line width = 0.05mm, dashed, ->] (0, 0) -- (1.895, 4.08);
\draw [line width = 0.05mm, dashed, ->] (0, 0) -- (0.6, 4.46);
\draw [line width = 0.05mm, dashed, ->] (0, 0) -- (-4.27, 1.423);
\draw [line width = 0.05mm, dashed, ->] (0, 0) -- (-3.86, -2.315);
\draw [line width = 0.05mm, dashed, ->] (0, 0) -- (1.671, -4.178);
\draw [line width = 0.05mm, dashed, ->] (0, 0) -- (-2.601, -3.672);
\draw [fill] (2.6, 0.6) circle [radius=0.07];
\draw [fill] (2.6, 1.5) circle [radius=0.07];
\draw [fill] (1.3, 2.8) circle [radius=0.07];
\draw [fill] (0.5, 3.7) circle [radius=0.07];
\draw [fill] (-3, 1) circle [radius=0.07];
\draw [fill] (-3, -1.8) circle [radius=0.07];
\draw [fill] (1, -2.5) circle [radius=0.07];
\draw [fill] (-1.7, -2.4) circle [radius=0.07];
\end{tikzpicture}
\caption{The tree representing the random walk after $3$ steps}
\label{basictree}
\end{figure}
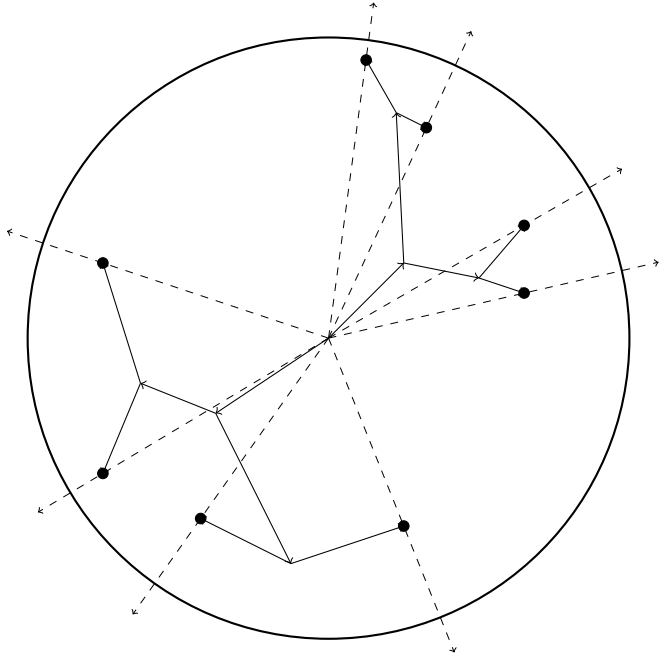

\subsection{Estimates on changes of the contracting directions}

The following lemma provides a lower bound on the expansion of a given matrix $F \in SL_2(\R)$ in the direction $\theta$, depending on the norm of $F$ and how far $\theta$ is from the contracting direction of $F$.
\begin{lemma} \label{distance}
For all $F \in SL_2(\R)$ with norm $\| F \| > 1$ and contracting direction $\theta_F \in \P^1$, we have 
$$ \| F(\theta) \| \geq \frac{2}{\pi} \| F \| \cdot d(\theta, \theta_F) \quaddd \text{ for all } \theta \in \P^1. $$
Here we recall that the metric $d$ on $\P^1$ is given by the identification $\P^1 \leftrightarrow \R / \pi \Z$. 
\end{lemma}

\begin{proof}
By the Cartan decomposition one may assume that $F$ is a diagonal matrix with entries $\ld$ and $\ld^{-1}$, with $\ld = \| F \|$. The lemma now follows from a direct calculation.
\end{proof}

For matrices $M_1, M_2 \in SL_2(\R)$, the following lemma shows that if $M_2$ has large norm $\ld_2$, then as long as the contracting direction of $M_1$ is far away from the contracting direction of $M_2^{-1}$, as we vary the contracting direction of $M_1$, the contracting direction of the product $M_1 M_2$ changes by $1/\ld_2^2$ of that amount. 

\begin{lemma} \label{basicangleestimate}
Let $M_1, M_2 \in SL_2(\R)$. Let $\ld_i = \| M_i \| > 1$ for $i = 1, 2$ and $\varphi = \theta_{M_1} + \pi/2 - \theta_{M_2^{-1}}$, i.e. $\varphi$ is the distance between the contracting direction of $M_1$ and the expanding direction of $M_2^{-1}$.
\begin{enumerate}[label=(\alph*)]
\item If $\| M_1 M_2 \| > 1$, then 
$$ \frac{d\theta_{M_1 M_2}}{d\theta_{M_2}} = 1, $$
where we treat $\theta_{M_1 M_2}$ as a function of $\theta_{M_2}$ by fixing $M_1$, $\theta_{M_2^{-1}}$ and $\ld_2$. 
\item If $\ld_2 \gg 1$ and $d(\varphi, \pi/2) \gtrsim \ld_2^{-1}$, then
$$ \frac{d\theta_{M_1 M_2}}{d\theta_{M_1}} \sim \frac{2(1 + k \cos 2\varphi)}{(k+\cos 2\varphi)^2} \frac{1}{\ld_2^2}, \quaddd \text{ where } \quadd k = \frac{\ld_1^2 + \ld_1^{-2}}{\ld_1^2 - \ld_1^{-2}} = 1 + \frac{2}{\ld_1^4 - 1}. $$
Here we treat $\theta_{M_1 M_2}$ as a function of $\theta_{M_1}$ by fixing $\theta_{M_1^{-1}}$, $\ld_1$ and $M_2$. Furthermore, if $\ld_1 \gg 1$ and $d(\varphi, \pi/2) \gtrsim \ld_1^{-1}$ as well, then
$$ \frac{d\theta_{M_1 M_2}}{d\theta_{M_1}} \sim \frac{2}{(1 + \cos 2\varphi)} \frac{1}{\ld_2^2}. $$
\end{enumerate}
\end{lemma}

\begin{proof}
For (a), write $M_2$ in its Cartan decomposition $M_2 = r_{-\varphi_2} a_{\ld_2} r_{\theta_2}$, and write $M_1 r_{-\varphi_2} a_{\ld_2}$ in its Cartan decomposition 
$$ M_1 r_{-\varphi_2} a_{\ld_2} = r_{-\varphi'} a_{\ld'} r_{\theta'}. $$
Then 
$$ M_1 M_2 = M_1 r_{-\varphi_2} a_{\ld_2} r_{\theta_2} = r_{-\varphi'} a_{\ld'} r_{\theta' + \theta_2}. $$
By the uniqueness of the Cartan decomposition (up to $\pm I$), we have $\theta_{M_1 M_2} = \theta_{M_2} + \theta'$, where $\theta'$ depends only on $M_1$, $\varphi_2 = \theta_{M_2^{-1}}$ and $\ld_2$, hence the result of (a). This statement can be visualized on the hyperbolic disk (Figure \ref{figurebasicangleestimate}).

\begin{figure}[ht]
\centering
\begin{tikzpicture}
\draw [thick] [domain=-5:185] plot ({6*cos(\x)}, {6*sin(\x)});
\draw [dotted, ->] (0, 0) -- (7, 0);
\draw [->] (0, 0) -- (3, 0);
\draw [->] (0, 0) -- (2.9, 0.7764);
\draw [->] (3, 0) -- (4, 4);
\draw [->] (2.9, 0.7764) -- (2.828, 4.9);
\draw [dashed, ->] (0, 0) -- (4, 4); 
\draw [dashed, ->] (0, 0) -- (2.828, 4.9); 
\draw [domain=0:15] plot ({1.5*cos(\x)}, {1.5*sin(\x)});
\draw [domain=45:60] plot ({1.5*cos(\x)}, {1.5*sin(\x)});
\node [right] at (1.5, 0.25) {$d\theta_{M_2}$};
\node [above] at (1.4, 1.2) {$d\theta_{M_1 M_2}$};
\node [below] at (1.5, 0) {$2 \log \ld_2$};
\node [right] at (2.75, 4.3) {$2 \log \ld_1$};
\node [right] at (3.65,  2.6) {$2 \log \ld_1$};
\node [above] at (-4.4, 4.4) {$\bb{D}$};
\end{tikzpicture}
\caption{The change of $\theta_{M_1 M_2}$ as $\theta_{M_2}$ varies.}
\label{figurebasicangleestimate}
\end{figure}
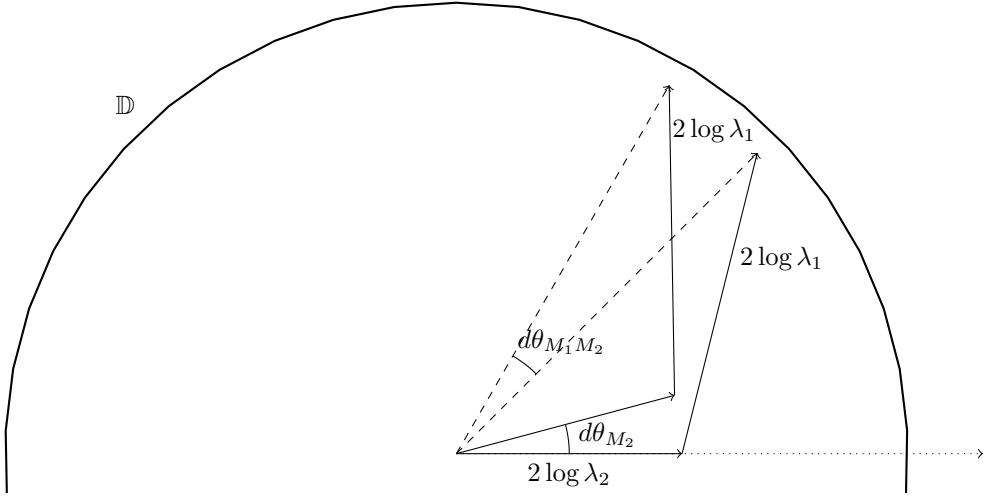

For (b), the assumptions $d(\varphi, \pi/2) \gtrsim \ld_2^{-1}$ and $\ld_1, \ld_2 > 1$ imply that $\| M_1 M_2 \| > 1$. Thus $\theta_{M_1 M_2}$ is well-defined. By applying the Cartan decomposition, we may, without loss of generality, assume that $\theta := \theta_{M_1 M_2}$ is the contracting direction of
$$ \begin{pmatrix} \ld_1 & 0 \\ 0 & \ld_1^{-1} \end{pmatrix} \begin{pmatrix} \cos \varphi & \sin \varphi \\ -\sin \varphi & \cos \varphi \end{pmatrix} \begin{pmatrix} \ld_2 & 0 \\ 0 & \ld_2^{-1} \end{pmatrix}. $$
Note that
$$ \frac{d\theta_{M_1 M_2}}{d\theta_{M_1}} = \frac{d\theta}{d\varphi}. $$
The statement can be illustrated on the hyperbolic disk (Figure \ref{figurebasicangleestimate2}). 

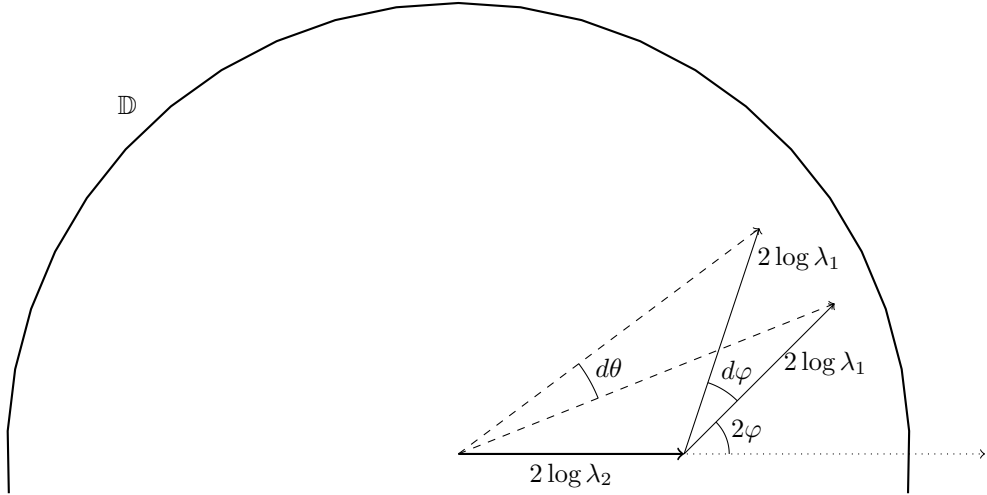
\begin{figure}[ht]
\centering
\begin{tikzpicture}
\draw [thick] [domain=-5:185] plot ({6*cos(\x)}, {6*sin(\x)});
\draw [dotted, ->] (0, 0) -- (7, 0);
\draw [->, thick] (0, 0) -- (3, 0);
\draw [->] (3, 0) -- (4, 3);
\draw [->] (3, 0) -- (5, 2);
\draw [dashed, ->] (0, 0) -- (4, 3); 
\draw [dashed, ->] (0, 0) -- (5, 2); 
\draw [domain=21.8:36.87] plot ({2*cos(\x)}, {2*sin(\x)});
\draw [domain=45:71.565] plot ({3+cos(\x)}, {sin(\x)});
\draw [domain=0:45] plot ({3+0.6*cos(\x)}, {0.6*sin(\x)});
\node [above] at (2, 0.9) {$d\theta$};
\node [above] at (3.7, 0.8) {$d\varphi$};
\node [below] at (1.5, 0) {$2 \log \ld_2$};
\node [right] at (4.2, 1.2) {$2 \log \ld_1$};
\node [right] at (3.85,  2.6) {$2 \log \ld_1$};
\node [right] at (3.5, 0.3) {$2 \varphi$};
\node [above] at (-4.4, 4.4) {$\bb{D}$};
\end{tikzpicture}
\caption{The change of $\theta$ as $\varphi$ varies.}
\label{figurebasicangleestimate2}
\end{figure}

Using Lemma \ref{simple}(a), one computes directly that
$$ \cot 2\theta = \frac{1}{2} (\ld_2^2 + \ld_2^{-2}) \cot 2\varphi + \frac{1}{2} \frac{(\ld_1^2 + \ld_1^{-2})(\ld_2^2 - \ld_2^{-2})}{\ld_1^2 - \ld_1^{-2}} \csc 2\varphi. $$
Hence upon taking derivative, one gets
$$ \frac{d\theta}{d\varphi} = \frac{\frac{(\ld_1^2 + \ld_1^{-2})(\ld_2^2 - \ld_2^{-2})}{\ld_1^2 - \ld_1^{-2}} \cos 2\varphi + (\ld_2^2 + \ld_2^{-2})}{2 \sin^2 2\varphi + \frac{1}{2} \left( \frac{(\ld_1^2 + \ld_1^{-2})(\ld_2^2 - \ld_2^{-2})}{\ld_1^2 - \ld_1^{-2}} + (\ld_2^2 + \ld_2^{-2}) \cos 2\varphi \right)^2 }. $$
Thus for $\ld_2 \gg 1$, let $k = (\ld_1^2 + \ld_1^{-2})/(\ld_1^2 - \ld_1^{-2})$, then
$$ \frac{d\theta}{d\varphi} \sim \frac{2(1 + k \cos 2\varphi)}{(k+\cos 2\varphi)^2} \frac{1}{\ld_2^2}. $$
In addition, by taking $\ld_1 \gg 1$, we have $k \sim 1$, so
$$ \frac{d\theta}{d\varphi} \sim \frac{2}{(1+\cos 2\varphi)} \frac{1}{\ld_2^2}. $$
It is clear from Figure \ref{figurebasicangleestimate2} that when $\varphi$ is close to $\pi/2$, the random walk ``backtracks'' towards the origin, so we do not expect a good estimate on $d\theta / d\varphi$.
\end{proof}

\subsection{A general criterion for uniform expansion}

We finish this section with a sufficient condition for uniform expansion on one step of the random dynamics. As mentioned in the introduction, this criterion illustrates that overlap of contraction directions and maps close to rotations are essentially the two obstructions to uniform expansion. Even though we will not use this criterion in the rest of the paper, one may consider the verification in the next few sections as proving a more refined version of Proposition \ref{generalcriterion} (depending on the specific features of each application) and the verification of this more refined criterion.

Given $F \in SL_2(\R)$, recall that we define $\ld_F := \| F \|$ to be the norm of $F$ with $\ld_F > 1$, and $\theta_{F} \in \P^1$ to be the contracting direction. 
\begin{proposition} \label{generalcriterion}
For all $\ld_{\mathrm{crit}} > 0$, $\ld_{\mathrm{max}} > 0$ and small enough $\ep > 0$ satisfying $\ds\frac{1}{\sin \ep} \sqrt{2 + \frac{1}{\ep}} < \ld_{\mathrm{crit}} \leq \ld_{\mathrm{max}}$, there exists $\eta = \eta(\ld_{\mathrm{crit}}, \ld_{\mathrm{max}}, \ep) \in (0, 1)$ such that if for all $(x, \theta) \in T^1 M$, 
$$ \mu(\{f: d(\theta_{D_xf}, \theta) > \ep \text{ and } \ld_{D_xf} > \ld_{\mathrm{crit}}\}) > \eta, \quad \text{ and } \quad \ld_{D_xf} \leq \ld_{\mathrm{max}} \quad \text{ for } \mu\text{-a.s. } f, $$
then $\mu$ is uniformly expanding. Furthermore, $\eta$ can be made explicit. 
\end{proposition}

We think of $\ep$ as measuring the separation of the contracting directions at each point $x \in M$, $\ld_{\mathrm{crit}}$ as measuring how far $D_x f$ is from a rotation, and $\ld_{\mathrm{max}}$ as the maximum norm over all the points $x \in M$ and all the possible maps $f$ in the support of $\mu$. 

The idea of the proposition is that if at every point, the contracting directions of the diffeomorphisms are spread out enough and most of the diffeomorphisms are far from being a rotation, then with high probability the random walk does not backtrack. Lemma \ref{basicangleestimate}(b) and the next two lemma then tell us that the contracting directions of the random walk will eventually be spread out as well. In this case, as long as none of the norms dominate the others (bounded by $\ld_{\mathrm{max}}$), we can obtain uniform expansion. In particular, as we will see, $\eta$ is an increasing function of $\ld_{\mathrm{max}}$ and a decreasing function of $\ld_{\mathrm{crit}}$ and $\ep$. 

\begin{lemma} \label{expansion}
Fix $m > 1$. Let $M_1, M_2 \in SL_2(\R)$. Let $\ld := \| M_1 \| > 1$ and $\tau := \| M_2 \| > 1$ be the norm of $M_1$ and $M_2$, $\varphi = \theta_{M_1} + \pi/2 - \theta_{M_2^{-1}}$ be the difference between the contracting direction of $M_1$ and the expanding direction of $M_2^{-1}$.  Then the norm of the product $M_1 M_2$ is at least $\ld \tau / m$ if and only if 
$$ \cos 2 \varphi \geq \frac{2((\ld \tau / m)^2 + (\ld \tau / m)^{-2})}{(\ld^2 - \ld^{-2})(\tau^2 - \tau^{-2})} - \frac{\ld^2 + \ld^{-2}}{\ld^2 - \ld^{-2}} \cdot \frac{\tau^2 + \tau^{-2}}{\tau^2 - \tau^{-2}}. $$
In particular, if $\ld > \sqrt{m}$, $\tau > \sqrt{m}$ and $|\cos \varphi| \geq 1/m$, then the norm of $M_1 M_2$ is at least $\ld \tau / m$. 
\end{lemma}

\begin{proof}
The first equivalence is a calculation using the Cartan decomposition. Note that the norm $\ld$ of a matrix $\begin{pmatrix} a & b \\ c & d \end{pmatrix} \in SL_2(\R)$ is the unique root of 
$$ \ld^2  + \ld^{-2} = a^2 + b^2 + c^2 + d^2$$
with $\ld \geq 1$. In particular $\ld$ is an increasing function of $a^2 + b^2 + c^2 + d^2$. Now the norm of $M_1 M_2$ is the same as that of
$$ \begin{pmatrix} \ld & 0 \\ 0 & \ld^{-1} \end{pmatrix} \begin{pmatrix} \cos \varphi & \sin \varphi \\ -\sin\varphi & \cos \varphi \end{pmatrix} \begin{pmatrix} \tau & 0 \\ 0 & \tau^{-1} \end{pmatrix} = \begin{pmatrix} \ld \tau \cos \varphi & \ld \tau^{-1} \sin \varphi \\ -\ld^{-1} \tau \sin \varphi & \ld^{-1}\tau^{-1} \cos \varphi \end{pmatrix}. $$
Thus $\| M_1 M_2 \| \geq \ld \tau / m$ if and only if 
\begin{equation} \label{tedious}
 (\ld \tau \cos \varphi)^2 + (\ld \tau^{-1} \sin \varphi)^2 + (\ld^{-1} \tau \sin \varphi)^2 + (\ld^{-1}\tau^{-1} \cos \varphi)^2 \geq \left( \frac{\ld \tau}{m} \right)^2 + \left( \frac{\ld \tau}{m} \right)^{-2}. 
\end{equation}
Rearranging (\ref{tedious}) gives the first assertion.
Finally, the left hand side of (\ref{tedious}) is an increasing function of $\cos^2 \varphi$ for $\ld > 1$ and $\tau > 1$. One can verify directly that (\ref{tedious}) holds when $\ld > \sqrt{m}$, $\tau > \sqrt{m}$, $\cos^2 \varphi = 1/m^2$, therefore it also holds for $\cos^2 \varphi \geq 1/m^2$. 
\end{proof}

The next lemma controls the contracting direction of $M_1 M_2$ assuming no backtracking.

\begin{lemma} \label{anglebound}
Fix $m > 1$ large (an explicit lower bound will be obtained in the proof). Let $M_1, M_2 \in SL_2(\R)$. Let $\ld := \| M_1 \| > 1$ and $\tau := \| M_2 \| > 1$, $\varphi = \theta_{M_1} + \pi/2 - \theta_{M_2^{-1}} \in \P^1 = \R/\pi\Z$ as in the previous lemma. If $|\cos \varphi| \geq 1/m$ and $\tau \geq m$, 
$$ d(\theta_{M_2}, \theta_{M_1 M_2}) \leq \frac{m^2}{\tau^2}. $$
If we further assume that $\tau \geq \sqrt{2}m$, the conclusion holds for all $m > 1$.
\end{lemma}

\begin{proof}
Note that if $\varphi = 0$, $d(\theta_{M_2}, \theta_{M_1 M_2}) = 0$. Therefore we need to give an upper bound on the increment of $\theta_{M_1 M_2}$ as we vary $\varphi$ within the given range. Again by the Cartan decomposition, it suffices to consider the matrix
$$ \begin{pmatrix} \ld & 0 \\ 0 & \ld^{-1} \end{pmatrix} \begin{pmatrix} \cos \varphi & \sin \varphi \\ -\sin\varphi & \cos \varphi \end{pmatrix} \begin{pmatrix} \tau & 0 \\ 0 & \tau^{-1} \end{pmatrix} = \begin{pmatrix} \ld \tau \cos \varphi & \ld \tau^{-1} \sin \varphi \\ -\ld^{-1} \tau \sin \varphi & \ld^{-1}\tau^{-1} \cos \varphi \end{pmatrix}, $$
and give an upper bound on the absolute value of its contracting direction $\theta$. By Lemma \ref{simple} (a), one obtains, 
$$ \tan 2\theta = \frac{(\ld^2 - \ld^{-2}) \sin 2\varphi}{\frac{1}{2}(\ld^2 + \ld^{-2})(\tau^2 - \tau^{-2}) + \frac{1}{2}(\ld^2 - \ld^{-2})(\tau^2 + \tau^{-2}) \cos 2\varphi}. $$
Since $|2\theta| \leq |\tan 2\theta|$, and also the right hand side is an odd function of $\varphi$, it remains to show that for $\varphi \in [0, \pi/2]$ with $|\cos \varphi| \geq 1/m$, 
\begin{equation} \label{boundbound}
 f(\varphi) := \frac{(\ld^2 - \ld^{-2}) \sin 2\varphi}{(\ld^2 + \ld^{-2})(\tau^2 - \tau^{-2}) + (\ld^2 - \ld^{-2})(\tau^2 + \tau^{-2}) \cos 2\varphi} \leq \frac{m^2}{\tau^2}. 
\end{equation}
Clearly $|\cos \varphi| \geq 1/m$ if and only if $\cos 2\varphi \geq -1 + 2/m^2$. \\
{\bf Case 1}: $\ld \leq \tau$. Then 
$$ (\ld^2 + \ld^{-2})(\tau^2 - \tau^{-2}) \geq (\ld^2 - \ld^{-2})(\tau^2 + \tau^{-2}). $$
Using the fact that $\cos 2\varphi \geq -1 + 2/m^2$, the denominator of $f(\varphi)$ has a lower bound
$$ (\ld^2 + \ld^{-2})(\tau^2 - \tau^{-2}) + (\ld^2 - \ld^{-2})(\tau^2 + \tau^{-2}) \cos 2\varphi \geq (\ld^2 - \ld^{-2})(\tau^2 + \tau^{-2})(1+\cos 2\varphi) \geq \frac{2}{m^2}(\ld^2 - \ld^{-2})\tau^2, $$
and (\ref{boundbound}) holds. \\
{\bf Case 2}: $\ld \geq \tau$. We let $k := (\ld^2 + \ld^{-2})/(\ld^2 - \ld^{-2}) > 1$ and write
$$ f(\varphi) = \frac{\sin 2\varphi}{k(\tau^2 - \tau^{-2}) + (\tau^2 + \tau^{-2})\cos 2\varphi}. $$
Since $\ld \geq \tau$, 
$$ (\ld^2 + \ld^{-2})(\tau^2 - \tau^{-2}) \leq (\ld^2 - \ld^{-2})(\tau^2 + \tau^{-2}), $$
and therefore $k(\tau^2 - \tau^{-2}) \leq (\tau^2 + \tau^{-2})$. 
Now compute
$$ f'(\varphi) = 2\frac{k(\tau^2 - \tau^{-2}) \cos 2\varphi + (\tau^2 + \tau^{-2})}{(k(\tau^2 - \tau^{-2}) + (\tau^2 + \tau^{-2})\cos 2\varphi)^2} > 0. $$
On the other hand, note that the denominator of $f(\varphi)$ is positive for $\varphi \in [0, \pi/2]$ with $|\cos \varphi| \geq 1/m$:
\begin{align*}
 (\ld^2 + \ld^{-2})(\tau^2 - \tau^{-2}) + (\ld^2 - \ld^{-2})(\tau^2 + \tau^{-2}) \cos 2\varphi &> (\ld^2 - \ld^{-2})[(\tau^2 - \tau^{-2}) + (\tau^2 + \tau^{-2})\cos 2\varphi] \\
&\geq (\ld^2 - \ld^{-2})[(\tau^2 - \tau^{-2}) + (\tau^2 + \tau^{-2})(-1+2/m^2)] \\
&\geq (\ld^2 - \ld^{-2})\left( \frac{2}{m^2}(\tau^2+\tau^{-2}) - 2\tau^{-2} \right). 
\end{align*}
Since $\tau \geq m$, $2\tau^2/m^2 \geq 2\tau^{-2}$, and hence the right hand side is positive. Therefore within the given range of $\varphi$, $f(\varphi)$ is a smooth increasing function of $\varphi$, hence its maximum occurs for $\varphi = \varphi_0$, where $\varphi_0 \in [0, \pi/2]$ is such that $\cos 2\varphi_0 = -1 + 2/m^2$, or equivalently $|\cos \varphi_0| = 1/m$. Now
$$ \sin 2\varphi_0 = 2\sin \varphi_0 \cos \varphi_0 < \frac{2}{m}. $$
Therefore recalling that $k > 1$,
$$ f(\varphi_0) = \frac{\sin 2\varphi_0}{k(\tau^2 - \tau^{-2}) + (\tau^2 + \tau^{-2})\cos 2\varphi_0} < \frac{2/m}{(\tau^2 - \tau^{-2}) + (\tau^2 + \tau^{-2})(-1+2/m^2)} = \frac{m^2}{\tau^2} \left( \frac{1/m}{1-(m^2-1)\tau^{-4}} \right). $$
Finally, as $\tau \geq m$, we have
$$ \frac{1/m}{1-(m^2-1)\tau^{-4}} \leq \frac{1/m}{1-(m^2-1)m^{-4}} = \frac{m^3}{m^4 + 1 - m^2}. $$
As $m \to \infty$, the right hand side goes to $0$, therefore for large enough $m$, it is less than $1$, hence for large enough $m$ (can take, say, $m > 1.4$), 
$$ f(\varphi_0) \leq m^2/\tau^2, $$
and the result follows. If we assume that $\tau \geq \sqrt{2}m$, then we have instead
$$ \frac{1/m}{1+(1-m^2)\tau^{-4}} \leq \frac{1/m}{1+(1-m^2)m^{-4}/4} = \frac{m^3}{m^4 + (1 - m^2)/4}. $$
The right hand side is a smooth decreasing function for all $m > 1$ and is exactly $1$ at $m = 1$, hence it is at most $1$ for all $m \geq 1$, and so $f(\varphi_0) \leq m^2 / \tau^2$ for all $m \geq 1$. 
\end{proof}

\begin{proof}[Proof of Proposition \ref{generalcriterion}]
Let $m_0 := 1/\sin \ep$. Clearly $\ld_{\mathrm{crit}} > m_0$. Fix $x \in M$ and $\theta \in T_x^1 M$. Consider $n$ maps $f_1, f_2, \ldots, f_n \in \Diff^2(M)$ satisfying
\begin{align} \label{anyway1}
 \quaddd \ld_{D_{f_{i-1} f_{i-2} \cdots f_1(x)} f_i} > \ld_{\mathrm{crit}} \quadd \text{ and } \quadd \ld_{D_{f_{i-1} f_{i-2} \cdots f_1(x)} f_i} \leq \ld_{\mathrm{max}} \quaddd \text{ for all } i, 
\end{align}
and
\begin{align} \label{anyway2}
 d(\theta_{D_x f_1}, \theta) > \ep, \quadd d(\theta_{D_{f_{i-1} f_{i-2} \cdots f_1(x)} f_i}, \theta_{(D_x f_{i-1} f_{i-2} \cdots f_1)^{-1}} ) > \ep \quaddd \text{ for all } i.
 \end{align}
For each $i > 1$, we apply Lemma \ref{expansion} with $M_1 = D_{f_{i-1} f_{i-2} \cdots f_1(x)} f_i$, $M_2 = D_x f_{i-1} f_{i-2} \cdots f_1$ and $m = m_0$. Then $M_1 M_2 = D_x f_{i} f_{i-1} \cdots f_1$. Note that the corresponding 
$$ \varphi = \theta_{D_{f_{i-1} f_{i-2} \cdots f_1(x)} f_i} + \pi/2 - \theta_{(D_x f_{i-1} f_{i-2} \cdots f_1)^{-1}} $$
satisfies $|\cos \varphi| = |\sin (\theta_{D_{f_{i-1} f_{i-2} \cdots f_1(x)} f_i} - \theta_{(D_x f_{i-1} f_{i-2} \cdots f_1)^{-1}})| \geq |\sin \ep| = 1/m_0$. Also $\| D_{f_{i-1} f_{i-2} \cdots f_1(x)} f_i \| > \ld_{\mathrm{crit}} > m_0 > \sqrt{m_0}$ for all $i$, thus by induction using Lemma \ref{expansion} we have
$$ \ld_{f_i f_{i-1} \cdots f_1} \geq \frac{\ld_{\mathrm{crit}}^i}{m_0^{i-1}} $$
(note that the right hand side is greater than $\ld_{\mathrm{crit}} > m_0 > \sqrt{m_0}$.) 
Since $\ld_{\mathrm{crit}} > \sqrt{2}m_0$, by Lemma \ref{anglebound}, we get that 
$$ d(\theta_{D_x f_{i-1} f_{i-2} \cdots f_1}, \theta_{D_x f_{i} f_{i-1} \cdots f_1}) \leq m_0^2 \left( \frac{\ld_{\mathrm{crit}}^i}{m_0^{i-1}} \right)^{-2} = \left( \frac{m_0}{\ld_{\mathrm{crit}}} \right)^{2i}. $$
Since $d(\theta_{D_x f_1}, \theta) > \ep$, we have
$$ d(\theta_{D_x f_n f_{n-1} \cdots f_1}, \theta) > \ep - \left( \left( \frac{m_0}{\ld_{\mathrm{crit}}} \right)^2 + \left(\frac{m_0}{\ld_{\mathrm{crit}}} \right)^4 + \cdots + \left(\frac{m_0}{\ld_{\mathrm{crit}}} \right)^{2n} \right) > \ep - \frac{(m_0/\ld_{\mathrm{crit}})^2}{1 - (m_0/\ld_{\mathrm{crit}})^2}. $$
As $\ds\frac{1}{\sin \ep} \sqrt{2 + \frac{1}{\ep}} < \ld_{\mathrm{crit}}$, we have $ \ds\frac{(m_0/\ld_{\mathrm{crit}})^2}{1 - (m_0/\ld_{\mathrm{crit}})^2} < \ep/2 $ (recall that $m_0 = 1/\sin \ep$). Thus $ d(\theta_{D_x f_n f_{n-1} \cdots f_1}, \theta) > \ep/2$. By Lemma \ref{distance}, 
$$ \log \| D_x f_n f_{n-1} \cdots f_1 (\theta) \| \geq \log \left( \frac{2}{\pi} \ld_{f_n f_{n-1} \cdots f_1} d(\theta_{D_x f_n f_{n-1} \cdots f_1}, \theta) \right) > \log \frac{\ld_{\mathrm{crit}}^n}{m_0^{n-1}} \frac{\ep}{\pi}. $$
By assumption we know that the $\mu^{(n)}$-probability that the chosen $f_1, \ldots, f_n$ satisfy (\ref{anyway1}) and (\ref{anyway2}) is at least $\eta^n$. Moreover for $\mu^{(n)}$-almost every $f$, $\log \| D_x f(\theta)\| \geq -n \log \ld_{\mathrm{max}}$. Hence
\begin{align} \label{anyway3}
 \int \log \| D_x f(\theta) \| d \mu^{(n)}(f) \geq \eta^n \left( \log \frac{\ld_{\mathrm{crit}}^n}{m_0^{n-1}} \frac{\ep}{\pi} \right) + (1 - \eta^n) (-n \log \ld_{\mathrm{max}}). 
 \end{align}
Take $n$ large enough so that 
$$ \log \frac{\ld_{\mathrm{crit}}^n}{m_0^{n-1}} \frac{\ep}{\pi} > 0. $$
Now fix such $n$, as the right hand side of (\ref{anyway3}) increases to $\log \ds\frac{\ld_{\mathrm{crit}}^n}{m_0^{n-1}} \frac{\ep}{\pi} $ as $\eta \to 1$, there is some $\eta \in (0, 1)$ such that the right hand side of (\ref{anyway3}) is positive. 
\end{proof}

\section{Discrete random perturbation of the standard map} \label{smap}

In this section, we show an example of a random dynamical system satisfying uniform expansion. 

Let $L \in \R$ be a parameter. The standard map $\Phi_L$ of the $2$-torus $\T^2 = \R^2 / (2\pi \Z)^2$, given by
$$ \Phi_L(I, \theta) = (I + L \sin \theta, \theta + I + L \sin \theta), $$
is a well-known example of a chaotic system for which it is hard to show positivity of Lyapunov exponents (with respect to the Lebesgue measure on $\T^2$). For $L  \gg 1$, it has strong expansion and contraction on a large but non-invariant region. Nonetheless on two narrow strips near $\theta = \pm \pi/2$, vectors can be arbitrarily rotated. The area of these ``bad regions'' goes to zero as $L \to \infty$, so one expect the Lyapunov exponent to be roughly $\log L$, reflecting the expansion rate in the rest of the phase space. However, positivity of Lyapunov exponents has not been shown for any single $L$. 

In \cite{BXY}, the authors considered a kind of random perturbations of a family of maps including the standard map, and showed positivity of Lyapunov exponents for this perturbation for sufficiently large $L$.  More precisely, under a linear change of coordinates $x = \theta$, $y = \theta - I$, the standard map is conjugate to the map
\begin{equation} \label{stanmap}
 F(x, y) = (L \sin x + 2x - y, x) 
\end{equation}
on $\T^2 = \R^2 / (2\pi \Z)^2$. Note that $F$ preserves the Lebesgue measure on $\T^2$. They considered the composition of random maps
$$ F_{\underline{\og}}^n = F_{\og_n} \circ \cdots \circ F_{\og_1} \quaddd \text{ for } \quadd n = 1, 2, 3, \ldots, $$
where
$$ F_\og = F \circ S_\og, \quaddd S_\og(x, y) = (x + \og, y), $$
and the sequence $\underline{\og} = (\og_1, \og_2, \ldots) \in \O^\N$ is chosen with the probability measure $\mu^\N$, where $\mu = \mathrm{Leb}_{[-\ep, \ep]}$ is the uniform distribution on the interval $[-\ep, \ep]$ for some $\ep > 0$. 

For this Markov chain, any stationary measure is absolutely continuous with respect to Lebesgue measure. Hence they were able to use this in the subsequent estimates of the Lyapunov exponents, using the fact that the Lebesgue measure of the ``bad regions'' goes to zero as $L \to \infty$. 

In this section, we consider a discrete version of the random perturbation, where at each step, one can choose from only finitely many maps with equal probability. In this case it is not \emph{a priori} clear that every stationary measure is absolutely continuous with respect to Lebesgue. In particular it is possible that the stationary measure may have positive measure concentrated in the bad region. In fact, one of our results is a classification of the ergodic stationary measures of this perturbation. 

We shall show that this random dynamical system satisfies uniform expansion. As a corollary we show that the maps have a Lyapunov exponent $\sim \log L$. Moreover, from the previous sections, it follows that the stationary measures are either finitely supported or Lebesgue, and the orbits are either finite or dense. 

Let $r \in \N$ and $\O := \{k\ep: k = 0, \pm 1, \pm 2, \ldots, \pm r\}$. We consider the composition of random maps
$$ F_{\underline{\og}}^n = F_{\og_n} \circ \cdots \circ F_{\og_1} \quaddd \text{ for } \quadd n = 1, 2, 3, \ldots, $$
where
$$ F_\og = F \circ S_\og, \quaddd S_\og(x, y) = (x + \og, y), $$
and the sequence $\underline{\og} = (\og_1, \og_2, \ldots) \in \O^\N$ is chosen with the probability measure $\mu^\N := \left( \ds\frac{1}{|\O|}\ds\sum_{\og \in \O} \delta_{\og} \right)^\N$. Here $\delta_{k\ep}$ is the delta mass on $\Diff^2(\bb{T}^2)$ at the map $F_{k\ep}$. 

The main proposition in this section is the following. 

\begin{proposition} \label{uniexpand}
Let $\delta \in (0, 1)$. There exists an integer $r_0 = r_0(\delta) > 0$ such that if $r \geq r_0$ and $\ep \in [L^{-1+\delta}, 1/(2r+1))$, then the measure $\mu = \ds\frac{1}{|\O|}\ds\sum_{\og \in \O} \delta_{\og}$ is uniformly expanding on $\T^2$ for all large enough $L$.
\end{proposition}
\noindent Throughout this section, estimates containing $\gg, \gtrsim$ and $\sim$ are with respect to $L \to \infty$. More precisely, we write
\begin{itemize}
\item $f(L) \gg g(L)$ if $|f(L) / g(L)| \to \infty$ as $L \to \infty$.
\item $f(L) \gtrsim g(L)$ if $\ds\liminf_{L \to \infty} |f(L) / g(L)| > 0$ (possibly infinite). 
\item $f(L) \sim g(L)$ if $f(L) / g(L) \to 1$ as $L \to \infty$. 
\end{itemize}

\noindent For $A \in \R$, let $G(A) := \begin{pmatrix} A & -1 \\ 1 & 0 \end{pmatrix} \in SL_2(\R)$. Note that $DF_{(x, y)} = \begin{pmatrix} L \cos x +2 & -1 \\ 1 & 0 \end{pmatrix} = G(L \cos x + 2)$. Let $n \in \N$ to be determined. By Lemma \ref{simple}, we observe that if $A \gg 1$, then
$$ \| G(A) \| \sim A, \quaddd \theta_{G(A)} \sim \frac{\pi}{2}, \quaddd \text{ and } \quaddd \theta_{G(A)^{-1}} \sim 0. $$

The next lemma estimates the change of the contracting direction of products of $G(A_i)$ as we vary one of $A_i$ and fix the rest, assuming $A_j$ is large for all $j \neq i$. 

\begin{lemma} \label{singleangleestimate}
Let $\theta_n$ be the contracting direction of $G(A_n)G(A_{n-1}) \cdots G(A_2)G(A_1)$. If $A_i \gg 1$ for all $i = 1, 2, \ldots n$, then for each $i$ with $1 \leq i \leq n$, 
$$ \frac{d\theta_n}{dA_i} \sim \frac{1}{A_1^2 A_2^2 \cdots A_i^2}. $$
More precisely, let $\theta'_n$ be the contracting direction of $G(A_n')G(A_{n-1}') \cdots G(A_2')G(A_1')$. \\
For each $i = 1, 2, \ldots, n$, if $A_j = A_j' \gg 1$ for all $j \neq i$, and $A_i, A_i' \gg 1$, then
$$ \theta_n' - \theta_n \sim \ds\frac{1}{A_1^2A_2^2 \cdots A_{i-1}^2} \left(\frac{1}{A_i} - \frac{1}{A_i'} \right). $$

\end{lemma}

\begin{proof}
By Lemma \ref{simple}(a), we know that 
$$ \tan 2\theta_{G(A)} = -\frac{2}{A}. $$
By differentiating in $A$,
$$ \frac{d\theta_{G(A)}}{dA} = \frac{1}{A^2 + 4}. $$
By Lemma \ref{basicangleestimate} (a), for all $1 \leq i \leq n$, 
$$ \frac{d\theta_{G(A_n)G(A_{n-1})\cdots G(A_i)}}{dA_i} = \frac{d\theta_{G(A_i)}}{dA_i} = \frac{1}{A_i^2 + 4}. $$
Moreover, using Lemma \ref{simple}(b), one can show that for all $1 \leq i \leq n$, 
$$ \| G(A_i) G(A_{i-1}) \cdots G(A_1) \| \sim A_1 A_2 \cdots A_i $$
since the top left corner of $G(A_i) G(A_{i-1}) \cdots G(A_1)$ is $A_1 A_2 \cdots A_i$ and the other three entries are of strictly lower order if $A_k \gg 1$ for all $1 \leq k \leq i$. Also notice that 
$$ \theta_{G(A_n)G(A_{n-1})\cdots G(A_i)} \sim \pi/2 \quaddd \text{ and  } \quaddd \theta_{(G(A_{i-1}) G(A_{i-2}) \cdots G(A_1))^{-1}} \sim 0. $$ Apply Lemma \ref{basicangleestimate}(b) with $M_1 = G(A_n)G(A_{n-1})\cdots G(A_i)$ and $M_2 = G(A_{i-1}) G(A_{i-2}) \cdots G(A_1)$, we have
$$ \frac{d\theta_n}{d\theta_{G(A_n)G(A_{n-1})\cdots G(A_i)}} \sim \frac{1}{(A_1 A_2 \cdots A_{i-1})^2}. $$
Hence
$$ \frac{d\theta_n}{dA_i} = \frac{d\theta_n}{d\theta_{G(A_n)G(A_{n-1})\cdots G(A_i)}} \frac{d\theta_{G(A_n)G(A_{n-1})\cdots G(A_i)}}{dA_i} \sim \frac{1}{(A_1 A_2 \cdots A_{i-1})^2} \frac{1}{A_i^2 + 4} \sim \frac{1}{A_1^2 A_2^2 \cdots A_i^2}. $$

%
%
\end{proof}

The next lemma estimates the change of the contracting direction of $DF_{\underline{\og}}^n$ if we fix the first $i-1$ letters in $\underline{\og}$ and change $\og_j$ for all $j \geq i$.

\begin{lemma} \label{angleestimate} 
Let $\underline{\og}, \underline{\og}' \in \O^\N$, $\ep > L^{-1}$. Given $(x, y) \in \bb{T}^2$, for $i = 0, 1, 2, \ldots, n$, let 
\begin{itemize}
\item $(x_i, y_i) := F_{\underline{\og}}^i(x, y)$ and $(x_i', y_i') := F_{\underline{\og}'}^i(x, y)$, 
\item $A_i := L \cos x_{i-1} + 2$ and $A_i' := L \cos x_{i-1}' + 2$ for $i = 1, 2, 3, \ldots$, 
\item $\theta, \theta'$ be the contracting directions of $G(A_n)G(A_{n-1}) \cdots G(A_2)G(A_1)$ and $G(A_n')G(A_{n-1}') \cdots G(A_2')G(A_1')$.
\end{itemize}
For each $i = 1, 2, \ldots, n$, suppose $A_j = A_j' \gg 1$ for all $j < i$,  $A_j, A_j' \gg 1$ for all $j \geq i$ and $A_i - A_i' \gtrsim \ep L/2$. Then 
\begin{equation} \label{diff}
 \theta - \theta' \gtrsim \ds\frac{1}{A_1^2 A_2^2 \cdots A_{i-1}^2}\frac{\ep L/2}{A_i A_i'}. \tag{a}
\end{equation}
As a result,
\begin{equation} \label{diff2}
\| DF_{\underline{\og}}^n (\theta') \| \gtrsim \frac{A_{i+1} A_{i+2} \cdots A_n}{A_1 A_2 \cdots A_{i-1} } \frac{\ep L/2}{A_i'}. \tag{b}
\end{equation}
\end{lemma}

\begin{proof}
Without loss of generality, assume that $A_i > A_i'$. For all $j \geq i$, let $\theta_j$ be the contracting direction of 
$$ G(A_n')G(A_{n-1}') \cdots G(A_j') G(A_{j-1}) G(A_{j-2}) \cdots G(A_1). $$
Then $\theta' = \theta_i$. We also use the notation $\theta_{n+1} := \theta$, the contracting direction of $G(A_n)G(A_{n-1}) \cdots G(A_1)$. By Lemma \ref{singleangleestimate}, for all $i \leq j \leq n$, 
$$ \theta_{j} - \theta_{j+1} \sim \frac{1}{A_1^2A_2^2 \cdots A_{j-1}^2} \left(\frac{1}{A_j} - \frac{1}{A_j'} \right). $$
For all $j > i$, since $A_j, A_j' \gg 1$, $A_i > A_i'$ and $\ep > L^{-1}$, we have
$$ \theta_j - \theta_{j+1} \sim \frac{1}{A_1^2A_2^2 \cdots A_{j-1}^2} \left(\frac{1}{A_j} - \frac{1}{A_j'} \right) \ll \ds\frac{1}{A_1^2 A_2^2 \cdots A_{i-1}^2}\frac{\ep L / 2}{A_i A_i'} \lesssim \frac{1}{A_1^2A_2^2 \cdots A_{i-1}^2} \left(\frac{1}{A_i} - \frac{1}{A_i'} \right) \sim \theta_i - \theta_{i+1}. $$
Therefore $\theta_j - \theta_{j+1}$ is dominated by $\theta_i - \theta_{i+1}$ for all $i < j \leq n$. Hence
$$ \theta' - \theta = \theta_i - \theta_{n+1} = (\theta_i - \theta_{i+1}) + (\theta_{i+1} - \theta_{i+2}) + \cdots (\theta_n - \theta_{n+1}) \sim \frac{1}{A_1^2A_2^2 \cdots A_{i-1}^2} \left(\frac{1}{A_i} - \frac{1}{A_i'} \right). $$
The second statement follows from the first by Lemma \ref{distance} since $\| G(A_n) G(A_{n-1}) \cdots G(A_1) \| \sim A_n A_{n-1} \cdots A_1$ by Lemma \ref{simple}(b). 
\end{proof}

\begin{proof}[Proof of Proposition \ref{uniexpand}]
We are now ready to prove the main proposition of the section. The idea is as follows: for each point $(x, y) \in \bb{T}^2$, since the elements in $\O$ are of distance at least $\ep \geq L^{-1+\delta}$ apart, for each $k\ep \in \O$, for all $k'\ep \in \O \setminus \{k\ep\}$, all except possibly one of them satisfy (let $A(x) := L \cos x + 2$ for $x \in \R / 2\pi \Z$)
\begin{equation} \label{good}
 |A(x + k'\ep) - A(x+k\ep)| \gtrsim \ep L  / 2 \quaddd \text{ and } \quaddd |A(x+k'\ep)| \gtrsim L^\delta. 
\end{equation}
Geometrically, this means that firstly, all except one of them has norm growing to infinity with $L$, and the contracting directions of the corresponding differential maps 
$$ DF_{(x+\og, y)} = \begin{pmatrix} L \cos (x+\og) + 2 & -1 \\ 1 & 0 \end{pmatrix} $$
are all pointing in roughly the vertical direction. Moreover, each of the contracting direction is separated from all others (except one) by a significiant amount ($\sim \ep / \| F_{(x+\og, y)} \|$). Hence after $n$ steps, for many of the words $\uog \in \O^n$, the contracting directions are close to the vertical direction and yet well separated (Figure \ref{tree}). Thus each $\theta \in \P^1$ has distance from all but one of these contracting direction bounded from below. From Lemma \ref{angleestimate}, we know that the distance between the contracting directions of two words are dominated by their distance at the first letter they differ, and yet the norm grows by at least $L^\delta$ after every step. Using Lemma \ref{distance}, as long as the word does not enter a bad region (where the contracting direction is rotated drastically), the log expansion $\log \| DF_{\uog}^n \|$ will eventually be large. Since most words do not enter a bad region, and those that do enter a bad region admit a trivial lower bound $\log \|DF_{\uog}^n \| \geq -n \log L$, eventually we will obtain positive expansion on average. 

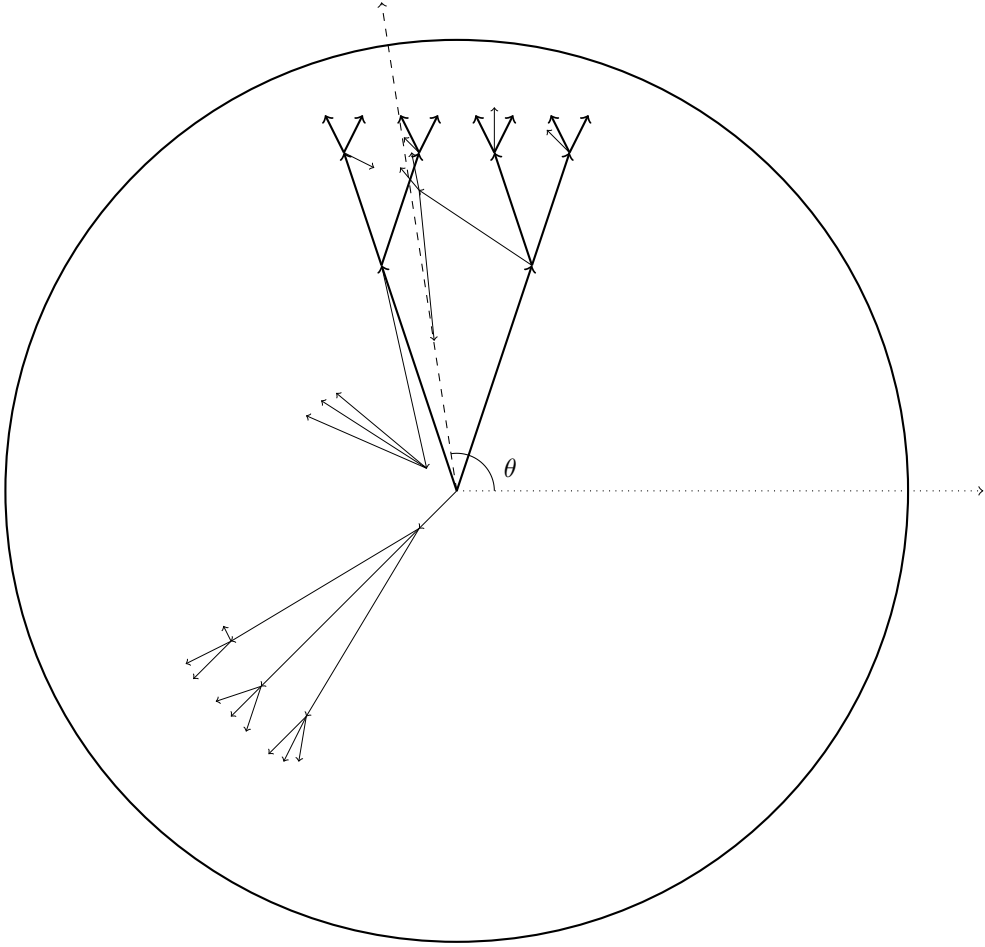
\begin{figure}[ht]
\centering
\begin{tikzpicture}
\draw [thick] (0,0) circle [radius=6];
\draw [dotted, ->] (0, 0) -- (7, 0);
\draw [thick, ->] (0, 0) -- (1, 3);
\draw [thick, ->] (0, 0) -- (-1, 3);
\draw [thick, ->] (1, 3) -- (0.5, 4.5);
\draw [thick, ->] (1, 3) -- (1.5, 4.5);
\draw [thick, ->] (-1, 3) -- (-0.5, 4.5);
\draw [thick, ->] (-1, 3) -- (-1.5, 4.5);
\draw [thick, ->] (0.5, 4.5) -- (0.75, 5);
\draw [thick, ->] (0.5, 4.5) -- (0.25, 5);
\draw [thick, ->] (1.5, 4.5) -- (1.75, 5);
\draw [thick, ->] (1.5, 4.5) -- (1.25, 5);
\draw [thick, ->] (-0.5, 4.5) -- (-0.75, 5);
\draw [thick, ->] (-0.5, 4.5) -- (-0.25, 5);
\draw [thick, ->] (-1.5, 4.5) -- (-1.75, 5);
\draw [thick, ->] (-1.5, 4.5) -- (-1.25, 5);
\draw [line width = 0.05mm, ->] (0, 0) -- (-0.5, -0.5);
\draw [line width = 0.05mm, ->] (1, 3) -- (-0.5, 4);
\draw [line width = 0.05mm, ->] (-1, 3) -- (-0.4, 0.3);
\draw [line width = 0.05mm, ->] (0.5, 4.5) -- (0.5, 5.1);
\draw [line width = 0.05mm, ->] (1.5, 4.5) -- (1.2, 4.8);
\draw [line width = 0.05mm, ->] (-0.5, 4.5) -- (-0.7, 4.7);
\draw [line width = 0.05mm, ->] (-1.5, 4.5) -- (-1.1, 4.3);
\draw [line width = 0.05mm, ->] (-0.5, -0.5) -- (-2, -3);
\draw [line width = 0.05mm, ->] (-0.5, -0.5) -- (-3, -2);
\draw [line width = 0.05mm, ->] (-0.5, -0.5) -- (-2.6, -2.6);
\draw [line width = 0.05mm, ->] (-2, -3) -- (-2.5, -3.5);
\draw [line width = 0.05mm, ->] (-2, -3) -- (-2.3, -3.6);
\draw [line width = 0.05mm, ->] (-2, -3) -- (-2.1, -3.6);
\draw [line width = 0.05mm, ->] (-3, -2) -- (-3.5, -2.5);
\draw [line width = 0.05mm, ->] (-3, -2) -- (-3.6, -2.3);
\draw [line width = 0.05mm, ->] (-3, -2) -- (-3.1, -1.8);
\draw [line width = 0.05mm, ->] (-2.6, -2.6) -- (-3, -3);
\draw [line width = 0.05mm, ->] (-2.6, -2.6) -- (-3.2, -2.8);
\draw [line width = 0.05mm, ->] (-2.6, -2.6) -- (-2.8, -3.2);
\draw [line width = 0.05mm, ->] (-0.5, 4) -- (-0.3, 2);
\draw [line width = 0.05mm, ->] (-0.5, 4) -- (-0.6, 4.5);
\draw [line width = 0.05mm, ->] (-0.5, 4) -- (-0.75, 4.3);
\draw [line width = 0.05mm, ->] (-0.4, 0.3) -- (-2, 1);
\draw [line width = 0.05mm, ->] (-0.4, 0.3) -- (-1.8, 1.2);
\draw [line width = 0.05mm, ->] (-0.4, 0.3) -- (-1.6, 1.3);
\draw [dashed, ->] (0, 0) -- (-1, 6.5);
\draw [domain=0:98.75] plot ({0.5*cos(\x)}, {0.5*sin(\x)});
\node [right] at (0.5, 0.3) {$\theta$};
%
\end{tikzpicture}
\caption{The random walk after $3$ steps. The bold directions form a well-separated "tree".}
\label{tree}
\end{figure}

We now make the above discussion precise using the previous lemmas. For $x \in \R / 2\pi \Z$, let $A(x) = L \cos x + 2$. Recall that at each point $(x, y) \in \bb{T}^2 = \R^2 / (2 \pi \Z)^2$, the differential map of $F(x, y) = (L \sin x + 2x - y, x)$ is
$$ DF = \begin{pmatrix} L \cos x + 2 & -1 \\ 1 & 0 \end{pmatrix} = G(A(x)). $$
Let $\ep \in [L^{-1+\delta}, 1/(2r+1))$. For each $\og \in \O = \{k\ep: k = 0, \pm 1, \pm 2, \ldots, \pm r\}$, 
$$ F_\og = F \circ S_\og = (L \sin (x+\og) + 2(x+\og) - y, x+\og). $$
Hence the differential $DF_\og$ is
$$ \begin{pmatrix} L \cos (x+\og) + 2 & -1 \\ 1 & 0 \end{pmatrix} = G(A(x+\og)). $$
Fix a point $(x, y) \in \bb{T}^2$. For $\uog, \uog' \in \O^n$ and $0 \leq i \leq n$, let
$$ (x_i, y_i) := F_{\uog}^i(x, y) \quaddd \text{ and } \quaddd (x_i', y_i') := F_{\uog'}^i(x, y).$$
Let 
$$ A_i := L \cos x_{i-1} + 2, \quaddd \text{ and } \quaddd A_i' := L \cos x_{i-1}' + 2. $$
We say that a word $\uog \in \O^n$ is \emph{long} (with respect to $(x, y) \in \bb{T}^2$) if 
$$ |A_i| \gtrsim L^\delta \quaddd \text{ for all } 1 \leq i \leq n. $$
For each word $\uog \in \O^n$, let 
$$ \theta_{\uog} := \theta_{DF_{\uog}^n} $$
be the contracting direction of the matrix $DF_{\uog}^n$. 

Observe by (\ref{good}) that for each long $\uog \in \O^n$, there are at least $(|\O| - 2)(|\O| - 1)^{n-1}$ long words $\uog' \in \O^n$ such that
$$ |A_1 - A_1' | \gtrsim \frac{\ep L}{2}. $$
By Lemma \ref{angleestimate}(a), since $A_1 \leq L+2$, 
$$ |\theta_{\uog} - \theta_{\uog'} | \gtrsim \frac{\ep L / 2}{A'_1 A_1} \gtrsim \frac{\ep/2}{A'_1}. $$
Similarly, for all $1 \leq i \leq n$, there are at least $(|\O| - 2)(|\O| - 1)^{n-i}$ long words $\uog' \in \O^n$ such that
$$ \og_j = \og_j' \quad \text{ for all } j < i, \quaddd \text{ and } \quaddd |A_i - A_i'| \gtrsim \frac{\ep L}{2}. $$
Thus again by Lemma \ref{angleestimate}(a), 
$$ |\theta_{\uog} - \theta_{\uog'}| \gtrsim \frac{1}{A_1^2 A_2^2 \cdots A_{i-1}^2} \frac{\ep L/2}{A_i A_i'} \gtrsim \frac{1}{A_1^2 A_2^2 \cdots A_{i-1}^2} \frac{\ep /2}{A_i'}. $$
For all $\theta \in \P^1$, take a long word $\uog \in \O^n$ that minimizes $|\theta_{\uog} - \theta|$ (among long words). Then from above, we know that for each $1 \leq i \leq n$, there are at least $(|\O| - 2)(|\O| - 1)^{n-i}$ long words $\uog' \in \O^n$ such that
$$ |\theta_{\uog'} - \theta| \gtrsim \frac{1}{2} \frac{1}{A_1^2 A_2^2 \cdots A_{i-1}^2} \frac{\ep /2}{A_i'} = \frac{1}{A_1^2 A_2^2 \cdots A_{i-1}^2} \frac{\ep /4}{A_i'}. $$
Hence by Lemma \ref{distance} (note that $A_j = A_j'$ for $j < i$ since $\og_j = \og_j'$), 
\begin{align*}
 \| DF_{\uog'}^n(\theta) \| &\gtrsim  \frac{1}{A_1^2 A_2^2 \cdots A_{i-1}^2} \frac{\ep /4}{A_i'} \cdot \| DF_{\uog'}^n \| 
\gtrsim  \frac{1}{A_1^2 A_2^2 \cdots A_{i-1}^2} \frac{\ep /4}{A_i'} (A_1' A_2' \cdots A_n') \\
&\gtrsim \frac{A_{i+1}' A_{i+2}' \cdots A_n'}{A_1 \cdots A_{i-1}} \frac{\ep}{4} \gtrsim \frac{(L^\delta)^{n-i}}{L^{i-1}} \frac{L^{-1+\delta}}{4} \\
&\gtrsim L^{\delta(n-i+1) - i}.
\end{align*}



Thus for each direction $\theta \in \P^1$, for each $i = 1, 2, \ldots, n$, we have at least $(|\O| - 2)(|\O| - 1)^{n-i}$ words $\underline{\og}$ in $\O^n$ such that 
$$ \log \| DF_{\underline{\og}}^n (\theta) \| \gtrsim (\delta(n-i+1) - i) \log L. $$
For the remaining $|\O|^n - (|\O| - 1)^n + 1$ words $\underline{\og}$, we have 
$$ \log \| DF_{\underline{\og}}^n (\theta)  \| \geq -n \log L. $$

Hence 
$$ \int \log \| DF_{\underline{\og}}^n (\theta)  \| d \mu^{(n)}(\underline{\og}) \gtrsim \frac{\Xi(|\O|, n, \delta)}{|\O|^n} \log L, $$
where $\Xi(|\O|, n, \delta) = \ds\sum_{i=1}^n (|\O|-2)(|\O| - 1)^{n-i} (\delta(n-i+1) - i) - (|\O|^n - (|\O|-1)^n + 1)n. $

The coefficient of $|\O|^n$ in $\Xi(|\O|, n, \delta)$ is $n \delta - 1$, hence is positive if $n > 1/\delta$. The coefficient of $|\O|^{n-1}$ is $-(\delta + 1)(n^2 + 1) + n$. If $n > 1/\delta$, for large enough $r$ (hence large enough $|\O| = 2r+1$), we have $\Xi(|\O|, n, \delta) > 1$. Hence $\mu$ is uniformly expanding for all large enough $r$ (depending only on $\delta$) and large enough $L$, with $N := \lceil 1/\delta \rceil$ and $C = |\O|^{-N} \log L$. Moreover, for $\delta \in (1/3, 1)$, we can take $n = 3$, and $|\O| \geq \ds\frac{10\delta + 7}{3\delta - 1}$. 

\end{proof}

\section{Computer-assisted verification of uniform expansion} \label{computer}

In this section we outline an algorithm to verify uniform expansion numerically, when $\mu$ is finitely supported on $\Diff^2(M)$. Uniform expansion is \emph{a priori} an infinite condition in the sense that there are infinitely many points on the manifold and infinitely many directions on each fiber of the unit tangent bundle. Nonetheless since the maps in the support of $\mu$ are $C^2$ and the left hand side of the uniform expansion condition is Lipschitz in $v$, using the fact that the unit tangent bundle $T^1 M$ is compact, one can take a finite grid on $T^1 M$, verify the uniform expansion at each grid point, and then prove uniform expansion on the whole $T^1 M$ by the Lipschitz condition. 

This algorithm checks a sufficient condition of uniform expansion when $N = 1$. Nonetheless, by replacing $\mu^{(N)}$ with $\mu$, one may in principle apply the same algorithm to verify uniform expansion for any $N$.

Let $f_1, \ldots f_d$ be the maps in the support of $\mu$ and $\mu = c_1 \delta_{f_1} + \cdots + c_d \delta_{f_d}$ for $c_i \in (0, 1]$. For each $i = 1, 2, \ldots, d$, $P \in M$ and $\theta \in \P^1$, we consider the function
$$ F_i(P, \theta) := \log \| D_P f_i (\theta) \|. $$
Our goal is to verify that
\begin{align} \tag{UE} \label{uuee}
F(P, \theta) := \sum_{i=1}^d c_i F_i(P, \theta) > C 
\end{align}
for some $C > 0$.

We now outline the algorithm.

\begin{enumerate}[label={\bf Step \arabic*:}]
\item Choose local coordinates $t_1, t_2$ on $M$, and find $C_M, C_\theta > 0$ such that 
$$ \left| \frac{\pl F_i}{\pl t} \right| < C_M, \quaddd \left| \frac{\pl F_i}{\pl \theta} \right| < C_\theta $$
for $t = t_1, t_2$. Such constants exist since $F_i$ is $C^1$ and $M$ is compact.
\item Fix some $C > 0$. 
\item Pick $r, \rho > 0$ such that $rC_M < C/4$ and $\rho C_\theta < C/4$. 
\item Take a finite grid $\ml{G}$ on the unit tangent bundle $T^1 M$ that is $r$-dense on the manifold and $\rho$-dense on the unit tangent space $T_P^1 M$ for each grid point $P \in M$. 
\item Verify (\ref{uuee}) for each grid point $(P, \theta) \in \ml{G}$. 
\item From the derivative bounds in {\bf Step 1} and the choices of $r$ and $\rho$ in {\bf Step 3}, one can conclude that (\ref{uuee}) holds with $C$ replaced by $C/4$.
\end{enumerate}

\section{Outer automorphism group action on character variety} \label{cvariety}

\subsection{Introduction}

In this section, we consider an example of a random dynamical system where the uniform expansion property can be checked numerically using the algorithm outlined in Section \ref{computer}.

Let $F_n$ be a free group of rank $n > 1$, $G$ be a compact Lie group. The natural volume form on $\Hom(F_n, G)$ is invariant under $\Aut(F_n)$. This form descends to a natural finite measure $\ld$ on the character variety $\Hom(F_n, G)//G$ that is invariant under $\Out(F_n)$. We refer the reader to \cite{G} for more details about ergodic properties of this system, and the celebrated work of Goldman \cite{G2} for a detailed account in the case when $F_n$ is replaced by the mapping class group of a surface.

Goldman \cite{G} proved that in the case when $G = \mathrm{SU(2)}$ and $n > 2$, the $\Out(F_n)$-action on $\Hom(F_n, G)//G$ is ergodic. On the other hand, the action is not ergodic when $n = 2$, since it preserves the surjective function
\begin{align*}
\kappa: \Hom(F_n, G) // G &\to [-2, 2] \\
[\rho] &\mapsto \tr(\rho([X, Y])) 
\end{align*}
where $X, Y$ is a pair of free generators of $F_2$, and $[X, Y] := XYX^{-1}Y^{-1}$ is the commutator of $X$ and $Y$. The ergodic components are the disintegration $\ld_s$ of $\ld$ on the fibers $\mk{X}_s := \kappa^{-1}(s)$ of $\kappa$ for $s \in [-2, 2]$. 

In the case when $n = 2$, the topological dynamics of this action was studied by Previte and Xia \cite{PX}, who proved, in particular, that on each shell $\mk{X}_s$, the $\Out(F_2)$-invariant sets are either finite or dense. In fact, they classified all the finite $\Out(F_2)$-invariant sets, and gave a condition for when the invariant set is dense. On the other hand, Brown \cite{B} showed using standard KAM techniques that for any nontrivial cyclic subgroup $\G \subset \Out(F_2)$ and $s$ close enough to $-2$, there is a $\G$-invariant set with positive measure on $\mk{X}_s$ that is not dense. We refer our readers to \cite{G2} and \cite{PX2} for analogous analysis of the measurable and topological dynamics of the mapping class group $\Out(\pi_1(M))$-action on the character variety $\Hom(\pi_1(M), \mathrm{SU(2)})/\mathrm{SU(2)}$. 

The analysis in \cite{PX} relies crucially on the fact that $\Out(F_2)$ is generated by Dehn twists. In fact with minor modification their method also applies to the action of a subsemigroup $\G \subset \Out(F_2)$ generated by at least two powers of distinct Dehn twists. In this section, we consider a set of generators $\ml{S}$ of a semigroup $\G \subset \Out(F_2)$ that does not contain any Dehn twists or powers of Dehn twists, and attempt to show that the $\G$-invariant sets are finite or dense by showing uniform expansion on $\ml{S}$ and applying Theorem D. The uniform expansion property is checked using a computer program. For $s$ close to $2$, the expansion is large enough that uniform expansion is observed after $1$ iteration. However, for $s$ close to $-2$, the expansion cannot be checked numerically due to the limitation of computational power. We will verify uniform expansion for a specific $s$ as a proof of concept, though the same algorithm carries for other $s$ close to $2$ as well.

More precisely, consider the following two elements of $\Out(F_2)$:
$$ \tau_X: X \mapsto X, \quad Y \mapsto XY, \quaddd \tau_Y: X \mapsto YX, \quad Y \mapsto Y. $$
Note that $\tau_X$ and $\tau_Y$ generate a subgroup $\langle \tau_X, \tau_Y \rangle$ that has index $2$ in $\Out(F_2)$. Let $\tau_{ABC} := \tau_A \circ \tau_B \circ \tau_C$ where $A, B, C \in \{X, Y\}$. Define the subsemigroup 
$$ \G = \langle f_i: i = 1, 2, \ldots, 16 \rangle \subset \Out(F_2), $$
where
\begin{itemize}
\item $f_1 = \tau_{XXXXY}$
\item $f_2 = \tau_{XXXYY}$,
\item $f_3 = \tau_{XXYYY}$,
\item $f_4 = \tau_{XYYYY}$,
\item $f_5 = \tau_{YXXXX}$,
\item $f_6 = \tau_{YYXXX}$,
\item $f_7 = \tau_{YYYXX}$,
\item $f_8 = \tau_{YYYYX}$,
\end{itemize}
and $f_i = f_{17-i}^{-1}$ for $i = 9, 10, \ldots, 16$. Now define
the measure $\mu := \ds\frac{1}{16}\left( \ds\sum_{i=1}^{16} \delta_{f_i} \right)$ on $\Out(F_2)$. 

The result of this section is the following.

\begin{proposition} \label{ue}
For $s = 1.99$, the measure $\mu$ is uniformly expanding as an action on the surface $\mk{X}_s$. 
\end{proposition}

\begin{corollary}
For $s = 1.99$, the $\G$-invariants sets on $\mk{X}_s$ are either finite or dense. 
\end{corollary}

\subsection{Character variety as a subvariety of $\R^3$} \label{tracecoordinates}

We now describe the character variety $\Hom(F_2, \mathrm{SU(2)})//\mathrm{SU(2)}$ in more explicit terms. The character variety $\Hom(F_2, \mathrm{SU(2)})//\mathrm{SU(2)}$ injects into $\R^3$ under the trace coordinates
\begin{align*}
\Hom(F_2, \mathrm{SU(2)}) // \mathrm{SU(2)} &\to \R^3 \\
[\rho] &\mapsto \begin{pmatrix} \tr(\rho(X)) \\ \tr(\rho(Y)) \\ \tr(\rho(XY)) \end{pmatrix}.
\end{align*}
This is injective, with image
$$ \mk{X} := \{(x, y, z) \in \R^3: -2 \leq x^2 + y^2 + z^2 - xyz - 2 \leq 2\}. $$
Hence we may identify $\Hom(F_2, \mathrm{SU(2)})//\mathrm{SU(2)}$ with $\mk{X}$. In these coordinates, the map $\kappa: \Hom(F_2, \mathrm{SU(2)}) // \mathrm{SU(2)} \to [-2, 2]$ described in the introduction is then
$$ \kappa(x, y, z) = x^2 + y^2 + z^2 - xyz - 2. $$
For $s \in [-2, 2]$, the ergodic components are
$$ \mk{X}_s := \kappa^{-1}(s) = \{(x, y, z) \in \R^3: x^2 + y^2 + z^2 - xyz - 2 = s\}. $$
In trace coordinates, the maps $\tau_X$ and $\tau_Y$ are
$$ \tau_X: \begin{pmatrix} x \\ y \\ z \end{pmatrix} \mapsto \begin{pmatrix} x \\ z \\ xz - y \end{pmatrix}, \quaddd \tau_Y: \begin{pmatrix} x \\ y \\ z \end{pmatrix} \mapsto \begin{pmatrix} z \\ y \\ yz - x \end{pmatrix}. $$
At each point $P = (x, y, z)$, a normal vector is given by ${\bf n}(P) = (2x - yz, 2y - zx, 2z - xy)$, with the unit normal ${\bf v_3}(P) = \ds\frac{{\bf n}(P)}{\| {\bf n}(P) \|}$. 

From \cite[Sect. 5.3]{G}, a cosymplectic structure on $\mk{X}_t$ can be given explicitly by (up to a multiplicative constant)
$$ (2x - yz) \frac{\pl}{\pl y} \wedge \frac{\pl}{\pl z} + (2y - zx) \frac{\pl}{\pl z} \wedge \frac{\pl}{\pl x} + (2z - xy) \frac{\pl}{\pl x} \wedge \frac{\pl}{\pl y}. $$

Since $\G$ preserves the symplectic structre, if we take the metric $\| \cdot \|_P := \| {\bf n}(P) \|^{-1/2} \| \cdot \|$ on $T_P \mk{X}_s$, where $\| \cdot \|$ is the restriction of the Euclidean metric from $\R^3$ to the tangent space $T_P \mk{X}_s$, then for each $f \in \Out(F_2)$, we have the area-preserving linear map 
$$ D_P f: T_P \mk{X}_s \to T_{f(P)} \mk{X}_s. $$
Note that each element $f \in \Out(F_2)$ is the restriction of a map $f_0: \R^3 \to \R^3$ to $\mk{X}_s$ in terms of the trace coordinates. Therefore $D_P f$ can be expressed as the restriction of a volume-preserving linear map $D_P f_0: \R^3 \to \R^3$, i.e. an element of $SL_3(\R)$, to $T_P \mk{X}_s$. For instance, writing $P = (x, y, z)$, 
$$ D_P \tau_X = \begin{pmatrix} 1 & 0 & 0 \\ 0 & 0 & 1 \\ z & -1 & x \end{pmatrix}, \quaddd D_P \tau_Y = \begin{pmatrix} 0 & 0 & 1 \\ 0 & 1 & 0 \\ -1 & z & y \end{pmatrix}, $$
both restricted to the tangent space $T_P \mk{X}_s$. 

\subsection{Choice of metric}

We will choose a convenient metric to work with. To do so, it suffices to give an orthonormal basis at each point. For each $P = (x, y, z) \in \mk{X}_s$, let ${\bf n}(P) = (n_1(P), n_2(P), n_3(P)) := (2x - yz, 2y - zx, 2z - xy)$ be the normal vector. Consider the following three tangent vectors in $T_P(\mk{X}_s)$
$$ {\bf v}_1(P) = \begin{pmatrix} 0 \\ n_3(P) \\ -n_2(P) \end{pmatrix}, \quaddd {\bf v}_2(P) = \begin{pmatrix} -n_3(P) \\ 0 \\ n_1(P) \end{pmatrix}, \quaddd {\bf v}_3(P) = \begin{pmatrix} n_2(P) \\ -n_1(P) \\ 0 \end{pmatrix}. $$
Clearly these are tangent vectors at $P$. Moreover since the normal vector ${\bf n}(P) = (n_1(P), n_2(P), n_3(P))$ is nonzero, at least one of $n_i(P)$, $i = 1, 2, 3$ is nonzero, thus at least two of ${\bf v}_1(P), {\bf v}_2(P), {\bf v}_3(P)$ are linearly independent. In fact, for $s < 2$, there is a positive lower bound $c = c(s)$ such that $\max_{i = 1, 2, 3} |n_i(P)| \geq c(s)$ for all $P \in \mk{X}_s$, so at least two of ${\bf v}_1(P), {\bf v}_2(P), {\bf v}_3(P)$ have Euclidean norm larger than $c(s)$. 

Now at each $P \in \mk{X}_s$, we define a positive definite inner product $\la \cdot, \cdot \ra_P$ on $T_P\mk{X}_s$ such that 
$$ \left\{ \frac{{\bf v}_i(P)}{\sqrt{n_k(P)}}, \frac{{\bf v}_j(P)}{\sqrt{n_k(P)}} \right\} $$ 
form an orthonormal basis, where $k \in \{1, 2, 3\}$ is  the index that maximizes $|n_k(P)|$, and $\{i, j, k\}$ form an even permutation of $\{1, 2, 3\}$ (we will comment on the normalizing factor $\sqrt{n_k(P)}$ in the next section). The map $P \mapsto \la \cdot, \cdot \ra_P$ is smooth except along the curves on $\mk{X}_s$ where at least two of $x, y, z$ are equal. Therefore strictly speaking they do not form a smooth metric. Nonetheless from the end of the previous paragraph, we know that there exists a constant $c'(s) > 0$ such that 
$$ c'(s)^{-1} \la \cdot, \cdot \ra \leq \la \cdot, \cdot \ra_P \leq c'(s) \la \cdot, \cdot \ra, $$
where $\la \cdot, \cdot \ra$ is the Euclidean inner product induced from $\R^3$. It is evident from the definition of uniform expansion that it is invariant under change of equivalent metrics, so it suffices to verify uniform expansion with respect to $\{\la \cdot, \cdot \ra_P\}_{P \in \mk{X}_s}$. 

The advantage of considering this metric is that, with respect to this metric and the specific orthonormal basis chosen above, $D_P \tau_X$ and $D_P \tau_Y$ (and hence the compositions) are $2 \times 2$ matrices such that up to the factor $n_k(P)$, the entries are polynomials in $x, y, z$. For instance, 
\begin{align*}
D_P \tau_X {\bf v}_1(P) &= {\bf v}_1(\tau_X(P)), \\
D_P \tau_X {\bf v}_2(P) &= \frac{n_1(P)}{n_3(\tau_X(P))} {\bf v}_1(\tau_X(P)) + \frac{n_3(P)}{n_3(\tau_X(P))} {\bf v}_2(\tau_X(P)), \\
D_P \tau_Y {\bf v}_1(P) &= \frac{n_3(P)}{n_3(\tau_Y(P))} {\bf v}_1(\tau_Y(P)) + \frac{n_2(P)}{n_3(\tau_Y(P))} {\bf v}_2(\tau_Y(P)), \\
 D_P \tau_Y {\bf v}_2(P) &= {\bf v}_2(\tau_Y(P)). 
\end{align*}
The matrices with respect to other bases can be found using the identity
$$ n_1(P) {\bf v}_1(P) + n_2(P) {\bf v}_2(P) + n_3(P) {\bf v}_3(P) = 0. $$

\subsection{Derivative bounds}

To choose the bounds $C_M$ and $C_\theta$ in the algorithm, it is necessary to compute bounds on $|\pl F_i / \pl t|$ and $|\pl F_i / \pl \theta|$ for $F_i(P, \theta) = \log \| D_P f_i(\theta)\|$ and local coordinates $t = t_1, t_2$ near $P$. If we treat $f_i$ as a function $\R^3 \to \R^3$, we can compute $D_P f_i$ as an element $L_i$ of $SL_3(\R)$. 

With respect to the metric and the corresponding orthonormal basis chosen above, $D_P f_i$ can be written as a $2 \times 2$ matrix with entries being the square root of rational functions of $x, y, z$, say $D_P f_i = \begin{pmatrix} a_{i, P} & b_{i, P} \\ c_{i, P} & d_{i, P} \end{pmatrix}$. For instance, if the orthonormal basis for $P$ is $\ds\frac{\{{\bf v}_1(P), {\bf v}_2(P)\}}{\sqrt{n_3(P)}}$ and that of $f_i(P)$ is $\ds\frac{\{{\bf v}_1(f_i(P)), {\bf v}_2(f_i(P))\}}{\sqrt{n_3(f_i(P))}}$, we can write explicitly that
$$ D_P f_i = \frac{1}{\sqrt{n_3(P) n_3(f_i(P))}} \begin{pmatrix} (L_i{\bf v}_1)_2 & (L_i{\bf v}_2)_2 \\ -(L_i{\bf v}_1)_1 & -(L_i{\bf v}_2)_1 \end{pmatrix}. $$
In particular, $\sqrt{n_3(P) n_3(f_i(P))} D_P f_i$ has polynomial entries and 
$$ \det D_P f_i = 1 $$
(the primary reason to have the normalizing factor $\sqrt{n_k(P)}$ is to ensure this matrix has determinant $1$.) Similar expressions can be obtained for the other points where the other two orthonormal bases are chosen. Hence if we choose $x$ and $y$ to be the local coordinates near $P$ (corresponding to the ${\bf v}_1$ and ${\bf v}_2$ directions), the derivatives with respect to $x$ and $y$ can be explicitly computed and bounded. 

More explicitly, for $M = \begin{pmatrix} a & b \\ c & d \end{pmatrix} \in SL_2(\R)$, let $F_M(\theta) = \log \| M(\theta) \|$. Then
$$ F_M(\theta) = \frac{1}{2} \log \left( \frac{1}{2}(a^2 + b^2 + c^2 + d^2) + \frac{1}{2}(a^2 - b^2 + c^2 - d^2) \cos 2\theta + (ab + cd)\sin 2\theta \right). $$
Thus $\pl F_M(\theta) / \pl t$ can be represented explicitly in terms of $a, b, c, d, a', b', c', d'$ and $\theta$, where $a' = \pl a / \pl t$ etc. Since for all $P = (x, y, z) \in \ml{X}_s$, the coordinates $x, y, z$ are in $[-2, 2]$, while $a, b, c, d$ are polynomials in $x, y, z$ divided by $\sqrt{n_3(P) n_3(f_i(P))}$, all these can be explicitly bounded. Furthermore by the choice of the orthonormal bases at $P$ and $f_i(P)$ we know that $|n_3(P)| > |n_1(P)|, |n_2(P)|$ and similarly for $|n_3(f_i(P))|$, we have that $\sqrt{n_3(P) n_3(f_i(P))}$ is bounded below by an explicit positive number depending only on $s$. We shall omit the explicit expressions here as they are written in the program (see Program 1). 

\subsection{Choice of Parameters in the verification}

In this section we choose the parameters in the algorithm to check that $\mu$ is uniformly expanding.

\begin{proof}[Proof of Proposition \ref{ue}]
We verify uniform expansion using the algorithm from the previous section. Let $f_i$ be the maps in the support of $\mu$ with $i = 1, 2, \ldots, d$, where $d = 16$. We choose the grid $\ml{G}$ in the following process: recall that
$$ \mk{X}_s = \{(x, y, z) \in \R^3: x^2 + y^2 + z^2 - xyz - 2 = s\}. $$
Let ${\bf n}(P) = (n_1(P), n_2(P), n_3(P)) = (2x - yz, 2y - zx, 2z - xy)$. Within the region $\{P \in M \mid |n_3(P)| = \max_{k=1,2,3} |n_k(P)|\}$, we use the $x$ and $y$ directions as local coordinates. This corresponds to using ${\bf v}_1$ and ${\bf v}_2$ as an orthonormal coordinate system. Similarly for the other two regions where $|n_1(P)|$ and $|n_2(P)|$ dominate. 
We verify uniform expansion for $s = 1.99$.
\begin{enumerate}[label={\bf Step \arabic*:}]
\item We take $C_M = 600$ and $C_\theta = 600$. \\ (these are computed using the explicit expressions of $\pl F_M(\theta) / \pl t$ on a grid ({\bf Program 1}) and then a na\"{i}ve bound on second derivatives of $F_M(\theta)$. ). 
\item Fix $C = 0.25$. 
\item Let $r = 0.0001 < C/(4C_M)$ and $\rho = 0.0001 < C / (4C_\theta)$.
\item Take an $r$-dense grid on $\mk{X}_s$ using the specified local coordinates. We fix a $\rho$-grid in the unit tangent space direction.
\item We verify with {\bf Program 2} that (\ref{uuee}) holds on the grid with $C = 0.25$ as in {\bf Step 2}.
\item From the derivative bounds in {\bf Step 1} and the choices of $r$ and $\rho$ in {\bf Step 3}, one can conclude that (\ref{uuee}) holds on the whole surface with $C$ replaced by $C/4$.
\end{enumerate}
The programs were run on the University of Chicago Midway compute cluster partition broadwl. \\
 Specification: 28 cores of Intel E5-2680v4 2.4 GHz. Memory: 64 GB. Runtime: 47714 seconds. \\

\noindent {\bf Program 1} ($C^2$ bounds in {\bf Step 1}):
\begin{itemize}
\item Code: \url{http://math.uchicago.edu/~briancpn/derivative_single.cpp}
\item Output: \url{http://math.uchicago.edu/~briancpn/secondderivative.txt}
\end{itemize}
\noindent {\bf Program 2} ($C^1$ bounds and (\ref{uuee}) in {\bf Step 5}):
\begin{itemize}
\item Code: \url{http://math.uchicago.edu/~briancpn/actual.cpp}
\item Output: \url{http://math.uchicago.edu/~briancpn/character_variety_test.txt}
\end{itemize}


%
%
\end{proof}

\end{document}